\documentclass[reqno,final]{amsart}

\usepackage{subfigure,amsmath,amsfonts,amssymb,graphicx,latexsym,hyperref,cite}
\usepackage[english]{babel}
\usepackage{showkeys}
\usepackage{cite}
\usepackage{todonotes}

\newtheorem{theorem}{Theorem}
\newtheorem{corollary}[theorem]{Corollary}
\newtheorem{lemma}[theorem]{Lemma}
\newtheorem{proposition}[theorem]{Proposition}

\newtheorem{definition}[theorem]{Definition}
\theoremstyle{remark}
\newtheorem{remark}[theorem]{Remark}

\theoremstyle{definition}

\usepackage{color}

\newcommand{\C}{\mathbb{C}}

\numberwithin{theorem}{section}

\begin{document}
	\title[Kinetic Models for Semiflexible Polymers]{Kinetic Models for Semiflexible Polymers in a Half-plane}
	\author{Jin Woo Jang}
	\address{Department of Mathematics, Pohang University of Science and Technology (POSTECH), Pohang, South Korea (37673)}
	\email{jangjw@postech.ac.kr}
	\author{Juan J. L. Vel\'azquez}
	\address{Institute for Applied Mathematics, University of Bonn,
		Endenicher Allee 60, D-53115 Bonn, Germany}
	\email{velazquez@iam.uni-bonn.de}


	\let\thefootnote\relax
	\newcommand{\eqdef }{\overset{\mbox{\tiny{def}}}{=}}
	\newcommand{\rth}{{\mathbb{R}^3}}
	\newcommand{\supt}{\varphi_{sup}}
	\newcommand{\varsup}{\bar{\varphi}_{sup}}
	\newcommand{\varinf}{\bar{\varphi}_{\infty}}
	\thispagestyle{empty}

\begin{abstract}Based on a general discrete model for a semiflexible polymer chain, we introduce a formal derivation of a kinetic equation for semiflexible polymers in the half-plane via a continuum limit. It turns out that the resulting equation is the kinetic Fokker-Planck-type equation with the Laplace-Beltrami operator under a non-local \textit{trapping} boundary condition. We then study the well-posedness and the long-chain asymptotics of the solutions of the resulting equation. In particular, we prove that there exists a unique measure-valued solution for the corresponding boundary value problem. In addition, we prove that the equation is hypoelliptic and the solutions are locally H\"older continuous near the singular boundary. Finally, we provide the asymptotic behaviors of the solutions for large polymer chains.
\end{abstract}

\maketitle

\tableofcontents
\section{Introduction}
	\subsection{Polymer models in statistical physics. }In this paper, we consider a homogeneous semiflexible polymer chain in the absence of self-avoidance and torsional stress. Polymer models can be obtained by means of limits of random walks, and they have been extensively studied \cite{ doi1988theory,maggs1989unbinding,granek1997semi,werlich2014wang,yang2019study}.
	In particular, semi-flexible polymers which do not self-intersect have also been studied in probability theory \cite{caravenna2009scaling,caravenna2008pinning, duminil2012connective}. The computation of the statistical properties of the resulting polymers has been a difficult problem, in which relevant information has been obtained for some class of models.
	The Brownian motion has been extensively used as a model of polymers in the physical literature (cf. \cite{MR1236140,MR701921,le1994exponential,geman1984local,dynkin1988regularized,ADLER1997271,boal_2012}), in spite of the fact that the trajectories can self-intersect. 
	Another relevant model that has been much less studied, in particular, in the mathematical literature is the model of semi-flexible chains. The model has been developed under the assumption that the polymer consists in a chain of $N$ segments each of which has the same length $\epsilon>0$ so the total length $L$ of the polymer is $L=N\epsilon.$ Then the statistical properties of these polymer chains are given by a Gibbs measure 
	$$\mu_N=\frac{1}{Z_N}\exp \left(-\frac{\mathcal{H}_N}{k_B\mathrm{T}}\right),$$ where $Z_N$ is the normalization factor, $\mathcal{H}_N$ is the Hamiltonian that penalyzes the angle between consecutive polymer segments, $k_B$ is the Boltzmann constant, and $\mathrm{T}$ is the temperature so that $k_B\mathrm{T}$ stands for the thermal energy.  Here the energy $\mathcal{H}_N$ will be assumed to have the form of\begin{equation}\label{ham}\mathcal{H}_N=-\frac{B_s}{\epsilon}\sum_{j=1}^{N-1} (n_j\cdot n_{j+1}-1),\end{equation}
	where $n_j$ is the orientation vector of $j^{\text{th}}$ polymer segment and $B_s$ is the bending stiffness.
	Here, the bending stiffness $B_s$ can be written in terms of the persistence length $l_p$ and the thermal energy as $B_s=l_p k_B\mathrm{T},$ where 
	the persistence length $l_p$ is defined as the projection of the end-to-end vector of the total polymer chain onto the first vector, and it provides the information on the stiffness or the rigidity of the polymer chain. 
	Throughout this paper, we normalize the thermal energy $k_B\mathrm{T}=1$ and assume that the unit length is equal to the persistence length $l_p$. Then we have $B_s=1$ and hence $\frac{\epsilon}{B_S}=\epsilon\ll1$.

	The continuous limit which has been obtained by taking $N\epsilon$ of order one and $\epsilon\to 0$ was introduced in  \cite{kratky1949rontgenuntersuchung} in 1949, and it has been usually called
	the Kratky-Porod model and also the wormlike chain model (WLC).
	The polymer paths associated to the WLC (or Kratky-Porod) model can be described by means of the Ornstein-Uhlenbeck process. This has been noticed after in the physical literature \cite{van1992stochastic,doi:10.1021/ma60010a018,hoshikawa1975average}. In particular, the paths associated to the WLC model in the whole space $\mathbb{R}^d$ can be described by means of the stochastic process associated to the equation:
	\begin{equation}\label{first FP}\partial_t f(t,x,n)+n\cdot \nabla_x f(t,x,n)=\Delta_n f(t,x,n),\end{equation} for $f\in C([0,\infty);\mathcal{M}_+(\mathbb{R}^d\times \mathbb{S}^{d-1}))$  where $t$ is the polymer length parameter and $\Delta_n$ denotes the Laplace-Beltrami operator on $\mathbb{S}^{d-1}$ (cf. \cite{lipowsky1986unbinding,kierfeld2003unbundling,gompper1989unbinding,maggs1989unbinding}). The derivation of \eqref{first FP} in \cite{lipowsky1986unbinding} is formally made by replacing the Gibbs measure $\mu_N$ by the exponents of a path integral. The equation \eqref{first FP} is then derived using the so-called \textit{transfer matrix methods}. We will rederive \eqref{first FP} heuristically in Appendix \ref{sec:formalderivation} by the limit of suitable Markov chains which define the polymer distribution. 
	The rigorous mathematical theory associated to the WLC model is very limited. The stochastic process obtained for the probability measure $\mu_N=\frac{1}{Z_n}e^{-\beta \mathcal{H}_N}$ with $\mathcal{H}_N$ in \eqref{ham} in the whole space $\mathbb{R}^d$ has been studied in \cite{MR2641772}. In particular, the behavior of long polymer chains has been discussed also in the paper. 
	
	In this paper, we are concerned with the interactions between polymer chains and the boundaries of the domain containing them. This issue has been discussed in the physical literature where different types of interactions between the boundaries of the domain and the polymer chain have been introduced (cf. \cite{gompper1989unbinding, maggs1989unbinding, lipowsky1986unbinding, gomez2016hypoelliptic, MR2124366}).  Regarding the discrete semiflexible heterogeneous polymers and their long-chain behavior, we mention \cite{MR2641772}. We also introduce \cite{hryniv2009some} for the rigorous macroscopic scaling limit from the $N$-body Hamiltonian dynamics.
	In this paper, we will study one of the simplest types of interaction potentials between the polymer chain and the boundary of the domain $\Omega$. Namely, we will assume that the boundary $\partial \Omega$ is a constraint that restricts the possible geometry of the polymer chains, but it does not modify the energy of the segments of the chain in any other way. 
	The domain $\Omega$ will be assumed to be a half-plane $\mathbb{R}^2_+=\{x\in \mathbb{R}^2:x_2>0\},$ and we will see formally that the effect of the boundary of $\Omega$ yields a boundary condition for \eqref{first FP}, specifically the so-called \textit{trapping} boundary condition, at least for the polymer lengths $L=N\epsilon$ of order one. If $L\gg 1$, the polymer chains could separate from the boundary after touching it due to large deviation effects, and we will ignore this issue in this paper.

	%
	%

	\subsection{Derivation of the Fokker-Planck equation}
	At a formal level, we derive an initial-boundary value problem for a kinetic partial differential equation starting from the energy of a given discrete chain configuration as introduced in the previous section. The solutions are given by the probability density distribution of the polymers. In the continuum limit, we can formally obtain a boundary-value problem for the kinetic Fokker-Planck equation with the standard Laplacian operator replaced by the Laplace-Beltrami operator which restricts that each monomer has the velocity on $\mathbb{S}^{d}$. 
	
	In the derivation below, we use the variable $t\in[0,T]$ with $N\epsilon = t$ such that we parametrize the semiflexible polymer chain by the total contour length.

	\subsubsection{Dynamics in the whole plane $\mathbb{R}^2$}
	\label{wholeplanederiv}
	First of all, suppose that a semiflexible polymer chain lies in the whole plane $\mathbb{R}^2$ with the initial end of the first monomer segment is located at a given position $x_0$. Suppose that a semiflexible chain consists of $N$ monomers of size $\epsilon$ whose ends are denoted as $\{\xi_j\}_{j=1}^{N}$ and let $n_j\in \mathbb{S}^1$ denote the orientation of the $j^{th}$ monomer that connects $\xi_{j-1}$ and $\xi_{j}$ with $\xi_0=x_0$. In other words, we define for $j\ge 0$,
	$$\xi_j\eqdef x_0+\sum_{i=1}^{j}\epsilon n_i.$$We then introduce the Hamiltonian 
	$$\mathcal{H}_N=\frac{1}{2\epsilon}\sum_{j=1}^{N-1}(n_{j+1}-n_j)^2=-\frac{1}{\epsilon}\sum_{j=1}^{N-1} (n_j\cdot n_{j+1}-1).$$
	\begin{remark}
		We remark that a polymer chain is Markovian from the very first element of the monomers. In other words, the state $(x_{j+1}, n_{j+1})$ would depend only on the previous state $(x_j,n_j)$. 
		This is because we assume that the monomers are just point-particles that do not occupy any volume in the space. Thus, the probability for a monomer meeting one of the previous monomers is zero. Therefore, the evolution depends only on the previous step right before the state. Even in a 2-dimensional space (i.e., a plane), point-monomers yield a Markovian evolution due to the absence of collisions. Then one can ask what is the critical size of the volume in which collisions can begin taking place, and this is one of the open problems. One can also ask if there is some kind of ``kinetic limit" for some scaling of the sizes of the monomer. The list of open problems also includes the evolution of a self-avoiding random walk. In this case, we need the memory of the whole path.
	\end{remark}
	Now we consider the Gibbs probability measure $\mu_N\in \mathcal{M}_+((\mathbb{R}^2)^{N}\times (\mathbb{S}^1)^{N})$, which is given by
	$$\mu_N(x_1,...,x_N;n_1,...,n_N) =\frac{1}{Z_N}\prod_{j=1}^N\delta(x_j-\xi_j)\exp \left(\frac{1}{\epsilon}(n_{j-1}\cdot n_{j}-1)\right),$$ $n_0$ is defined as $n_1$. Here the normalization factor $Z_N$ is defined as 
	$$Z_N\eqdef \iint_{(\mathbb{R}^2)^{N}\times(\mathbb{S}^1)^{N}}\prod_{j=1}^{N}dx_jdn_j \ \mu_N(x_1,...,x_N;n_1,...,n_{N}).$$We also define the probability density distribution $f_N(x,n)$ as
	\begin{multline*}f_N(x,n)\eqdef \int_{(\mathbb{R}^2)^{N}\times(\mathbb{S}^1)^{N-1}}\prod_{j=1}^{N-1}dx_jdn_j\  \mu_N(x_1,...,x_{N-1},x;n_1,...,n_{N-1},n)\\
		=\int_{(\mathbb{S}^1)^{N-1}}\prod_{j=1}^{N-1}dn_j\ \delta\left(x-x_0-\sum_{j=1}^{N} \epsilon n_j\right) \frac{1}{Z_N}\exp \left(\frac{1}{\epsilon}\sum_{j=1}^{N-j}(n_j\cdot n_{j+1}-1)\right),\end{multline*} with $n_N=n$. The total length of the polymer chain $L$ is equal to $L=N\epsilon.$  
	Then the next iterated sequence $f_{N+1}(x,n)$ is given by
	\begin{multline*}f_{N+1}(x,n)=\int_{(\mathbb{S}^1)^{N}} \prod_{j=1}^{N}dn_j\ \mu_{N+1}(x_1,...,x_N,x;n_1,...,n_N,n)\\
		=\int_{(\mathbb{S}^1)^{N}} \prod_{j=1}^{N}dn_j\frac{Z_N}{Z_{N+1}}\mu_{N}(x_1,...,x_N;n_1,...,n_N)\\\times \exp\left(\frac{1}{\epsilon}(n_{N}\cdot n-1)\right) \delta \left(x-x_0-\sum_{j=1}^{N} \epsilon n_j-\epsilon n\right)\\
		=C_N\int_{\mathbb{S}^1}dn_N f_N(x-\epsilon n,n_N)\exp\left(\frac{1}{\epsilon}(n_N\cdot n-1)\right),
	\end{multline*}where $C_N\eqdef \frac{Z_N}{Z_{N+1}}.$
	Now define $F(t,x,n)=f_j(x,n)$ with $j=\frac{t}{\epsilon}.$ Then we have
	$$F(N\epsilon +\epsilon,x,n)
	=C_N\int_{\mathbb{S}^1}dn_N F(N\epsilon, x-\epsilon n,n_N)\exp\left(\frac{1}{\epsilon}(n_N\cdot n-1)\right).$$Together with the formal ansatz that $F$ is smooth, we can take the Taylor expansion as
	$$
	F(N\epsilon +\epsilon,x,n)=	F(N\epsilon ,x,n)+\epsilon \frac{\partial F}{\partial t}(N\epsilon, x,n),
	$$and 
	\begin{multline*}
		C_N\int_{\mathbb{S}^1}dn_N F(N\epsilon, x-\epsilon n,n_N)\exp\left(\frac{1}{\epsilon}(n_N\cdot n-1)\right)\\
		=C_N\int_{\mathbb{S}^1}dn_N \bigg(F(N\epsilon, x,n)-\epsilon n\cdot \frac{\partial F}{\partial x}+(n_N-n)\cdot \nabla_n F(N\epsilon,x,n)\\+\frac{1}{2}(n_N-n)\nabla_n^2 F(N\epsilon,x,n)(n_N-n)\bigg)\exp\left(\frac{1}{\epsilon}(n_N\cdot n-1)\right).
	\end{multline*}
	Note that \begin{equation}\label{normalizedgibbs}C_N\int_{\mathbb{S}^1}dn_N \exp (\epsilon^{-1} (n_N\cdot n))=1,\end{equation} and $$\int_{\mathbb{S}^1}dn_N \ (n_N-n)\exp (\epsilon^{-1} (n_N\cdot n))=0.$$
	Then by defining $t=N\epsilon$, we have
	\begin{multline*}
		\epsilon \frac{\partial F}{\partial t}(t, x,n)+\epsilon n\cdot \frac{\partial F}{\partial x}(t, x,n)\\
		=\Delta_n F(t,x,n)C_N\int_{\mathbb{S}^1}dn_N \bigg(\frac{1}{2}(n_N-n)\otimes (n_N-n)\bigg)\exp\left(\frac{1}{\epsilon}(n_N\cdot n-1)\right).
	\end{multline*}
	Note that 
	\begin{multline}
	\label{CnOe1}C_N\int_{\mathbb{S}^1}dn_N \bigg(\frac{1}{2}(n_N-n)\otimes (n_N-n)\bigg)\exp\left(\frac{1}{\epsilon}(n_N\cdot n-1)\right)\\\approx C_N\int_{\mathbb{S}^1}(\xi\otimes \xi) \exp(-|\xi|^2/(2\epsilon))d\xi\approx O(\epsilon),
	\end{multline}by \eqref{normalizedgibbs} where $\xi=n_N-n$.
	Thus, in the limit $\epsilon\to 0$, we obtain the Fokker-Planck equation in the whole plane $x\in \mathbb{R}^2$ and $n\in \mathbb{S}^1$ as
	$$\frac{\partial F}{\partial t}(t, x,n)+ n\cdot \frac{\partial F}{\partial x}(t, x,n)= D\Delta_n F(t,x,n),$$ with the diffusion coefficient $D$ which is defined as
	$$D= \lim_{\epsilon\to 0} \frac{C_{\frac{t}{\epsilon}}}{\epsilon} \int_{\mathbb{S}^1}(\xi\otimes \xi) \exp(-|\xi|^2/(2\epsilon))d\xi\approx O(1),$$ by \eqref{CnOe1}. 
	\begin{remark}
In the physical situation, the diffusion coefficient $D$ would depend on some physical constants appearing in the Hamiltonian; in this paper, we normalize those constants to be 1.  
	\end{remark}
	
	\subsubsection{Dynamics in the half-plane $\mathbb{R}^2_+$ with the boundary}
	In the paper, we are interested in the boundary effect on the polymer chain in the half-plane. We restrict ourselves to a 2-dimensional model in this paper. For a general 3-dimensional problem, additional geometrical difficulties as well as the effects like the diffusion in the polymer orientation on the surface can arise.
	
	We assume that the polymer that reaches the boundary of the half-plane tends to minimize the bending energy $\frac{1}{2\epsilon}(n_j-n_{j+1})^2$. Then we formally demonstrate in Appendix \ref{sec:formalderivation} that the polymer that reaches the boundary will keep moving along the boundary. We will call this boundary condition the \textit{trapping} boundary condition, as it literally stands for the situation that the polymer is being \textit{trapped} on the boundary. 
	We assume that the boundary of the container does not yield any energy to the polymer chain except for the energy-minimizing modeling assumption.
	The details for the formal derivation of the boundary-value problem will be provided in Appendix \ref{sec:formalderivation}. In particular, we will justify the \textit{trapping} boundary condition near the boundary.
	
	Then, we can characterize the properties of the measure describing the polymer distribution by means of a kinetic equation. The total length of the polymer chain $t = N\epsilon$ plays the role of the time variable of the kinetic equation.
	In this paper we restrict ourselves to the case of two spatial dimensions where a polymer chain lies in a half plane.
	
	\subsection{The 2D Fokker-Planck equation with boundaries}
	In this paper, we mainly consider the Fokker-Planck system in a 2-dimensional half-plane $\mathbb{R}^2_+=\{x\in \mathbb{R}^2:x_2>0\}.$  Throughout the paper, our phase space is then $(x,n)\in \mathbb{R}^2_+\times \mathbb{S}^1$, as we restrict the velocity of each monomer to be 1 as shown in the derivation above and in Appendix \ref{sec:formalderivation}. If we use the phase variable $(x,n)$, then we denote the probability measure as $F(t,x,n)$. On the other hand, we also use another coordinate system of $(t,x_1,x_2,\theta)$ where $t\in[0,T]$, $x_1\in\mathbb{R}$, $x_2\in (0,\infty)$, and $\theta \in [-\pi,\pi]$ such that $n=(\cos\theta,\sin\theta).$ In this case, we will use the notation for the measure as $f(t,x,\theta)$, so that $F(t,x,n)=f(t,x,\theta).$ Here we emphasize that the usual time-variable $t\in [0,T]$ means the total polymer length throughout the paper. The velocity variable $n$ is in $\mathbb{S}^1$ and it is parametrized in terms of $\theta \in [-\pi,\pi]$. The phase boundary is defined as $x_1\in \mathbb{R},\ x_2=0,\ \text{and}\ \theta \in [-\pi,\pi].$

	The Fokker-Planck equation for semiflexible polymers reads
	$$\partial_t F+n\cdot \nabla_x F = \Delta_n F, \text{ on } t\in [0,T]\times \mathbb{R}^2_+\times \mathbb{S}^1,$$ where $\Delta_n$ is the Laplace-Beltrami operator. Using the other coordinate representation, we also have
	\begin{equation}\label{FP eq}
		\partial_t f+(\cos\theta,\sin\theta)\cdot \nabla_x f =\partial_\theta^2 f,\ \text{on} \ [0,T]\times \mathbb{R}^2_+\times [-\pi,\pi],
	\end{equation}
	with the initial condition
	\begin{equation}
		\label{initial}
		f(0,x,\theta)=f_{in}(x,\theta),
	\end{equation}
	and the $2\pi$-periodic boundary condition with respect to $\theta$
	\begin{equation}
		\label{periodic}
		f(t,x,-\pi)=f(t,x,\pi)\text{ and }\partial_\theta f(t,x,-\pi)=\partial_\theta f(t,x,\pi).
	\end{equation}
	Here $f_{in}$ is any nonnegative Radon measure.
	Due to the linearity and the invariance under translations, it is enough to consider the case in which $f_{in}$ is a Dirac mass at some point $(x_1,x_2)=(0,a)$ for some $a\ge0$ with the direction $n_0\in \mathbb{S}^1$ without loss of generality.
	A particular case is when $a=0$. Then $n_0$ can be only in two directions (\textit{trapping} boundary condition), either with $\theta=0$ or $\theta =-\pi$.
	The polymer undergoes the full Brownian motion and the polymer will eventually approaches to the boundary $x_2=0$. The asymptotics in the limit $a \rightarrow \infty$ or $a\rightarrow 0$ are also interesting problems to be considered.

	\subsection{The trapping boundary condition}
	We consider the boundary condition of the 2D Fokker-Planck equation where the polymer chain aligns in the direction in which it makes the smallest angle with the  angle made by the tangent vector to the polymer arriving to the boundary.
	It is convenient to write the model in geometrical terms, using the variables $(x,n)$ and to explain how is the angle condition after the polymer reaches the boundary.
	
	The boundary conditions are obtained under the assumption that the only effect of the boundary is to impose a constraint on the directions connecting the monomers of the polymer chain. 
	The energy is defined by means of local interactions between the monomers of the polymer chain; i.e., the minimum of the energy corresponds to the polymers locally aligned. We assume that the same definition of energy is valid after the monomers reach the boundary, but this imposes constraints in the admissible directions.

	We assume that the probability measure $F(t,x,n)$ satisfies the following boundary conditions:
	\begin{equation}
		\label{boundaryconF}
		\begin{split}
			F(t,x_1,0,n)&=\ 0,\text{  if } n \neq \pm e,\\
			\partial_t F(t,x_1,0,e)&=\lim_{x_2\to 0^+}\int_{\nu\cdot n\ge 0,\ e\cdot n >0} \ F(t,x_1,x_2,n)(\nu\cdot n)\ dS_n \text{, and}\\
			\partial_t F(t,x_1,0,-e)&=\lim_{x_2\to 0^+}\int_{\nu\cdot n\ge 0,\ e\cdot n <0} \ F(t,x_1,x_2,n)(\nu\cdot n)\ dS_n,
		\end{split}
	\end{equation}where $e\eqdef (1,0)\in \mathbb{S}^1$ and $\nu=(0,-1)$ is the outward normal vector at the boundary. The weight $(\nu\cdot n)$ on the measure $dS_n$ is physical in the manner that the weight $(\nu\cdot n)$ describes the total net flux of particles in the direction $\nu$, whose velocities are in the direction of $n$. Then $(\nu\cdot n)dS_n$ describes the number of particles passing through the boundary per unit length of the boundary.
	
	In terms of $f(t,x,\theta)$ the boundary conditions are equivalent to
	\begin{equation}
		\label{boundarycon}
		\begin{split}
			f(t,x_1,0,\theta)&=\ 0,\text{  if } \theta \neq 0, \ -\pi,\\
			\partial_t f(t,x_1,0,0)&=\lim_{x_2\to 0^+}\int_{(-\pi/2,0]} \ f(t,x_1,x_2,\theta)(-\sin\theta)\ d\theta ,\\
			\partial_t f(t,x_1,0,-\pi)&=\lim_{x_2\to 0^+}\int_{[-\pi,-\pi/2)} \ f(t,x_1,x_2,\theta)(-\sin\theta)\ d\theta,
		\end{split} 
	\end{equation} and the periodic boundary condition. Physically, the first line \eqref{boundarycon}$_1$ describes that the polymer that reaches the boundary $x_2=0$ can have only two directions $\theta=0$ or $\theta=-\pi$. The second and the third lines \eqref{boundarycon}$_2$ and \eqref{boundarycon}$_3$ describe that the rate of changes in the probability distributions $f(t,x_1,0,\theta)$ with $\theta=0$ or $-\pi$ at the boundary can be expressed as the sum of the probability distribution that approaches to the boundary $x_2=0$ with either the angle $\theta \in(-\pi/2,0]$ or $\theta \in [-\pi,-\pi/2)$ with an additional multiplier $(-\sin\theta)$.
	\subsection{Reformulation of the problem}\label{sec:reformulation}
	Equivalently, each probability distribution \\$f(t,x_1,x_2,\theta)$ can further be decomposed into three parts as the following:
	$$
	f(t,x_1,x_2,\theta)=\rho_+(t,x_1)\delta(x_2)\delta(\theta)+\rho_-(t,x_1)\delta(x_2)\delta(\theta+\pi)+f_r(t,x_1,x_2,\theta),$$ where $\rho_+ (t),\ \rho_-(t) \in \mathcal{M}_+( \mathbb{R})$, and  $f_r(t)\in \mathcal{M}_+( \mathbb{R}^2_+\times \mathbb{S}^1)$ is supported on $x_2>0$ with $f_r(\{x_2=0\})\equiv 0$. Then obtaining a solution $f$ is also equivalent to obtaining the tuple $(f_r,\rho_+,\rho_-)$.
	
	By \eqref{FP eq}- \eqref{periodic} and \eqref{boundarycon}, one can check that the system which the tuple $(f_r,\rho_+,\rho_-)$ satisfies is now 
	\begin{equation}\label{FP eq tu}\begin{split}
			\partial_t f_r+(\cos\theta,\sin\theta)\cdot \nabla_x f_r &=\partial_\theta^2 f_r,\ \text{on} \ [0,T]\times \mathbb{R}\times (0,\infty)\times [-\pi,\pi],\\
			f_r(\{x_2=0\})&=0,\\
			f_r(t,x,-\pi)=f_r(t,x,\pi),\ &\partial_\theta f_r(t,x,-\pi)=\partial_\theta f_r(t,x,\pi),
		\end{split}
	\end{equation}
	and
	\begin{equation}\label{boundarycon tu}
		\begin{split}
			\partial_t \rho_++\partial_{x_1}\rho_+&=\lim_{x_2\to 0^+}\int_{(-\pi/2,0]} \ f_r(t,x_1,x_2,\theta)(-\sin\theta)\ d\theta \text{, and}\\
			\partial_t \rho_--\partial_{x_1}\rho_-&=\lim_{x_2\to 0^+}\int_{[-\pi,-\pi/2)} \ f_r(t,x_1,x_2,\theta)(-\sin\theta)\ d\theta.
		\end{split}
	\end{equation}

	In this paper, we are interested in measure valued solutions of \eqref{FP eq}, \eqref{initial}, \eqref{periodic}, and \eqref{boundarycon}. To this end, we will study a suitable adjoint problem and show that the adjoint problem has the maximum principle. The main tool for the well-posedness of the problem is the classical Hille-Yosida theorem, via which we will consider the corresponding elliptic problem associated to the adjoint problem which would also encode the information about the \textit{trapping} boundary condition \eqref{boundarycon} for polymers.
	
	Here we also introduce a system for the mass density  $\rho_1=\rho_1(t,x_2,\theta)$. Here $\rho_1$ is defined as 
	$$
	\rho_1(t,x_2,\theta)\eqdef \int_{\mathbb{R}}f_r(t,x_1,x_2,\theta)dx_1 ,
	$$ and is obtained via the integration of $f$ with respect to $x_1$ variable. It physically stands for the mass density distribution at each point $(t,x_2,\theta)$ in the set $[0,T]\times (0,\infty)\times [-\pi,\pi].$
	By integrating \eqref{FP eq tu} with respect to $x_1$ on $\mathbb{R}$, we obtain
	\begin{equation}\label{FP eq rho}\begin{split}
			\partial_t \rho_1+\sin\theta\partial_{x_2}\rho_1 &=\partial_\theta^2 \rho_1,\ \text{on} \ [0,T]\times  (0,\infty)\times [-\pi,\pi],\\
			\rho_1(\{x_2=0\})&=0,\\
			\rho_1(t,x_2,-\pi)=\rho_1(t,x_2,\pi),\ &\partial_\theta 	\rho_1(t,x_2,-\pi)=\partial_\theta \rho_1(t,x_2,\pi),
		\end{split}
	\end{equation}
	and
	\begin{equation}\label{boundarycon rho}
		\begin{split}
			\partial_t \int_{\mathbb{R}}dx_1\ \rho_+(t,x_1)&=\lim_{x_2\to 0^+}\int_{(-\pi/2,0]} \ \rho_1(t,x_2,\theta)(-\sin\theta)\ d\theta \text{, and}\\
			\partial_t \int_{\mathbb{R}}dx_1\ \rho_-(t,x_1)&=\lim_{x_2\to 0^+}\int_{[-\pi,-\pi/2)} \ \rho_1(t,x_2,\theta)(-\sin\theta)\ d\theta.
		\end{split}
	\end{equation}

	\subsection{Compactification of the phase space and a topological set \texorpdfstring{$X$}{}}In order to define the notion of weak solutions, we first define a topologically compact set $X$ for our phase space and a Banach space $C(X)$ under the uniform topology.
	
	\begin{definition}\label{setS2d}We define a set $X_0$ as $X_0\eqdef (-\infty,\infty)\times[0,\infty)\times [-\pi,\pi],  $ with the additional identifications that 
		for any $(x_1,x_2)\in (-\infty,\infty)\times [0,\infty)$, $(x_1,x_2,\pi)$ and $(x_1,x_2,-\pi)$ are identified.
		Then we define the extended space $X=X_0\cup \{\infty\}$ and endow it a natural topology inherited from $X_0$ complemented by the following set of neighborhoods of the point $\infty$: 
		$$\mathcal{O}_M=\{(x_1,x_2,\theta)\in (-\infty,\infty)\times[0,\infty)\times [-\pi,\pi]\ :\ |x_1|>M\text{ or }x_2>M\},\ M>0.$$
	\end{definition}
	Note that $X$ is topologically a compact set. 
	A $C^0$ function $\phi$ on this set can be identified with the bounded $C^0$ function $\phi$ on $\mathbb{R}\times[0,\infty)\times [-\pi,\pi]$ that satisfies 
	\begin{equation}
		\label{adjoint boundarycon3}
		\begin{split}
			\phi(x_1,x_2,-\pi)&=\phi(x_1,x_2,\pi)\ \text{for all} \ x\in \mathbb{R}\times [0,\infty),
		\end{split}
	\end{equation}and the limit of $$\lim_{|x|\rightarrow \infty} \sup_{\theta\in  [-\pi,\pi]} |\phi(x,\theta)|$$ exists.
	We denote the set of these $C^0$ functions as $C(X)$. We endow the set $C(X)$ a norm 
	$$\|\phi\|\eqdef \sup_{(x,\theta)\in X}|\phi (x,\theta)|,$$ so that the set $C(X)$ is now
	a Banach space. Also, we define the set $C^m(X)$ and $C^\alpha(X)$ for a non-negative integer $m$ and a multi-index $\alpha=(\alpha_1,\alpha_2,\alpha_3) \in \mathbb{N}^3_0$ as
	\begin{equation}\label{setc2s2d}C^m(X)\eqdef \left\{\phi\in C(X) : \|\phi\|_{C^m}\eqdef \sum_{|\alpha|\le m}\sup_{(x,\theta)\in X}|\partial^\alpha_{x,\theta}\phi (x,\theta)|<\infty
		\right\},\end{equation}
	and\begin{equation}\label{setXcalpha}C^\alpha(X)\eqdef \left\{\phi\in C(X) : \|{\phi}\|_{C^\alpha}\eqdef \sum_{\beta\le \alpha}\sup_{(x,\theta)\in X}|\partial^\beta_{x,\theta}{\phi} (x,\theta)|<\infty
		\right\},\end{equation}
	where we used the partial order notation for the multi-indices $\beta\le \alpha$ which means that $\forall i=1,2,3,\ 0\le\beta_i\le \alpha_i.$
	\subsection{Weak formulation}
	In this subsection, we define the notion of a weak solution. Motivated by the discussion in Section \ref{sec:reformulation}, we define the notion of a weak solution as follows: \begin{definition}\label{weaksoldef}
Let $X$ be the topological \textit{compatification} of $\mathbb{R}^2_+\times [-\pi,\pi]$ by means of Definition \ref{setS2d}. 		We call that a nonnegative Radon measure $f(t)\in \mathcal{M}_+(X)$ is a weak solution to the system \eqref{FP eq}, \eqref{initial}, \eqref{periodic}, and \eqref{boundarycon} if  we have\begin{equation}\label{decomposition}
			f(t,x_1,x_2,\theta)=\rho_+(t,x_1)\delta(x_2)\delta(\theta)+\rho_-(t,x_1)\delta(x_2)\delta(\theta+\pi)+f_r(t,x_1,x_2,\theta),\end{equation} for some $\rho_+ (t),\ \rho_-(t) \in \mathcal{M}_+( \mathbb{R})$, and some $f_r(t)\in \mathcal{M}_+(X)$ supported on $x_2>0$, and $(f_r,\rho_+,\rho_-)$ further solves
		\begin{multline}\label{eq.weaksol}\int_0^T dt\iiint_{\mathbb{R}^2_+\times [-\pi,\pi]} dx_1dx_2d\theta \left[\partial_t \phi +(\cos\theta,\sin\theta)\cdot \nabla_x \phi + \partial_\theta^2 \phi\right]f_r\\
			+\int_0^T dt\int_{\mathbb{R}}dx_1 \ [\partial_t\phi(t,x_1,0,0)+\partial_{x_1}\phi (t,x_1,0,0)]\rho_+(t,x_1)\\+\int_0^T dt\int_{\mathbb{R}}dx_1 \ [\partial_t\phi(t,x_1,0,-\pi)-\partial_{x_1}\phi (t,x_1,0,-\pi)]\rho_-(t,x_1)\\=\iiint_{\mathbb{R}^2_+\times [-\pi,\pi]} dx_1dx_2d\theta \ f(T)\phi(T)-\iiint_{\mathbb{R}^2_+\times [-\pi,\pi]} dx_1dx_2d\theta \ f_{in}\phi(0),\end{multline}for any $\phi\in C([0,T]\times X)$ such that $\phi \in C^{1,1,1,2}_{t,x_1,x_2,\theta}(\tilde{\Omega})$ with $\tilde{\Omega}\eqdef  (0,T)\times (-\infty,\infty)\times \{(0,\infty)\times (-\pi,\pi) \cup \{x_2=0\}\times \{(-\pi,-\pi/2)\cup(-\pi/2,0)\cup (0,\pi)\}\}\subset X$ and satisfies
		the boundary condition 
		\begin{equation}\notag
			\begin{split}
				\phi(t,x_1,0,\theta)&=\phi(t,x_1,0,\pi)\ \text{if}\ -\pi\le \theta<-\frac{\pi}{2},\ \\
				\phi(t,x_1,0,\theta)&=\phi(t,x_1,0,0)\ \text{if}\ -\frac{\pi}{2}<\theta\le 0,\ \\
				\phi(t,x_1,x_2,-\pi)&=\phi(t,x_1,x_2,\pi).
			\end{split}
		\end{equation}
	\end{definition}
	\begin{remark}
		We remark a posteriori that we recover the strong formulation from the weak formulation once we show that any weak solution is sufficiently regular. 
	\end{remark}
	\subsection{Main theorems}
	We are now ready to state the main theorems of the paper. \begin{theorem}\label{mainexistencethm}
		Suppose that $f_{in}=f_{in}(x,\theta)\in \mathcal{M}_+(X)$ is a non-negative Radon measure. Then there exists a unique measure solution $f$ to \eqref{FP eq}, \eqref{initial}, \eqref{periodic}, and \eqref{boundarycon} in the sense of Definition \ref{weaksoldef}.
	\end{theorem}In addition, the solution satisfies the following properties:
	\begin{theorem}\label{maintheorem}
		The unique weak solution $f$ of Theorem \ref{mainexistencethm} satisfies the following properties:
		\begin{enumerate}
			\item \label{hypoellipticitymaintheorem}(Hypoellipticity) Define the domain $\tilde{\Omega}\eqdef  (0,T)\times (-\infty,\infty)\times \{(0,\infty)\times (-\pi,\pi) \cup \{x_2=0\}\times \{(-\pi,-\pi/2)\cup(-\pi/2,0)\cup (0,\pi)\}\}$. For any point $(t_0,x_0,\theta_0)\in \tilde{\Omega}$, there exists $r>0$ such that the weak solution $f$ in the sense of Definition \ref{weaksoldef} is $C^\infty$ on 
			$B_r(t_0,x_0,\theta_0)\cap \tilde{\Omega}$.
			\item (Local H\"older continuity in the domain including the singular boundary)\label{holdercontinuity} The weak solution $f$ further satisfies the H\"older regularity in $x_2$ and $\theta$ variables  $f(t,x_1,\cdot,\cdot)\in C^{0,\alpha}_{x_2, loc}([0,\infty) ; C^{0,3\alpha}_{\theta,loc}([-\pi,-\pi/2)\cup(-\pi/2,\pi])),$ for $t\ge 0$, $x_1\in \mathbb{R}$ for any $\alpha \in (0,1/6)$ in the domain including the singular boundary $x_2=0$ and $\theta\in \{-\pi,0\}$. \item (Accumulation of mass on the boundary) \label{mainconmass1}For any $0\le t_1<t_2<\infty$, we have
			\begin{multline*}
				\ \  \quad
			 \int_{\mathbb{R}}dx_1\int_{\{x_2>0\}}dx_2 \int_{-\pi}^\pi d\theta \ f_r(t_1,x_1,x_2,\theta)\\
				 \ge \int_{\mathbb{R}}dx_1\int_{\{x_2>0\}}dx_2 \int_{-\pi}^\pi d\theta \ f_r(t_2,x_1,x_2,\theta).\end{multline*}
			\item (Conservation of total mass)\label{mainconmass2} The total mass on the domain including the boundary is conserved; for any $t\ge 0$, 
			$$\frac{d}{dt}\int_{[-\pi,\pi]}d\theta\int_{[0,\infty)}dx_2 \int_{\mathbb{R}}dx_1 \ f(t,x,\theta)=0.$$ 
			\item (Long-chain asymptotics) \label{longtimeasymptoticmain}For all $f_{in}=f_{in}(x,\theta)\in \mathcal{M}_+(X)$, we have the convergence 
			$$\iiint_{\mathbb{R}^2_+\times [-\pi,\pi]}dxd\theta \ f_r(t,x,\theta)\rightharpoonup 0,$$ as $t\rightarrow \infty.$
		\end{enumerate}
	\end{theorem}
	We remark that the solution is very weak. We suppose that the initial distribution $f_{in}$ is a nonnegative Radon measure and the solution to the problem is also a Radon measure, so we do not expect to obtain an $L^\infty$ estimate for the measures for instance. Thus, we deal with suitable adjoint problems that have the maximum principle and are closely related to the generators of stochastic processes. Then the existence and the uniqueness of the solution to the original problem can be obtained by duality.

	\subsection{Adjoint problems}\label{adjointproblem section}In this section, we introduce corresponding dual adjoint problems to \eqref{FP eq tu}-\eqref{boundarycon tu}. We remark that the adjoint problems have the maximum principle and are closely related to the generators of stochastic processes.
	
	Motivated by the weak formulation in Definition \ref{weaksoldef}, we define a backward-in-$t$ dual adjoint problem for the system \eqref{FP eq tu}-\eqref{boundarycon tu} as
	\begin{equation}\label{backwardintimeproblem}\begin{split}
			-\partial_t\phi-(\cos\theta,\sin\theta)\cdot \nabla_x \phi &= \partial_\theta^2 \phi,\ \text{on} \ [0,T]\times X,\\
			\partial_t\phi(t,x_1,0,0)&=-\partial_{x_1}\phi (t,x_1,0,0),\\
			\partial_t\phi(t,x_1,0,-\pi)&=\partial_{x_1}\phi (t,x_1,0,-\pi),
		\end{split}
	\end{equation}
	with the initial condition
	\begin{equation}\label{backward initial}
		\phi(T,x,\theta)=\phi_T(x,\theta)\in C(X),
	\end{equation}
	and the boundary condition 
	\begin{equation}\label{backward boundary}
		\begin{split}
			\phi(t,x_1,0,\theta)&=\phi(t,x_1,0,\pi)\ \text{if}\ -\pi\le \theta<-\frac{\pi}{2},\ \text{and}\\
			\phi(t,x_1,0,\theta)&=\phi(t,x_1,0,0)\ \text{if}\ -\frac{\pi}{2}<\theta\le 0,\\
			\phi(t,x_1,x_2,-\pi)&=\phi(t,x_1,x_2,\pi).
		\end{split}
	\end{equation}In order to change the system to a forward-in-$t$ system, we make a change of variables $t\mapsto t'=T-t$ and obtain
	\begin{equation}\label{FP adjoint eq}   \begin{split}
			\partial_t\phi-(\cos\theta,\sin\theta)\cdot \nabla_x \phi &= \partial_\theta^2 \phi,\ \text{on} \ [0,T]\times X,\\
			\partial_t\phi(t,x_1,0,0)&=\partial_{x_1}\phi (t,x_1,0,0),\\
			\partial_t\phi(t,x_1,0,-\pi)&=-\partial_{x_1}\phi (t,x_1,0,-\pi),
		\end{split}
	\end{equation}
	with the initial condition
	\begin{equation}
		\label{adjoint initial}
		\phi(0,x,\theta)=\phi_{in}(x,\theta)\in C(X),
	\end{equation}
	and the boundary condition for $t\ge 0$
	\begin{equation}
		\label{adjoint boundarycon}
		\begin{split}
			\phi(t,x_1,0,\theta)&=\phi(t,x_1,0,\pi)\ \text{if}\ -\pi\le \theta<-\frac{\pi}{2},\ \text{and}\\
			\phi(t,x_1,0,\theta)&=\phi(t,x_1,0,0)\ \text{if}\ -\frac{\pi}{2}<\theta\le 0.
		\end{split}
	\end{equation}
	Also, we require the periodic boundary condition with respect to $\theta$ as
	\begin{equation}
		\begin{split}
			\label{adjoint periodic}
		\phi(t,x_1,x_2,-\pi)=\phi(t,x_1,x_2,\pi).
		\end{split}
	\end{equation}Thus we observe that the initial $\phi_{in}=\phi_{in}(x_1,x_2,\theta)\in C(X)$ is assumed to satisfy
	\begin{equation}\label{forward adjoint initial}
		\begin{split}
			\phi_{in}(x_1,0,\theta)&=\phi_{in}(x_1,0,\pi)\ \text{if}\ -\pi\le \theta<-\frac{\pi}{2},\ \\
			\phi_{in}(x_1,0,\theta)&=\phi_{in}(x_1,0,0)\ \text{if}\ -\frac{\pi}{2}<\theta\le 0\\
			\phi_{in}(x_1,x_2,-\pi)&=\phi_{in}(x_1,x_2,\pi).
		\end{split}
	\end{equation}

	Also, we introduce the dual adjoint problem for the system \eqref{FP eq rho}-\eqref{boundarycon rho} for the total mass density in $x_1$ variable.
	The (forward-in-time) dual adjoint problem of the system that a test function $\varphi=\varphi(t,x_2,\theta)$ satisfies is
	\begin{equation}\label{FP modified adjoint eq}\begin{split}
			\partial_t\varphi-\sin\theta\partial_{x_2} \varphi &= \partial_\theta^2 \varphi,\ \text{on} \ [0,T]\times (0,\infty)\times [-\pi,\pi],\\
			\partial_t\varphi(t,0,0)&=\partial_t\varphi(t,0,-\pi)=0,
		\end{split}
	\end{equation}
	with the initial condition
	\begin{equation}
		\label{modified adjoint initial}
		\varphi(0,x_2,\theta)=\varphi_{in}(x_2,\theta)\in C([0,\infty)\times [-\pi,\pi)),
	\end{equation}
	and the boundary condition for $t\ge 0$
	\begin{equation}
		\label{modified adjoint boundarycon}
		\begin{split}
			\varphi(t,0,\theta)&=\varphi(t,0,-\pi)\ \text{if}\ -\pi\le \theta<-\frac{\pi}{2},\ \text{and}\\
			\varphi(t,0,\theta)&=\varphi(t,0,0)\ \text{if}\ -\frac{\pi}{2}<\theta\le 0.
		\end{split}
	\end{equation}
	Also, we require the periodic boundary condition with respect to $\theta$ as
	\begin{equation}
		\label{modified periodic}
		\begin{split}
		\varphi(t,x_2,-\pi)=\varphi(t,x_2,\pi).
	\end{split}
	\end{equation}
	So we assume $\varphi_{in}=\phi_{in}(x_2,\theta)\in C([0,\infty)\times [-\pi,\pi))$ satisfies
	\begin{equation}\label{forward adjoint initial0}
		\begin{split}
			\varphi_{in}(0,\theta)&=\varphi_{in}(0,-\pi)\ \text{if}\ -\pi\le \theta<-\frac{\pi}{2},\ \\
			\varphi_{in}(0,\theta)&=\varphi_{in}(0,0)\ \text{if}\ -\frac{\pi}{2}<\theta\le 0\\
			\varphi_{in}(x_2,-\pi)&=\varphi_{in}(x_2,\pi).
		\end{split}
	\end{equation}on the initial condition $\varphi_{in}\in C([0,\infty)\times [-\pi,\pi))$.
	\subsection{Main novelties and strategies}\label{sec:main difficulties}In this subsection, we discuss several difficulties that the analysis of the polymer model with the boundary involves. The main difficulties and our corresponding novel approaches include the followings.
	
	\subsubsection{Semi-flexible polymers and the non-local trapping boundary conditions}In this paper, we cast a kinetic model for semi-flexible polymers and novel non-local boundary conditions for the kinetic PDE that models semi-flexible polymers. The novel non-local boundary condition, which we call as the \textit{trapping} boundary condition throughout this paper, has been derived under very careful analysis of the dynamics of semi-flexible polymers that minimizes polymer's bending energy at the boundary. Different from the standard Fokker-Planck-Kolmogorov type operator $(\partial_t-v\cdot\nabla_x-\Delta_v)$ in the whole space $v\in \mathbb{R}^{d}$, we obtain the Fokker-Planck-Kolmogorov-like operator $(\partial_t-n\cdot\nabla_x-\Delta_n)$ on the sphere $n\in \mathbb{S}^{d-1}$. In particular, this requires some modifications for the approaches for the singular boundaries that were developed in \cite{MR3237885,MR3897919,hwang2019nonuniqueness,MR3788197,MR3436235}. In this paper, we provide the first application of the kinetic equation to study semi-flexible polymers in a rigorous mathematical manner. 
	\subsubsection{Effects of the singular boundary}One of the main difficulties in our analysis arises from the presence of the singular boundary; it has been well-known that the kinetic equation with the boundaries have singular boundaries which are called the \textit{grazing} boundaries \cite{MR3897919,hwang2019nonuniqueness,MR3436235,MR3237885,MR3788197,ARMAjang,KL-2018-CPAM,MR2679358,MR3592757,MR2855537,MR3961294,MR3023388}.   In our problem, the singular boundaries occur on the boundary $x_2=0$ at $\theta=0$ and $\theta=-\pi$. Compared to the previous results on the mathematical analysis of the kinetic Fokker-Planck equation with boundaries, the velocity in this paper is not a homogeneous function, and therefore the singularities have to be studied locally.	The boundaries have a \textit{non-symmetric} behavior given that the characteristics enter into the domain in parts of the boundary, and they leave the domain in other parts of the boundary. 
	 Near the singular domain, we construct sub- and super-solutions via the self-similar profiles and derive the maximum principle to prove the H\"older regularity of solutions near the singular boundary. 
	\subsubsection{A pathological set}
	An additional difficulty arises from the analysis near the pathological set $\{x_2=0\}\times \left\{\theta=-\frac{\pi}{2}\right\}$ and this is one of the special properties that the kinetic polymer model has. This set refers to the polymers that approach the boundary $x_2=0$ in the perpendicular direction at right angles. Recall that the \textit{trapping} boundary condition \eqref{boundarycon} that we obtain in the derivation of the model in Appendix \ref{sec:formalderivation} creates the boundary conditions \eqref{adjoint boundarycon} for the adjoint problem. Then, we remark that a solution to the adjoint problem \eqref{FP adjoint eq}-\eqref{adjoint periodic} that are smooth in $x_2>0$ does not have a limit as $(x_2,\theta)\rightarrow (0,-\pi/2)$ by following the perpendicular trajectory if the two values at the boundary $\phi(t,x_1,0,0)$ and $\phi(t,x_1,0,\pi)$ are different. This makes it difficult to define a compact topological phase-space $S$. Thus, it does not guarantee that the set of continuous functions $C(S)$ is a Banach space under the uniform topology, which is crucial for the application of the classical Hille-Yosida theorem. As a remedy, we regularize the boundary condition on $x_2=0$ and $\theta\in[-\pi,0]$ such that the boundary condition \eqref{adjoint boundarycon} no longer has a jump discontinuity around the pathological set $(x_2,\theta)=(0,-\pi/2)$. Then, this allows us to define a compact phase-space $S$ and we can show that the domain $\mathcal{D}(\mathcal{L})$ of the operator $\mathcal{L}$ is dense in $C(S)$. This will be used for the proof that the operator $\mathcal{L}$ is indeed a Markov generator. It turns out that it is effective to define the regular boundary conditions as the solutions to the differential equations \eqref{regularized modified adjoint boundarycon} and \eqref{regularized adjoint boundarycon}, whose solutions are smooth, converge to the original Heaviside-type discontinuous boundaries, and have the decaying properties that naturally come from the construction of the boundary equations. 
	
	\subsubsection{Proof of the hypoellipticity}
Away from the singular boundary and the pathological set that we introduce above, we provide a proof of the hypoellipticity using the techniques developed by H\"ormander \cite{Hormander}. Indeed, the standard kinetic Fokker-Planck operator $(\partial_t-v\cdot\nabla_x-\Delta_v)$ for $v\in \mathbb{R}^d$ has been well-known to make the solution smoothing in all variables as shown in \cite{MR3237885,MR3897919,hwang2019nonuniqueness,MR3788197,MR3436235} away from the singular boundary, but the hypoellipticity for the operator $(\partial_t-n\cdot\nabla_x-\Delta_n)$ on the sphere $n\in \mathbb{S}^{d-1}$ has not been studied well. In this paper, we provide a much simpler proof of the hypoellipticity away from the singular boundary using the techniques developed by H\"ormander \cite{Hormander} and the use of the extension of the domain beyond the boundary, which can also be applied to the operator $(\partial_t-v\cdot\nabla_x-\Delta_v)$ for $v\in \mathbb{R}^d$, not just for the operator $(\partial_t-n\cdot\nabla_x-\Delta_n)$ on the sphere $n\in \mathbb{S}^{d-1}$. This will be provided in Section \ref{sec:hypoellipticity}.

	\subsubsection{Use of the generators of stochastic processes}\label{sec:difficulties Cadleg}
	We are considering the generators of stochastic processes in which the particle
	reach a point or a set and has an instantaneous jump to another point. Thus, we have
	to determine how the dynamics would be afterwards. There are several mathematical
	subtleties as well as some examples of difficulties that
	arise in some cases.
	
	In principle the main evolution of a particle and the boundary effect that we need to consider come from the following stochastic differential equation:%
	\begin{equation}\notag
		dX_2=\sin\Theta dt\ \ ,\ \ d\Theta=dW \label{A3}%
	\end{equation}
	where $X_2>0$ and $\Theta\in [-\pi,\pi]$. Here we neglect the $X_1$ variable for the moment as it is in the whole line without boundaries. Then the trajectories reach the set $X_2=0$
	with probability one in a finite time and this happens along the interval $-\pi\le \Theta\le 0.$
	We will assume that after the trajectory reaches the point $X_2\left(
	t_{0}\right)  =0,\ \ \Theta\left(  t_{0}\right)  \in (-\pi/2,0),$ it jumps instantaneously
	to $\Theta=0.$ Similarly, we assume that after the trajectory reaches the point $X_2\left(
	t_{0}\right)  =0,\ \ \Theta\left(  t_{0}\right)  \in (-\pi,-\pi/2),$ it jumps instantaneously
	to $\Theta=-\pi.$ 
	
	Then one problem that arises is the following. The usual theory of \textit{Markov} processes
	as considered in Liggett \cite{Liggett} and other classical literature assumes that the trajectories of the process are
	\textit{Cadleg} (continuous to the right and with a well defined limit by the left), and
	this is not the case for the processes that we are considering with the instantaneous jumps. Indeed, suppose that we write $\xi_{t}=\left(  X_2\left(
	t\right)  ,\Theta\left(  t\right)  \right)  $ and suppose that $X_2\left(
	t_{0}^{-}\right)  =0.$ Then we have to alternatives: (1) To impose $\Theta\left(
	t_{0}\right)  =\Theta_{0}<0.$ In this case, we have continuity by the left, but then
	the solution would not be continuous by the right, because $\Theta\left(
	t_{0}^{+}\right)  =0$ or $-\pi$. (2) Therefore, the only possible alternative in order to have a
	\textit{Cadleg} process is to define $\Theta\left(  t_{0}\right)  =0$ or $-\pi$ depending on the value of $\Theta(t_0^-)$. Then
	$\Theta\left(  t_{0}^{-}\right)  =\Theta_{0}<0$ (the limit exists), but
	$\Theta\left(  t_{0}^{+}\right)  =\Theta\left(  t_{0}\right)  =0$ or $-\pi$.
	Then the difficulty is that this process is not defined if we consider as an initial
	value the point $\left(  X_2,\Theta\right)  =\left(  0,\Theta_{0}\right)  $ with
	$\Theta_{0}<0.$ Therefore we cannot define the semigroup or the generator at
	this point. Indeed, we recall that given a function $u$ continuous in
	the space in which we consider the problem we have
	\begin{equation}\notag
		S\left(  t\right) u\left(  x_2,\theta\right)  =\mathbb{E}\left(  {u}\left(
		\xi_{t}^{x_2,\theta}\right)  \right)  \label{A2}%
	\end{equation}
	where $\xi_{t}^{x}$ is the stochastic process starting at $x$ at the time
	$t=0.$ This is not defined for the points $\left(  X_2,\Theta\right)  =\left(
	0,\Theta_{0}\right)  $ with $\Theta_{0}<0.$
	
	Seemingly this poses difficulties when it comes to applying the standard theory of
	\textit{Markov} processes. From the technical point of view, one of the possible ways of dealing
	with this problem is to define a different stochastic process in which all the
	points $\left(  X_2,\Theta\right)  =\left(  0,\Theta_{0}\right)  $ with
	$\Theta_{0}\in (-\pi/2,0]$ are just identified as a single point $(0,0)$ and all the
	points $\left(  X_2,\Theta\right)  =\left(  0,\Theta_{0}\right)  $ with
	$\Theta_{0}\in [-\pi,-\pi/2)$ are just identified as a single point $(0,-\pi)$ . In particular,
	continuous functions $g$ in that topological space take the same values in those subintervals as well as the solution. That new stochastic process does not
	allow to determine at which point the trajectories arrive to $X_2=0.$ 
	
	As discussed above, the natural topological set $S$ and the range $C(S)$ for the Markov generator $\mathcal{L}$ (cf. Definition \ref{setS} and Section \ref{sec:operator L}) which is naturally associated to the stochastic process for the polymer dynamics can be obtained via the identification of the subintervals $\Theta_{0}\in (-\pi/2,0]$ and $\Theta_{0}\in [-\pi,-\pi/2)$ as single pointes $(0,0)$ and $(0,-\pi)$, respectively. However, we observe that the set $S$ via these specific identifications is noncompact because of the point $(X_2,\Theta)=(0,-\pi/2)$; the continuous functions on the set $S$ has a jump discontinuity at $(0,-\pi/2)$. Therefore, as in Section \ref{sec:1dboundaryreg} and Section \ref{sec:2dboundaryreg}, we consider the regularization of the \textit{trapping} boundary condition and define the set $S$ and $X$ without the identifications introduced above. The new regularized boundary conditions are given by differential equations on the boundary, and it will be shown that the solutions (i.e., the boundary conditions) will converge to the original boundary condition with the jump discontinuities by passing it to the limit after we show the existence of Markov semigroups via the classical Hille-Yosida theorem.
	
	\subsubsection{Long-chain asymptotics and the control at infinity}Though the equation that we consider is a linear PDE, the analysis in the paper still involves other technical difficulties besides the regularizations of the jump-discontinuous boundary condition and the Laplace-Beltrami operator introduced above. The difficulties involve the constructions of sub- and super-solutions via self-similar profiles in the form of special functions in Section \ref{sec:holder cont} and Section \ref{sec:longtime asymptotics}. In particular, it is crucial to study the stationary equation and the steady-states in order to obtain the long-chain $(t\to \infty)$ asymptotics and to conclude that the mass are being accumulated at $x_2=0$ and the size of polymers increases linearly in length in the coordinate $x_1\in\mathbb{R}$. One needs to have the well-posedness of the stationary equation and the regularity of the steady-states. Then, one extends the analysis and have a control of $t$-dependent solutions at infinity as $t\to \infty$ and $x_2\to \infty$. This also involves the construction sub- and super-solutions under several types of boundary conditions. Then the maximum principle guarantees the boundedness of the solution.

	\subsection{Outline of the paper} 
	The paper is organized as follows. In Section \ref{sec:dual mass density}, we first study the dual adjoint problem for the total mass density $\rho_1$ in $x_1$ variable, which still shares a similar boundary-value structure on $x_2=0$ with that of the original dual problem for $f$. In Section \ref{sec: 1d wellposedness}, we prove that the dual adjoint problem for the mass density is well-posed by means of an associated elliptic problem, the Hille-Yosida theorem, and the construction of sub- and super-solutions for the comparison principle. In Section \ref{sec:2d section start} and \ref{sec:2d wellposedness}, we introduce the dual adjoint problem for the particle distribution $f(t,x_1,x_2,\theta)$ and obtain the global wellposedness of the dual problem. Finally, in Section \ref{sec: uniqueness}, we prove that a unique weak solution $f$ exists by the duality argument. In Section \ref{sec:hypoellipticity}, we then show that the weak solution that we obtained in the previous sections is indeed locally smooth in the whole domain except for the singular boundary set $x_2=0$ and $\theta=0 \text{ or }-\pi.$ In Section \ref{sec:holder cont}, we prove that the solution is indeed locally H\"older continuous even at the singular boundary. In Section \ref{sec:mass conservation}, we introduce that the total mass is conserved.  In Section \ref{sec:longtime asymptotics}, we study the long-chain asymptotics of the polymer distribution by studying the stationary equation and prove that all the monomer will eventually be trapped at either $(x_2,\theta)=(0,0)$ or $(0,-\pi).$ Lastly, we introduce in Appendix \ref{sec:formalderivation} a formal derivation of the boundary-value problem for the polymer model under the \textit{trapping} boundary condition in a half-plane.
	
	\section{The adjoint problem for the mass density \texorpdfstring{$\rho_1$}{}}\label{sec:dual mass density} In this section, we study a problem in a reduced dimension, which, however, still encodes the same major boundary effect at $x_2=0$. The reduced problem that we construct actually encodes the dynamics of the first-moment-in-$x_1$ variable, which physically means the total mass density distribution at each point $(x_2,\theta)$ with the total length parameter equal to $t.$ We denote this distribution as $\rho_1$ and define it as
	$$
	\rho_1(t,x_2,\theta)\eqdef \int_{\mathbb{R}}f_r(t,x_1,x_2,\theta)dx_1 .
	$$ 
	The corresponding system of our interest throught Section \ref{sec:dual mass density} and Section \ref{sec: 1d wellposedness} is the adjoint problem \eqref{FP modified adjoint eq}- \eqref{modified periodic}.

	\subsection{Asymptotics for large values of \texorpdfstring{$\lowercase{x}_2$}{}}\label{x2infty intro}The analysis of both the 1-dimensional reduced and the 2-dimensional original adjoint problems and their asymptotics crucially depend on the study of the stationary equation of the reduced problem below and on the full understanding of the solutions for large values of $x_2$. More precisely, we will study the stationary equation for $\bar{\varphi}=\bar{\varphi}(x_2,\theta)$
	\begin{equation}\label{stationary1d eq}\begin{split}
			-\sin\theta\partial_{x_2} \bar{\varphi} &= \partial_\theta^2 \bar{\varphi},\ \text{on} \ [0,T]\times (0,\infty)\times (-\pi,\pi),
		\end{split}
	\end{equation}
	with the boundary conditions 
	\begin{equation}
		\label{stationary1d boundary}
		\begin{split}
			\bar{\varphi}(0,\theta)&=\bar{\varphi}(0,-\pi)\ \text{if}\ -\pi\le \theta<-\frac{\pi}{2},\\
			\bar{\varphi}(0,\theta)&=\bar{\varphi}(0,0)\ \text{if}\ -\frac{\pi}{2}<\theta\le 0,
		\end{split}
	\end{equation}and
	\begin{equation}
		\label{stationary periodic}
		\bar{\varphi}(x_2,-\pi)=\bar{\varphi}(x_2,\pi)\text{ and }	\partial_\theta\bar{\varphi}(x_2,-\pi)=\partial_\theta\bar{\varphi}(x_2,\pi).
	\end{equation}
	Our main interest is to prove that the mass does not escape to the infinity $x_2=\infty$ and will eventually be concentrated on $x_2=0$. For this, we will first prove that the solution to the adjoint problem which has the boundary value $\equiv1$ at $x_2=0$ and $\theta\in[-\pi,-\pi/2)\cup(-\pi/2,0]$ will eventually converge to $\equiv1$ for any $x_2>0$ and $\theta \in [-\pi,\pi]$. The proof involves the study on the stationary equation \eqref{stationary1d eq}-\eqref{stationary periodic} and we obtain the bounds for $\varphi$ via the contruction of supersolutions and the maximum principle. This will be discussed in detail in Section \ref{x2 asymptotics section}.
	\subsection{The regularization of the boundary condition}\label{sec:1dboundaryreg} In this section, we first introduce the regularization of the \textit{trapping} boundary condition \eqref{modified adjoint boundarycon} for the analysis of the adjoint problem \eqref{FP modified adjoint eq}-\eqref{modified periodic}. As mentioned in Section \ref{sec:main difficulties}, we want to construct a topological set $S$ that is compact so that the space of continuous functions $C(S)$ on $S$ is a Banach space in the uniform topology. This will allow us to apply the classical Hille-Yosida theorem for the existence of the solutions to the adjoint problem.
	
	Recall the \textit{trapping} boundary conditions \eqref{modified adjoint boundarycon} and \eqref{FP modified adjoint eq}$_2$: 
	\begin{equation}\label{adjoint trapping boundary}\begin{split}
			\partial_t\varphi(t,0,\theta)&=0\ \text{if}\ -\pi\le \theta<-\frac{\pi}{2},\\
			\partial_t \varphi(t,0,\theta)&=0\ \text{if}\ -\frac{\pi}{2}<\theta\le 0,\\
			\varphi(t,0,\theta)&=\varphi(t,0,-\pi)\ \text{if}\ -\pi\le \theta<-\frac{\pi}{2},\ \text{and}\\
			\varphi(t,0,\theta)&=\varphi(t,0,0)\ \text{if}\ -\frac{\pi}{2}<\theta\le 0.
		\end{split}
	\end{equation}Since a solution $\varphi$ which satisfies the conditions above can have a discontinuity at $(x_2,\theta)=(0,-\pi/2)$, we regularize the boundary conditions as follows. 
	For each fixed small $\kappa>0$, we define the regularized boundary condition for $\theta \in [-\pi,0] $ as the solution $\varphi_\kappa(t,0,\theta)$ of
	\begin{equation}
		\label{regularized modified adjoint boundarycon}
		\begin{split} \partial_t\varphi_\kappa(t,0,\theta)&=\frac{1}{\kappa}\bigg(\chi_\kappa(\theta)\varphi_\kappa(t,0,0)+(1-\chi_\kappa(\theta))\varphi_\kappa(t,0,-\pi)-\varphi_\kappa(t,0,\theta)\bigg).\end{split}
	\end{equation}
	Here a smooth function $\chi_\kappa$ is defined as 
	\begin{equation}\label{smoothchi}
		\chi_{\kappa}(\theta)\eqdef \begin{cases}
			&1, \text{ if } -\frac{\pi}{2}+\kappa<\theta \le 0,\\
			&0, \text{ if }-\pi\le \theta < -\frac{\pi}{2}-\kappa,\\
			&\text{smooth and monotone, 
				if }-\frac{\pi}{2}-\kappa\le \theta \le -\frac{\pi}{2}+\kappa.
		\end{cases}
	\end{equation}Then note that if $\theta=0$ or $\theta=-\pi,$ we have
	$\partial_t \varphi_\kappa (t,0,\theta)=0.$ Therefore, $\varphi_\kappa (t,0,\theta)$ is constant for $t\ge 0$ if $\theta=0$ or $-\pi.$
	Solving the ODE \eqref{regularized modified adjoint boundarycon}, we have that for $\theta\in[-\pi,0],$
	\begin{multline}\label{boundary after solving ode}
		\varphi_\kappa(t,0,\theta)=\bigg(\chi_\kappa(\theta)\varphi_\kappa(0,0,0)+(1-\chi_\kappa(\theta))\varphi_\kappa(0,0,-\pi)\bigg)\\+ e^{-\frac{t}{\kappa}}\left(\varphi_\kappa(0,0,\theta)-(\chi_\kappa(\theta)\varphi_\kappa(0,0,0)+(1-\chi_\kappa(\theta))\varphi_\kappa(0,0,-\pi))\right).
	\end{multline}Then observe that we can formally recover the boundary condition \eqref{adjoint trapping boundary} as $\kappa\rightarrow 0$ for $t\ge 0$. In a regularized problem, we let $\varphi_\kappa(0,x_2,\theta)=\varphi_{in}(x_2,\theta)$ which is given.
	
	In Section \ref{sec:dual mass density} and \ref{sec: 1d wellposedness}, we will consider the adjoint problem \eqref{FP modified adjoint eq}, \eqref{modified adjoint initial}, and \eqref{modified periodic} with the regularized boundary condition \eqref{regularized modified adjoint boundarycon}. We will denote the solution as $\varphi_\kappa.$ After we show the existence of such a solution $\varphi_\kappa$ via the Hille-Yosida theorem, we take the limit $\kappa\rightarrow 0$ and recover the solution $\varphi$ to the original adjoint problem \eqref{FP modified adjoint eq}-\eqref{modified periodic}.

	\subsection{A topological set \texorpdfstring{$S$}{}}\label{setS}In order to consider the Hille-Yosida theorem for the existence of a generator of the semigroup, we would like to define a compact domain $S$ and the Banach space $C(S)$. We first define a topologically compact set $S$ in the uniform topology:
	\begin{definition}We define a set $S_0$ as $\{x_2\ge 0\}\times [-\pi,\pi] $ with the additional identification that we identify $(x_2,\pi)$ and $(x_2,-\pi)$ for any $x_2\ge 0$. Then we define the extended space $S=S_0\cup \{\infty\}$ and endow it a natural topology inherited from $[0,\infty)\times [-\pi,\pi]$ complemented by the following set of neighborhoods of the point $\infty$: 
		$$\mathcal{O}_M=\{(x_2,\theta)\in [0,\infty)\times [-\pi,\pi]\ :\ x_2>M\},\ M>0.$$
	\end{definition}
	Note that $S$ is topologically a compact set. A $C^0$ function ${\varphi_\kappa}$ on this set can be identified with the bounded $C^0$ function ${\varphi_\kappa}$ on $S_0$ that satisfies the periodic boundary condition
	\begin{equation}
		\begin{split}
			{\varphi_\kappa}(x_2,-\pi)&={\varphi_\kappa}(x_2,\pi)
			\text{ for } \ x_2\ge 0,
		\end{split}
	\end{equation}and the limit of $$\lim_{x_2\rightarrow \infty} \sup_{\theta\in  [-\pi,\pi]} |{\varphi_\kappa}(x_2,\theta)|$$ exists.
	We denote this set of functions as $C(S)$. We endow the set $C(S)$ a norm 
	\begin{equation}\label{uniform norm}\|{\varphi_\kappa}\|\eqdef \sup_{(x_2,\theta)\in S}|{\varphi_\kappa} (x_2,\theta)|,\end{equation} so that the set $C(S)$ is now
	a Banach space. Also, we define the sets $C^m(S)$ and $C^\alpha(S)$ for a non-negative integer $m$ and $\alpha=(\alpha_1,\alpha_2) \in\mathbb{N}_0^2$ as  
	\begin{equation}\label{setc2}C^m(S)\eqdef \left\{{\varphi_\kappa}\in C(S) : \|{\varphi_\kappa}\|_{C^m}\eqdef \sum_{|\alpha|\le m}\sup_{(x_2,\theta)\in S}|\partial^\alpha_{x_2,\theta}{\varphi_\kappa} (x_2,\theta)|<\infty
		\right\},\end{equation} and\begin{equation}\label{setcalpha}C^\alpha(S)\eqdef \left\{{\varphi_\kappa}\in C(S) : \|{\varphi_\kappa}\|_{C^\alpha}\eqdef \sum_{\beta\le \alpha }\sup_{(x_2,\theta)\in S}|\partial^\beta_{x_2,\theta}{\varphi_\kappa} (x_2,\theta)|<\infty
		\right\},\end{equation} where we used the partial order notation for the multi-indices $\beta\le \alpha$ which means that $\forall i=1,2,\ 0\le\beta_i\le \alpha_i.$
	We also define a set $U$ as $$U\eqdef \{(x_2,\theta)\in S_0 : (x_2,\theta)\neq (0,0)\}.
	$$
	
	\begin{remark}
		Seemingly it is not possible to construct a compact set $S$ that defines a good domain that contains the information about the adjoint boundary conditions \eqref{adjoint trapping boundary} due to the fact that the functions can be discontinuous at $(x_2,\theta)=(0,-\pi/2)$. Then the information about the adjoint boundary conditions \eqref{adjoint trapping boundary} is now in the generator of the semigroup.
	\end{remark}
	\subsection{Definition of the operators \texorpdfstring{$\mathcal{L} $}{} and the domain \texorpdfstring{$ \mathcal{D}(\mathcal{L})$}{}}\label{sec:operator L}
	Given the adjoint problem \eqref{FP modified adjoint eq}-\eqref{modified adjoint initial} with the regularized boundary condition \eqref{regularized modified adjoint boundarycon}, we will rewrite the equation \eqref{FP modified adjoint eq}$_1$ in the following equivalent form: 
	$$\partial_t \varphi_\kappa = \mathcal{L}\varphi_\kappa, \ t\in [0,T],\ \varphi_\kappa(t,\cdot)\in \mathcal{D}(\mathcal{L}),\text{ if } t\ge 0,\ \varphi_\kappa(0,x_2,\theta)=u_\kappa(x_2,\theta),$$ for an operator $\mathcal{L}$ and its domain $\mathcal{D}(\mathcal{L})$. 
	In Section \ref{sec: 1d wellposedness}, we will prove that we can define Markov semigroups $S(t)$ whose corresponding generator is the operator $\mathcal{L}$ via the Hille-Yosida theorem. We will first introduce the definitions of the operator $\mathcal{L}$ and its domain $\mathcal{D}(\mathcal{L})$.
	
	We first define the operator $\mathcal{L}$ as 
	\begin{equation}\label{operator L}\mathcal{L}\eqdef \sin\theta \partial_{x_2}+\partial^2_\theta.\end{equation}
	Depending on various types of the possible boundary dynamics, we can define the operator $\mathcal{L}$ for the different boundary conditions. In this paper, we discuss the case where the polymer that approaches to the boundary $x_2=0$ becomes trapped on the boundary $x_2=0$ as we observe in the derivation in the half-plane (Appendix \ref{sec:formalderivation}). We call the boundary that gives this dynamics the \textit{trapping} boundary: the condition \eqref{boundarycon} for $f$. Note that the boundary conditions for the adjoint problem that we consider in this section is \eqref{adjoint trapping boundary} and its regularized version \eqref{regularized modified adjoint boundarycon}.

	The domain $\mathcal{D}(\mathcal{L})$ of the operator $\mathcal{L}$ with the regularized adjoint boundary condition \eqref{regularized modified adjoint boundarycon} is defined as 
	\begin{multline}\label{domainD}\mathcal{D}(\mathcal{L})=\bigg\{{u_\kappa}, \mathcal{L}{u_\kappa}\in C^1(S):\text{ for }\theta \in[-\pi,0],\\ (\mathcal{L}{u_\kappa})(0,\theta)=\frac{1}{\kappa}\bigg(\chi_\kappa(\theta)u_\kappa(0,0)+(1-\chi_\kappa(\theta))u_\kappa(0,-\pi)-u_\kappa(0,\theta)\bigg)\bigg\}.\end{multline}
	In order to prove that there is a one-to-one correspondence between \textit{Markov} generators on $C(S)$ and \textit{Markov} semigroups on $C(S)$ via the Hille-Yosida theorem, we are interested in proving that the operator $\mathcal{L}$ defined as above with the \textit{trapping} boundary condition satisfies
	$$\mathcal{R}(I-\lambda \mathcal{L})=C(S),$$ for $\lambda>0$. Here, $\mathcal{R}(I-\lambda \mathcal{L})$ means the range of the operator $I-\lambda \mathcal{L}$. Therefore, we have to consider elliptic problems with the form
	\begin{equation}
		\lambda \mathcal{L}{u_\kappa}={u_\kappa}-g \label{A1}%
	\end{equation}
	where $g\in C(S)$ and
	$
	\mathcal{L}=\sin\theta \partial_{x_2}+\partial^2_\theta.
	$
	\subsection{General discussions on the operator \texorpdfstring{$\mathcal{L}$}{} and a stochastic process} In this section, we discuss the relationship between the operator $\mathcal{L}$ and a stochastic process in general. For the general discussion below, we consider the situation of the limiting system $\kappa\rightarrow 0$ without posing the regularization of the boundary in this subsection. 
	
	We basically consider the adjoint problem \eqref{FP modified adjoint eq}-\eqref{modified periodic} in the set $x_2>0$ and $\theta \in [-\pi,\pi]$. Locally near $\theta \approx 0^-$, let us consider $x_2>0,\ \theta\in\mathbb{R}$ at the moment. We need to
	impose that ${u}$ is constant in the whole half-line $x_2=0,\ \theta\leq0.$
	This is due to the fact that the points on the line should be identified as introduced in Section \ref{sec:difficulties Cadleg}, in order to have a well defined \textit{Cadleg} stochastic process.
	Then the following issues arise. The first one is that the maximum principle
	property which is a characteristic of the \textit{Markov} pregenerators fails. Indeed, the same
	type of arguments can be made for other kinds of elliptic/parabolic operators $\mathcal{L}$ including the following ones:%
	\[
	\mathcal{L}=-\frac{\partial}{\partial x}+\frac{\partial^{2}}{\partial y^{2}%
	} ,\  x>0,\  y\in\mathbb{R\ }\text{\ or }y\in\left[  -\pi,\pi\right]
	\text{ with periodic boundary conditions},
	\]%
	\[
	\mathcal{L}=\frac{\partial^{2}}{\partial x^{2}}+\frac{\partial^{2}}{\partial y^{2}%
	} ,\ x>0,\ y\in\mathbb{R\ }\text{\ or }y\in\left[  -\pi,\pi\right]
	\text{ with periodic boundary conditions}.
	\]
	The simplest example yielding the same type of difficulty is the following
	\[
	\mathcal{L}=-\frac{\partial}{\partial x}\text{ \ in }x>0.
	\]
	In this case the trajectories move at constant speed in the direction of
	decreasing $x.$ They reach the boundary of the domain $x_2=0$ in finite time. We
	remark that in order to solve (\ref{A1}) we need to impose a suitable boundary
	condition at $x_2=0,$ and we want to see if it is possible to impose a condition
	${u}={u}\left[  g\right]  $ at the boundary. In the other cases that
	have greater dimensionality we impose that ${u}$ is constant along the
	line $x_2=0.$ The constant would depend also on $g$.
	
	The simplest case corresponds to taking the constant at the boundary
	$\mathcal{L}{u}=0.$ This corresponds to the adjoint \textit{trapping} boundary condition before the regularization of the boundary condition. The
	question is to determine if this is the only possible boundary condition that
	can be imposed. In order to see how we impose the boundary condition
	$\mathcal{L}{u}=0$ we argue as follows. We use the formula of the semigroup
	(\ref{A2}) to see that in the case of \textit{trapping} boundary conditions we have
	$S\left(  t\right)  {u}\left(  x_2=0\right)  =0.$ Then, since the generator
	is the derivative of the semigroup we obtain $\mathcal{L}{u}\left(  x_2=0\right)  =0.$
	This gives the boundary condition for \textit{trapping} boundary conditions. In the
	evolution equation this would be equivalent to $\partial_{t}{u}\left(
	x_2=0\right)  =0.$ Notice that this boundary condition implies in particular
	that $\mathcal{L}\left(  1\right)  =0$ as could be expected for a \textit{Markov} pregenerator.
	Then, the boundary value problem associated to the \textit{trapping} boundary
	condition is (\ref{A1}) with the boundary condition ${u}\left(
	x_2=0\right)  =g\left(  x_2=0\right)  .$

	We can also consider other types of boundary conditions, that would not be related,
	however, to \textit{trapping} boundary conditions. For instance, if we impose that the
	particle arriving to $x_2=0$ has a probability of jumping to an arbitrary orientation $n$, we would
	obtain a boundary condition with the form:%
	\[
	\mathcal{L}{u}\left(  x_2=0\right)  =\int_{0}^{\infty}\mu\left(  y\right)  \left[
	{u}\left(  y\right)  -{u}\left(  x_2=0\right)  \right]  dy.
	\]
	Notice that this gives a different type of boundary condition than before. We
	have, as expected for a \textit{Markovian} pregenerator, the condition $\mathcal{L}\left(
	1\right)  =0.$ Notice that this shows that the constant value at the boundary
	is not uniquely determined. Using the previous boundary condition we obtain
	the boundary condition:%
	\[
	\lambda\int_{0}^{\infty}\mu\left(  y\right)  \left[  {u}\left(  y\right)
	-{u}\left(  x_2=0\right)  \right]  dy={u}\left(  x_2=0\right)  -g\left(
	x_2=0\right).
	\]
	The rationale behind this is that it is possible to have different stochastic
	processes. Notice that the property of \textit{Markov} pregenerator holds. Indeed, in
	the operators above, we observe that, at the minimum of ${u}$, we always have
	$\mathcal{L}{u}\left(  x\right)  \geq0$, and this gives the minimum property.
	
	The same interpretation can also be made using the other operators, including the one
	that appears in the case of polymers. Notice that we can define a domain for
	the operator imposing that the whole function $\mathcal{L}{u}$ is continuous in the
	space under consideration.
	Then we argue that the only \textit{Markov} process with paths having the property that the path $(X_2,\Theta)$ is continuous and that the equations \eqref{A3} hold if $x_2>0$ is the one having the \textit{trapping} boundary conditions.
	\begin{remark}
		Notice that other processes in which there is a large-angle separating the
		particle from the plane $x_2=0$ would result in the equations that
		are the adjoint of the one above.
		We should remark that we can have continuous $X_2\left(  t\right)  $ except for the angle
		$\Theta\left(  t\right)  $ switching from $0$ to $\pi$ and vice versa by keeping
		$X_2\left(  t\right)  =0.$
		There are other \textit{Markov} processes that are not continuous in $X_2$ different from
		the one with \textit{trapping} boundary conditions and they contain jumps in $X_2$.
	\end{remark}
	\subsection{The Hille-Yosida theory }In this subsection, we introduce the Hille-Yosida theory of the semigroups of linear partial differential operators on a general Banach space. We follow the approach of the Hille-Yosida theory that can be found in the book of Liggett \cite{Liggett}. For the Banach space $C(S)$, we first define a \textit{Markov pregenerator} on it as follows:
	\begin{definition}
		A linear operator $\Omega$ on $C(S)$ with the domain $\mathcal{D}(\Omega)$ is said to be a \textit{Markov pregenerator} if it satisfies the following conditions:
		\begin{enumerate}
			\item $1\in\mathcal{D}(\Omega)$ and $\Omega 1=0$.
			\item $\mathcal{D}(\Omega)$ is dense in $C(S)$. 
			\item If ${u_\kappa}\in \mathcal{D}(\Omega),$ $\lambda\ge 0$, and ${u_\kappa}-\lambda\Omega {u_\kappa}=g,$ then 
			$$\min_{\zeta\in S}{u_\kappa}(\zeta)\ge \min_{\zeta\in S}g(\zeta).$$
		\end{enumerate} 
	\end{definition}
	Then we observe that the following proposition holds:
	\begin{proposition}\label{pregenerator proposition}
		The operator $\mathcal{L}$ is a Markov pregenerator.
	\end{proposition}
	\begin{proof}
		We first observe that by the definition of the domain $\mathcal{D}(\mathcal{L})$ from \eqref{domainD}, we have
		$1\in \mathcal{D}(\mathcal{L})$ and $\mathcal{L}1=0.$
		
		Also, we claim that  $\mathcal{D}(\mathcal{L})$ is dense in $C(S)$.  Choose any $\xi\in C^\infty(S)$. Since $C^\infty(S)$ is dense in $C(S)$, it suffices to show that there exists a distribution $\xi_\varepsilon\in \mathcal{D}(\mathcal{L})$ such that 	$$\|\xi-\xi_\varepsilon\|<\varepsilon,$$ for any $\varepsilon>0$ where the uniform norm $\|\cdot\|$ is defined as \eqref{uniform norm}. For each $\xi$ and $\varepsilon>0$, we will construct $\xi_\varepsilon$ from $C^\infty(S)$ such that it also satisfies$$(\mathcal{L}\xi_\varepsilon)(0,\theta)=\frac{1}{\kappa}\bigg(\chi_\kappa(\theta)\xi_\varepsilon(0,0)+(1-\chi_\kappa(\theta))\xi_\varepsilon(0,-\pi)-\xi_\varepsilon(0,\theta)\bigg),\text{ for }\theta \in(-\pi,0).$$ 
		This is equivalent to
		$$\kappa (\sin\theta \partial_{x_2}+\partial_\theta^2)\xi_\varepsilon(0,\theta)=\chi_\kappa(\theta)\xi_\varepsilon(0,0)+(1-\chi_\kappa(\theta))\xi_\varepsilon(0,-\pi)-\xi_\varepsilon(0,\theta),$$ for $\theta \in(-\pi,0)$. 
		For this, we first define a non-negative smooth cutoff function $\lambda(x_2,\theta)\in [0,1]$ such that for an arbitrarily chosen small constant $\delta>0$, \begin{equation}\notag
			\lambda(x_2,\theta) \eqdef \begin{cases}
				&0,\text{ if } x_2 > 2\delta \text{ or }(x_2,\theta)\in [0,2\delta]\times \{[-\pi,-\pi+\delta)\cup (-\delta,\pi]\},\\
				&1,\text{ if } (x_2,\theta)\in [0,\delta]\times [-\pi+2\delta,-2\delta],\\
				&\text{smooth, otherwise.}
			\end{cases}
		\end{equation} Note that $\lambda$ is supported only on $(x_2,\theta)\in [0,2\delta]\times [-\pi+\delta,-\delta].$
		Then define a smooth function $\xi_\varepsilon$ as
		$$ \xi_\varepsilon\eqdef \lambda\bar{\xi}_\varepsilon +(1-\lambda) \xi,$$ where $\bar{\xi}_\varepsilon$ is a smooth solution of the following parabolic problem: for $x_2\ge 0$ and $\theta \in [-\pi+2\delta,-2\delta]$, $\bar{\xi}_\varepsilon$ solves
		\begin{equation}\label{eq for bxe}\kappa (\sin\theta \partial_{x_2}+\partial_\theta^2)\bar{\xi}_\varepsilon(x_2,\theta)=\chi_\kappa(\theta)\xi(0,0)+(1-\chi_\kappa(\theta))\xi(0,-\pi)-\bar{\xi}_\varepsilon(x_2,\theta),\end{equation}
		with the boundary condition \begin{align*}\bar{\xi}_\varepsilon(0,\theta)&=\xi(0,\theta)\text{ for }\theta \in [-\pi+\delta,-\delta],\\
			\bar{\xi}_\varepsilon(x_2,-\pi+2\delta)&=\xi(0,-\pi+2\delta) \text{ for }x_2\ge 0 \\
			\bar{\xi}_\varepsilon(x_2,-2\delta)&=\xi(0,-2\delta)\text{ for }x_2\ge 0 .
		\end{align*}In addition, we assume that $\bar{\xi}_\varepsilon$ is constant on the domain near $\theta =0$ and $\theta =-\pi$:
		\begin{align*}
			\bar{\xi}_\varepsilon(x_2,\theta)&=\xi(0,-\pi+2\delta) \text{ for }x_2\ge 0 \text{ and }\theta \in [-\pi,-\pi+2\delta] \\
			\bar{\xi}_\varepsilon(x_2,\theta)&=\xi(0,-2\delta)\text{ for }x_2\ge 0  \text{ and }\theta \in [-2\delta,0],
		\end{align*} Note that $\sin\theta<0$ for $\theta\in(-\pi,0)$ and the equation \eqref{eq for bxe} for $\bar{\xi}_\varepsilon$ is a parabolic equation with smooth coefficients. Thus the wellposedness and the regularity of a solution can be given by the classical parabolic theory with smooth coefficients. Then for a sufficiently small $\delta>0$, we have
		$ \|\lambda(\xi-\bar{\xi}_\varepsilon)\|<\varepsilon$ by continuity. Also, $\xi_\varepsilon=\lambda\bar{\xi}_\varepsilon +(1-\lambda) \xi$ satisfies the boundary condition $$(\mathcal{L}\xi_\varepsilon)(0,\theta)=\frac{1}{\kappa}\bigg(\chi_\kappa(\theta)\xi_\varepsilon(0,0)+(1-\chi_\kappa(\theta))\xi_\varepsilon(0,-\pi)-\xi_\varepsilon(0,\theta)\bigg)\text{ for }\theta \in(-\pi,0),$$since $\delta$ can be chosen arbitrarily small.   This completes the proof.

		Regarding the last condition for being a pregenerator, it suffices to prove that if ${u_\kappa}\in \mathcal{D}(\mathcal{L})$ and ${u_\kappa}(\eta)=\min_{\zeta\in S} {u_\kappa}(\zeta)$, then $\mathcal{L} {u_\kappa}(\eta)\ge 0,$ by Proposition 2.2 of \cite{Liggett}. If ${u_\kappa}\in \mathcal{D}(\mathcal{L}),$ then ${u_\kappa},\mathcal{L}{u_\kappa}\in C^1(S)$ and $$(\mathcal{L}{u_\kappa})(0,\theta)=\frac{1}{\kappa}\bigg(\chi_\kappa(\theta)u_\kappa(0,0)+(1-\chi_\kappa(\theta))u_\kappa(0,-\pi)-u_\kappa(0,\theta)\bigg),\text{ for }\theta \in[-\pi,0].$$ Also, suppose that ${u_\kappa}(\eta)=\min_{\zeta\in S} {u_\kappa}(\zeta)$. 
		Then observe that
		$$\mathcal{L}{u_\kappa}(\eta)= \sin\theta \partial_{x_2}{u_\kappa}(\eta)+\partial^2_\theta {u_\kappa}(\eta).$$  
		If the minimum $\eta$ is at $x_2=0$ and some $\theta_0 \in[-\pi,0],$ then we have $$\mathcal{L}{u_\kappa}(0,\theta_0)=\frac{1}{\kappa}\bigg(\chi_\kappa(\theta_0)u_\kappa(0,0)+(1-\chi_\kappa(\theta_0))u_\kappa(0,-\pi)-u_\kappa(0,\theta_0)\bigg).$$ Thus, $\mathcal{L}{u_\kappa}(0,\theta_0)\ge 0$, as
		$$\mathcal{L}{u_\kappa}(0,\theta_0)\ge \frac{1}{\kappa}\bigg(\chi_\kappa(\theta_0) \min_{\zeta\in S}u_\kappa+(1-\chi_\kappa(\theta_0)) \min_{\zeta\in S}u_\kappa-u_\kappa(0,\theta_0)\bigg)=0.$$
		On the other hand, if the minimum $\eta$ is on $x_2=0$ with $\theta \in(0,\pi),$ then note that $\sin\theta >0$, $\partial_{x_2}{u_\kappa}\ge 0$, and $\partial^2_\theta \ge 0.$ Thus, $\mathcal{L}{u_\kappa}(\eta)\ge 0.$ In addition, if $\eta$ is in the interior of $S$, then note that $\partial_{x_2}{u_\kappa}(\eta)=0$ and $\partial^2_\theta {u_\kappa}(\eta)\ge 0,$ and hence $\mathcal{L}{u_\kappa}(\eta)\ge 0.$ Finally, if the minimum $\eta$ occurs at $x_2=\infty$, 
		suppose on the contrary that $\mathcal{L}{u_\kappa}(\infty)<0.$ Then there exists a sufficiently small $\varepsilon>0$ such that $\mathcal{L}{u_\kappa}(\infty)+\varepsilon\le 0.$ 
		Thus, by the continuity of $\mathcal{L}{u_\kappa}$, there exists a sufficiently large constant $R>0$ such that if $x_2\ge R$, then $\mathcal{L}{u_\kappa}(x_2,\theta)<0,$ for $\theta \in [-\pi,\pi].$ Also, since ${u_\kappa}(\infty)=\min_{\zeta\in S} {u_\kappa}(\zeta)$, there exists a sufficiently large $x_2^*>2R$ and a small constant $\delta>0$ such that 
		the local minimum of ${u_\kappa}(x_2,\theta)$ on the neighborhood $N_\delta\eqdef [x_2^*-\delta,x_2^*+\delta]\times [-\pi/2-\delta,-\pi/2+\delta]$ occurs at the point $(y,\psi)$ on the upper boundary of $N_\delta$ with $y=x_2^*+\delta$ and $\psi \in [-\pi/2-\delta,-\pi/2+\delta].$ 
		Then note that
		$$
		0>\mathcal{L}{u_\kappa}(y,\psi) = \sin\psi\partial_{x_2}{u_\kappa}(y,\psi)+\partial_\theta^2 {u_\kappa}(y,\psi)  \ge \sin\psi\partial_{x_2}{u_\kappa}(y,\psi)\ge 0,
		$$which leads to the contradiction. This completes the proof.
	\end{proof}
	Indeed, a Markov pregenerator is a \textit{Markov generator} if 
	$$\mathcal{R}(I-\lambda \mathcal{L})=C(S),$$ for $\lambda>0$ small. We define this more precisely as follows:\begin{definition}
		A \textit{Markov} generator is a closed \textit{Markov} pregenerator $\mathcal{L}$ which satisfies 
		$$\mathcal{R}(I-\lambda \mathcal{L})=C(S),$$ for all sufficiently small $\lambda>0$.\end{definition} Equivalently, a closed Markov pregenerator $\mathcal{L}$ is a Markov generator if for any small $\lambda>0$ and for any $g\in C(S)$, there exists a solution $u_\kappa\in \mathcal{D}(\mathcal{L})$ to the following elliptic equation:
	\begin{equation}\label{eq1}
		\lambda \mathcal{L}{u_\kappa}={u_\kappa}-g.
	\end{equation}
	Our main goal in Section \ref{sec:dual mass density} and Section \ref{sec: 1d wellposedness} is to prove the following proposition:
	\begin{proposition}\label{hille1}
		For any $g\in C(S)$ there exists ${u_\kappa}\in \mathcal{D}(\mathcal{L})$ which solves \eqref{eq1}.
	\end{proposition}
	This will be proved in Section \ref{sec: 1d wellposedness}.
	Then, by Proposition 2.8 (a) of \cite{Liggett}, this provides the sufficient condition to the following theorem to be hold:
	\begin{theorem}\label{markovgen} The operator $\mathcal{L}$ is a \textit{Markov} generator.
	\end{theorem}
	Here we also state the classical Hille-Yosida theorem in the text of \cite[Theorem 2.9]{Liggett}:
	\begin{theorem}[Hille-Yosida]
		\label{hilleyosidathm}
		There is a one-to-one correspondence between \textit{Markov} generators on $C(S)$ and \textit{Markov} semigroups on $C(S)$. The correspondence is given by the following:
		$$\mathcal{D}(\mathcal{L})=\left\{{u}\in C(S)\ :\ \lim_{t\rightarrow 0^+}\frac{S(t)u-u}{t}\text{ exists}\right\},$$ and $$\mathcal{L}{u}= \lim_{t\rightarrow 0^+}\frac{S(t){u}-{u}}{t},\ {u}\in \mathcal{D}(\mathcal{L}).$$ If ${u}\in \mathcal{D}(\mathcal{L}),$ then $S(t){u} \in \mathcal{D}(\mathcal{L})$ and 
		$$\frac{d}{dt}(S(t)){u}=\mathcal{L}S(t){u}.$$
	\end{theorem}
	
	We first note that a consequence of the Hille-Yosida theorem is that for $g\in C(S)$ and $\lambda \ge 0$, the solution to $(I-\lambda\mathcal{L}) {u} =g$ is given by 
	$${u}=\int_0^\infty e^{-t}S(\lambda t) g dt,$$
	where $\mathcal{L}$ is called the generator of $S(t)$, and $S(t)$ is the semigroup generated by $\mathcal{L}$. 
	Therefore, in order to prove that $\mathcal{L}$ is a \textit{Markov generator} we need to prove that it is possible to solve the elliptic problem \eqref{eq1}
	where $g\in C(S)$ for any small $\lambda>0$ and
	$
	\mathcal{L}=\sin\theta \partial_{x_2}+\partial^2_\theta.
	$
	
	\subsection{A regularized equation }Throughout the paper, we will consider the regularized version of the equation. More precisely, we discretize the Laplace-Beltrami operator $\partial_\theta^2$ using
	\begin{equation}\label{discretized}Q^\epsilon[{u_\kappa}]\eqdef \frac{2}{\epsilon^2} \int_{-\pi}^\pi d\nu \left({u_\kappa}(x_2,\theta+\epsilon\nu)-{u_\kappa}(x_2,\theta)\right)\zeta(\nu),\end{equation}
	for some small $\epsilon>0$ and a cutoff function $\zeta(\nu)\in C^\infty_c (-\pi,\pi)$ such that 
	\begin{equation}\label{zeta property}\int_{-\pi}^\pi\zeta(\nu)d\nu=1,\ \int_{-\pi}^\pi \nu\zeta(\nu)d\nu=0,\ \int_{-\pi}^\pi \nu^2\zeta(\nu)d\nu=1.\end{equation}
	Then we note that $Q^\epsilon$ is a jump process, and $Q^\epsilon \to \partial^2_\theta$ as $\epsilon\to 0$ at least formally. 
	
	In the next section, we will then discuss the solvability of the elliptic equation \eqref{eq1}. In order for this, we regularize the operator $\partial^2_\theta$ as $Q^\epsilon$ of \eqref{discretized} and consider the regularized equation
	\begin{equation}\label{eq2}
		\lambda \mathcal{L}^\epsilon{u_\kappa^\epsilon}={u_\kappa^\epsilon}-g,
	\end{equation} with the regularized operator
	\begin{equation}
		\mathcal{L}^\epsilon\eqdef \sin\theta \partial_{x_2}+Q^\epsilon.
	\end{equation}
	%
	
	%
	
	\section{The solvability of the adjoint problem for the mass density \texorpdfstring{$\rho_1$}{}}
	\label{sec: 1d wellposedness}
	
	\subsection{A \texorpdfstring{$\lowercase{t}$}{}-dependent problem associated to the elliptic problem}\label{sec:1d intro}In this section, we introduce an associated $t$-dependent problem to the elliptic problem \eqref{eq2}. 
	We first observe that obtaining a solution ${\bar{\psi}^\epsilon_\kappa} ={\bar{\psi}^\epsilon_\kappa} (t,x_2,\theta)$ for the following $t$-dependent problem can guarantee a solution $u_\kappa^\epsilon$ to the regularized equation \eqref{eq2} for some $\lambda>0$:
	\begin{equation}
		\label{eq3}
		\begin{split}\partial_t {\bar{\psi}^\epsilon_\kappa}  &= \mathcal{L}^\epsilon{\bar{\psi}^\epsilon_\kappa} \text{ for }t> 0 \text{ and }  (x_2,\theta)\in S,\\
			{\bar{\psi}^\epsilon_\kappa} (0,x_2,\theta)&=\bar{\psi}_{\kappa,in}(x_2,\theta)=g(x_2,\theta),\\
			{\bar{\psi}^\epsilon_\kappa} (t,x_2,-\pi)&={\bar{\psi}^\epsilon_\kappa} (t,x_2,\pi),\text{ for } t\ge 0,\ x_2\ge 0,\text{ and }\\
			\partial_t {\bar{\psi}^\epsilon_\kappa}  (t,0,\theta)&=\frac{1}{\kappa}\bigg(\chi_\kappa(\theta)\bar{\psi}^\epsilon_\kappa(t,0,0)+(1-\chi_\kappa(\theta))\bar{\psi}^\epsilon_\kappa(t,0,-\pi)-\bar{\psi}^\epsilon_\kappa(t,0,\theta)\bigg),\\ &\qquad\qquad\qquad\qquad\qquad\qquad\text{ for }t\ge 0\text{ and }\theta\in[-\pi,0].
		\end{split}
	\end{equation}
	This is because we can recover a solution $u_\kappa^\epsilon$ to the elliptic problem \eqref{eq2} from a solution ${\bar{\psi}^\epsilon_\kappa} $ to the associated $t$-dependent problem \eqref{eq3} as
	$$u_\kappa^\epsilon(x_2,\theta)=\int_0^\infty e^{-t}{\bar{\psi}^\epsilon_\kappa} (\lambda t , x_2,\theta)\ dt.$$
	Then, by plugging $\theta=0$ and $\theta=-\pi$ to \eqref{eq3}$_3$, we observe that
	$$\partial_t {\bar{\psi}^\epsilon_\kappa}  (t,0,0)=\partial_t {\bar{\psi}^\epsilon_\kappa}  (t,0,-\pi)=0.$$
	\begin{figure}[h] 
		\includegraphics[width=0.6\linewidth]{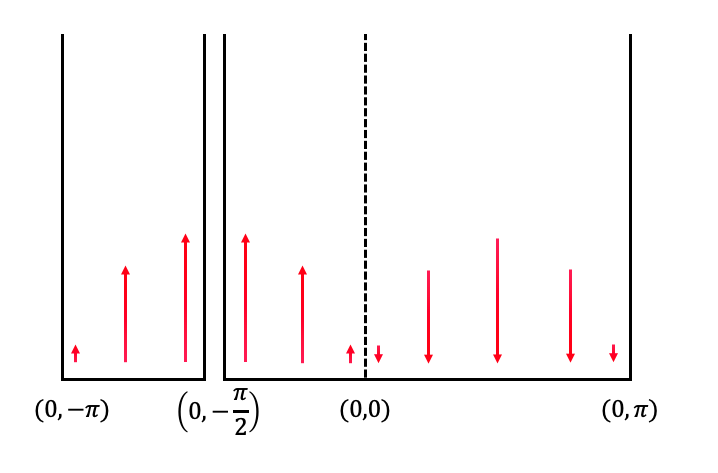}
		\caption{Characteristic flows}
	\end{figure}
	Without loss of generality, we will further assume that $g$ is arbitrarily chosen from a dense subset $C^\infty(S)$ of $C(S)$. 
	
	
	\subsection{Construction of a mild solution}In this section, we construct a mild solution to the $t$-dependent problem \eqref{eq3} considering solutions to the homogeneous equation and the Duhamel principle afterwards.
	\subsubsection{Homogeneous problem}We first consider a solution $W=W(t,x_2,\theta)$ to the following free transport equation:
	\begin{equation}\label{1dhomoeq}
		\begin{split}
			\partial_t W &=\sin\theta \partial_x W ,\text{ for }t>0,\ (x_2,\theta)\in S,\\
			W(0,x_2,\theta)&=g(x_2,\theta),\text{ for }x_2>0,\ \theta\in [-\pi,\pi], \\
			W (t,x_2,-\pi)&=W (t,x_2,\pi),\text{ for } t\ge 0,\ x_2\ge 0,\text{ and },\\
			\partial_t W(t,0,\theta)&=\frac{1}{\kappa}\bigg(\chi_\kappa(\theta)W(0,0,0)+(1-\chi_\kappa(\theta))W(0,0,-\pi)-W(t,0,\theta)\bigg),\\& \text{ for }t\ge 0,\ \theta \in \left[-\pi,0\right].
		\end{split}
	\end{equation}
	Then the solution $W$ is given by
	\begin{equation}
		W(t,x_2,\theta)=\begin{cases}g(x_2+t\sin\theta, \theta)\ &\text{ if } x_2+t\sin\theta>0 \text{ or } \theta\in[0,\pi],\\
			W\left(t+\frac{x_2}{\sin\theta},0,\theta\right) &\text{ if } x_2+t\sin\theta\le 0 \text { and } \theta \in\left(-\pi,0\right),
		\end{cases}
	\end{equation}
	where $W(s,0,\theta)$ 
	for $\theta\in[-\pi,0]$ is given by
	\begin{multline}\label{ws0theta}
		W(s,0,\theta)=\bigg(\chi_\kappa(\theta)g(0,0)+(1-\chi_\kappa(\theta))g(0,-\pi)\bigg)\\+ e^{-\frac{s}{\kappa}}\left(g(0,\theta)-(\chi_\kappa(\theta)g(0,0)+(1-\chi_\kappa(\theta))g(0,-\pi))\right),
	\end{multline} by solving the ODE \eqref{1dhomoeq}$_4$ with \eqref{1dhomoeq}$_2$.
	\subsubsection{Inhomogeneous problem}\label{duhamel}
	Then, by the Duhamel principle, the solution to \eqref{eq3} can be obtained. 
	
	\begin{enumerate}
		\item  If $x_2+t\sin\theta >0\text{ or } \theta\in[0,\pi]$, we have
		$${\bar{\psi}^\epsilon_\kappa} (t,x_2,\theta)=W(t,x_2,\theta)+\int_0^t Q^\epsilon[{\bar{\psi}^\epsilon_\kappa} ](s,U(t-s)(x_2,\theta) )ds,$$ where the semigroup $U(t)$ is given by
		$$U(t)(x_2,\theta)\eqdef (x_2+t\sin\theta,\theta).$$
		Therefore, we have
		$${\bar{\psi}^\epsilon_\kappa} (t,x_2,\theta)=g(x_2+t\sin\theta,\theta)+\int_0^t Q^\epsilon[{\bar{\psi}^\epsilon_\kappa} ](s,x_2+(t-s)\sin\theta,\theta )ds.$$
		
		\item  On the other hand, 
		if  $x_2+t\sin\theta\le 0 \text { and } \theta \in\left(-\pi,0\right),$ we have
		$${\bar{\psi}^\epsilon_\kappa} (t,x_2,\theta)=W\left(t+\frac{x_2}{\sin\theta},0,\theta\right)+\int_{t+\frac{x_2}{\sin\theta}}^t Q^\epsilon[{\bar{\psi}^\epsilon_\kappa} ](s,x_2+(t-s)\sin\theta,\theta )ds,$$ where $W\left(s,0,\theta\right)$ is given by \eqref{ws0theta}.
	\end{enumerate}
	\subsection{Solvability} 
	Define an operator $\mathcal{T}^\epsilon$ as
	$$\mathcal{T}^\epsilon[{\bar{\psi}^\epsilon_\kappa} ]\eqdef W(t,x_2,\theta)+\int_{\max\left\{0,\ t+\frac{x_2}{\sin\theta}\right\}} ^t  \ Q^\epsilon[{\bar{\psi}^\epsilon_\kappa} ](s,x_2+(t-s)\sin\theta,\theta)ds,$$ where 
	\begin{equation}
		W(t,x_2,\theta)= \begin{cases}
			& g(x_2+t\sin\theta,\theta),\text{ if }x_2+t\sin\theta >0\text{ or }\theta\in[0,\pi],\\
			& W\left(t+\frac{x_2}{\sin\theta},0,\theta\right), \text{ if }x_2+t\sin\theta\le 0\text { and }\theta\in (-\pi,0).\end{cases}
	\end{equation}
	Define a set $E(S)$ as
	$$E(S)=\left\{\psi \in C(S) \ : \ \|\psi(t,\cdot )\|_{C^{(1,2)}_{x_2,\theta}}\le C\|g\|_{C^{2}}\text{ for }t\in [0,T]\right\},$$ for some fixed $C>\frac{27}{2}.$
	Then, we obtain the following lemma.
	\begin{lemma}[Contraction mapping]\label{contraction1}If ${\bar{\psi}^\epsilon_\kappa} \in E(S)$, then we have $\mathcal{T}^\epsilon[{\bar{\psi}^\epsilon_\kappa} ]\in E(S).$ Moreover,
		$\mathcal{T}^\epsilon$ is a contraction in $E(S)$, for $T=T_1$ sufficiently small.
	\end{lemma}
	\begin{proof}
		Here we show that  $\mathcal{T}^\epsilon$  maps  $E(S)$  to itself and is a contraction, if  $T_1=T_1(\epsilon) >0$  is sufficiently small. 
		Note that for any $\psi\in E(S)$ and any multi-index $\alpha \le (1,2)$ in the partial order, we have
		\begin{equation*}
			\big|Q^\epsilon[\partial^\alpha\psi](t,x_2,\theta)\big|  \leq  
			\frac{4}{\epsilon^2}  \|\psi(t)\|_{ C^{(1,2)}_{x_2,\theta}}\! \int_{(-\pi,\pi)}\!\zeta(\nu) d\nu  \leq 
			\frac{4C}{\epsilon^2} \|g\|_{C^{2}},
		\end{equation*}
		recalling the expression of  $Q^\epsilon[\psi]$. Now we recall the mild form of solutions obtained in Section \ref{duhamel} and consider taking the derivatives for $\alpha\le (1,2).$ 
		
		If $x_2+t\sin\theta>0$ or $\theta\in [0,\pi]$, we observe that
		$$\partial_{x_2}{\bar{\psi}^\epsilon_\kappa} (t,x_2,\theta)=\partial_{x_2}g(x_2+t\sin\theta,\theta)+ Q^\epsilon[{\partial_{x_2}\bar{\psi}^\epsilon_\kappa} ](s,U(t-s)(x_2,\theta) )ds,$$ 
		\begin{multline*}\partial_{\theta}\bar{\psi}^\epsilon_\kappa (t,x_2,\theta)=t\cos\theta\partial_{x_2}g(x_2+t\sin\theta,\theta)+\partial_{\theta}g(x_2+t\sin\theta,\theta)\\+\int_0^t Q^\epsilon[{(t-s)\cos\theta\partial_{x_2}\bar{\psi}^\epsilon_\kappa+\partial_{\theta}\bar{\psi}^\epsilon_\kappa} ](s,x_2+(t-s)\sin\theta,\theta)ds,\end{multline*}
		and
		\begin{multline*}\partial^2_{\theta}\bar{\psi}^\epsilon_\kappa (t,x_2,\theta)=-t\sin\theta\partial_{x_2}g(x_2+t\sin\theta,\theta)+t^2\cos^2\theta\partial^2_{x_2}g(x_2+t\sin\theta,\theta)\\+2t\cos\theta\partial^{(1,1)}_{x_2,\theta}g(x_2+t\sin\theta,\theta)+\partial^2_{\theta}g(x_2+t\sin\theta,\theta)\\
			+\int_0^t Q^\epsilon\bigg[\bigg(-(t-s)\sin\theta\partial_{x_2}+(t-s)^2\cos\theta^2\partial_{x_2}^2\\+2(t-s)\cos\theta\partial^{(1,1)}_{x_2,\theta}+\partial^2_\theta\bigg)\bar{\psi}^\epsilon_\kappa\bigg](s,x_2+(t-s)\sin\theta,\theta)ds.\end{multline*}On the other hand, if $x_2+t\sin\theta\le 0$ and $\theta \in (-\pi,0)$, we have 	 $$\left|\partial^\alpha W\left(t+\frac{x_2}{\sin\theta},0,\theta\right)\right|\le 2\|g\|_{C^{(1,2)}_{x_2,\theta}},$$ for $\alpha\le (1,2)$ by \eqref{ws0theta}.
		Therefore, for any  $t\in[0,T_1]$ and $\alpha\le (1,2)$, we have
		\begin{align*}
			\partial^\alpha\mathcal{T}^\epsilon[\psi](t)
			&  \leq  \bigg(t+\frac{3}{2}\bigg)^2\|g\|_{C^{2}} +\bigg(\frac{t^3}{3}+\frac{3t^2}{2}+t\bigg) \big|Q^\epsilon[\partial^\alpha\psi]\big| \\[2pt]
			&  \leq\left(  \bigg(T_1+\frac{3}{2}\bigg)^2+	\frac{4C}{\epsilon^2}\bigg(\frac{T_1^3}{3}+\frac{3T_1^2}{2}+T_1\bigg) \right)\|g\|_{C^{2}}    \leq  \frac{C}{6}\|g\|_{C^{2}},
		\end{align*}
		provided  $T_1$  is chosen sufficiently small such that $ \bigg(\frac{T_1^3}{3}+\frac{3T_1^2}{2}+T_1\bigg) \frac{4}{\epsilon^2}\ll1$ and $$\bigg(T_1+\frac{3}{2}\bigg)^2+	\frac{4C}{\epsilon^2}\bigg(\frac{T_1^3}{3}+\frac{3T_1^2}{2}+T_1\bigg)\le \frac{C}{6} ,$$ for given $C>\frac{27}{2}.$ 
		Then we have
		\begin{equation*}
			\sup_{t\in [0,T_1]}\|\mathcal{T}^\epsilon[\psi](t)\|_{ C^{(1,2)}_{x_2,\theta}} \leq  C\|g\|_{C^{2}} .
		\end{equation*}
		All these above imply that  $\mathcal{T}^\epsilon[\psi]\in E(S)$  and hence  $\mathcal{T}^\epsilon$  maps  $E(S)$  into  $E(S)$.
		In addition, for any two $\psi_1, \ \psi_2 \in E(S)$  and  $t\in[0,T_1]$, similar arguments yield that
		\begin{multline*}
			\|\mathcal{T}^\epsilon[\psi_1](t) - \mathcal{T}^\epsilon[\psi_2](t)\|_{C^{(1,2)}_{x_2,\theta}}
			=  \|\mathcal{T}^\epsilon[\psi_1\!-\psi_2](t)\|_{C^{(1,2)}_{x_2,\theta}} \\[5pt]
			\leq  \bigg(\frac{T_1^3}{3}+\frac{3T_1^2}{2}+T_1\bigg)\sup_{t\in [0,T_1]}\sum_{\alpha\le (1,2)}\big\|\partial^\alpha Q^\epsilon[\psi_1\!-\psi_2]\big\|_{L^\infty} \\[1pt]
			\leq \bigg(\frac{T_1^3}{3}+\frac{3T_1^2}{2}+T_1\bigg) \frac{4}{\epsilon^2}  \sup_{t\in [0,T_1]}\|\psi_1(t)-\psi_2(t)\|_{C^{(1,2)}_{x_2,\theta}}.
		\end{multline*}
		Since  $\bigg(\frac{T_1^3}{3}+\frac{3T_1^2}{2}+T_1\bigg) \frac{4}{\epsilon^2}  <1$ with our choice of  $T_1$  above, we conclude that  $\mathcal{T}^\epsilon$  is a contraction.
	\end{proof}
	\begin{corollary}\label{contractioncor}There exists a unique global mild solution ${\bar{\psi}^\epsilon_\kappa}$ in $E(S)$ to \eqref{eq3}.
	\end{corollary}\begin{proof}
		Therefore, by the Schauder-type fixed-point theorem, there exists a unique mild solution in $E(S)$  on the interval  $[0,T_1]$  for such  $T_1=T_1(\epsilon)$  that$$\bigg(\frac{T_1^3}{3}+\frac{3T_1^2}{2}+T_1\bigg) \frac{4}{\epsilon^2}  <1$$ is satisfied.
		For the global existence, note that $T_1=T_1(\epsilon)$ does not depend on the initial data $g$. Then, by a continuation argument, we can extend the maximal interval of the existence to an arbitrary $T>0$ independent of  $\epsilon$.
	\end{proof}
	
	Therefore, we obtain the global wellposedness in $E(S)$ for the regularized problem \eqref{eq3}. 
	
	\subsection{Uniform-in-\texorpdfstring{$\epsilon$}{} estimates}In this subsection, we provide uniform-in-$\epsilon$ estimates for the solutions to \eqref{eq3} as follows. 
	\begin{lemma}[maximum principle]\label{maxprinciple}
		Define 
	\begin{equation}\label{M}
		\mathcal{M} h\eqdef \left(\partial_t-\sin\theta \partial_{x_2}-Q^\epsilon \right)h, 
	\end{equation} and let a $2\pi$-periodic(-in-$\theta$) function $h\in C^{1,1,2}_{t,x_2,\theta}$ solve $\mathcal{M} h\le 0$. Then $h$ attains its maximum only when $\{t=0\}$ or when $\{x_2=0\}$ and $\theta\in[-\pi,0]$. Hence, a solution ${\bar{\psi}^\epsilon_\kappa} $ to \eqref{eq3} satisfies 
		$$\|{\bar{\psi}^\epsilon_\kappa} \|_{L^\infty([0,T]\times S)}
		\le 2\|g\|_{L^\infty(S)}.$$\end{lemma}
			\begin{proof}
	We first assume that a $2\pi$-periodic-in-$\theta$ function ${\bar{\psi}^\epsilon_\kappa}$ satisfies $\mathcal{M}{\bar{\psi}^\epsilon_\kappa}<0$. \begin{itemize}
			\item If ${\bar{\psi}^\epsilon_\kappa}$ also attains
			its maximum either at an interior point $(t,x_2,\theta) \in \{(0,T) \times(0,\infty)\times
			(-\pi,\pi)\}\setminus \bar{R_\epsilon}$. Then we have $\partial_{t}
			{\bar{\psi}^\epsilon_\kappa} =\partial_{x_2} {\bar{\psi}^\epsilon_\kappa} =0$ while $Q^\epsilon {\bar{\psi}^\epsilon_\kappa} \le0$ at the maximum. Thus $\mathcal{M} {\bar{\psi}^\epsilon_\kappa}(t,x_2,\theta) \ge0$ and this contradicts the assumption.
			
			\item If the maximum occurs on $(T,x_2,\theta)$ with $(x_2,\theta) \in (0,\infty)\times
			(-\pi,\pi)$,  then $\partial_{t}
			{\bar{\psi}^\epsilon_\kappa}\ge 0$ and $ \partial_{x_2} {\bar{\psi}^\epsilon_\kappa} =0$ while $Q^\epsilon {\bar{\psi}^\epsilon_\kappa} \le0$ at the maximum. \newline Thus $\mathcal{M} {\bar{\psi}^\epsilon_\kappa}(t,x_2,\theta) \ge0$ and this contradicts the assumption.
			
			\item	If the maximum occurs on $(t,0,\theta)$ with $t\in (0,T)$ and $\theta \in [0,\pi)$, then $\partial_{t}
			{\bar{\psi}^\epsilon_\kappa}= 0$ and $-\sin\theta \partial_{x_2} {\bar{\psi}^\epsilon_\kappa} \ge 0$ while $Q^\epsilon {\bar{\psi}^\epsilon_\kappa} \le0$ at the maximum. Thus $\mathcal{M} {\bar{\psi}^\epsilon_\kappa}(t,x_2,\theta) \ge0$ and this contradicts the assumption.
			
			\item	If the maximum occurs on $(t,x_2,\pm \pi)$ 
			with $t\in (0,T)$ and $x_2\in (0,\infty)$, then $\partial_{t}
			{\bar{\psi}^\epsilon_\kappa}=\partial_{x_2} {\bar{\psi}^\epsilon_\kappa} =0$ while $Q^\epsilon {\bar{\psi}^\epsilon_\kappa} \le0$ at the maximum. Thus $\mathcal{M} {\bar{\psi}^\epsilon_\kappa}(t,x_2,\theta) \ge0$ and this contradicts the assumption.
			\item  Finally, if the  maximum occurs at $x_2=\infty$, there exists a sufficiently small $\varepsilon>0$ such that $\mathcal{M}{\bar{\psi}^\epsilon_\kappa}(\infty)+\varepsilon\le 0.$ 
			Then by the continuity of $\mathcal{M}{\bar{\psi}^\epsilon_\kappa}$,  there exists a sufficiently large constant $R>0$ such that if $x_2\ge R$, then $\mathcal{M}{\bar{\psi}^\epsilon_\kappa}(t,x_2,\theta)<0,$ for any $(t,\theta) \in [0,T]\times  [-\pi,\pi].$ Also, since ${\bar{\psi}^\epsilon_\kappa}(\infty)=\max_{\zeta\in S} {\bar{\psi}^\epsilon_\kappa}(\zeta)$, there exist a sufficiently large $x_2^*>2R$ and a small constant $\delta>0$ such that 
			the local maximum of ${\bar{\psi}^\epsilon_\kappa}(t,x_2,\theta)$ on the neighborhood $N_\delta\eqdef (T/2-\delta,T/2+\delta)\times (x_2^*-\delta,x_2^*+\delta)\times (-\pi/2-\delta,-\pi/2+\delta)$ occurs at the point $(s,y,\psi)$ on the upper-in-$x_2$ boundary of $N_\delta$ with  $y=x_2^*+\delta$, $s=(T/2-\delta,T/2+\delta)$ and $\psi \in (-\pi/2-\delta,-\pi/2+\delta).$ 
			Then note that
			\begin{multline*}\qquad\qquad	0>\mathcal{M}{\bar{\psi}^\epsilon_\kappa}(s,y,\psi) = \partial_t {\bar{\psi}^\epsilon_\kappa}-\sin\psi\partial_{x_2}{\bar{\psi}^\epsilon_\kappa}(s,y,\psi)-\partial_\theta^2 {\bar{\psi}^\epsilon_\kappa}(s,y,\psi) 
				\\
				 \ge -\sin\psi\partial_{x_2}{\bar{\psi}^\epsilon_\kappa}(y,\psi)\ge 0,
			\end{multline*}which leads to the contradiction. 
		\end{itemize}	
		Therefore, if some function ${\bar{\psi}^\epsilon_\kappa}$ satisfies $\mathcal{M}{\bar{\psi}^\epsilon_\kappa}<0$, then the maximum occurs only when $t=0$ or when $x_2=0$ and $\theta\in [-\pi, 0].$
		
		Now, if $\mathcal{M} {\bar{\psi}^\epsilon_\kappa}$ is just $ \le0$ then we define a new function  $\bar{\psi}^{\epsilon, k}_\kappa := {\bar{\psi}^\epsilon_\kappa}-kt$ for some
		$k>0$ so that $\mathcal{M} \psi^{\epsilon, k} <0$. Then we have
		$$
		\sup_{(t,x_2,\theta)\in[0,T] \times S} \bar{\psi}^{\epsilon,k}_\kappa(t,x_2,\theta) =
		\sup_{\{t=0\}\ \text{or} \ \{x_2=0 \text{ and }\theta\in [-\pi,0]\} } \bar{\psi}^{\epsilon,k}_\kappa(t,x_2,\theta).$$
		Taking $k \rightarrow0$, we obtain
		$$
		\sup_{(t,x_2,\theta)\in[0,T] \times S} \bar{\psi}^{\epsilon}_\kappa(t,x_2,\theta) =
		\sup_{\{t=0\}\ \text{or} \ \{x_2=0 \text{ and }\theta\in [-\pi,0]\} } \bar{\psi}^{\epsilon}_\kappa(t,x_2,\theta).$$
		Recall \eqref{eq3}$_2$ that $\bar{\psi}^{\epsilon}_\kappa(0,x_2,\theta)=g(x_2,\theta).$ Also, by solving \eqref{eq3}$_3$ and applying \eqref{eq3}$_2$, we have
		\begin{multline*}	\bar{\psi}^{\epsilon}_\kappa(t,0,\theta)=\bigg(\chi_\kappa(\theta)g(0,0)+(1-\chi_\kappa(\theta))g(0,-\pi)\bigg)\\+ e^{-\frac{t}{\kappa}}\left(g(0,\theta)-(\chi_\kappa(\theta)g(0,0)+(1-\chi_\kappa(\theta))g(0,-\pi))\right),\end{multline*} if $\theta \in [-\pi,0].$ Thus, we have 
		$|\bar{\psi}^{\epsilon}_\kappa(t,0,\theta)|\le 2\|g\|_{L^\infty(S)},$ if $\theta\in [-\pi,0].$ This completes the proof.
	\end{proof}
	
	\begin{remark}As a corollary, we obtain the uniqueness of the solution via the maximum principle.
	\end{remark}
	Regarding the $x_2$ derivatives, we obtain a similar estimate:
	\begin{corollary}[Estimates for $x_2$ derivatives]\label{x2 derivative estimates}
		For any $m\ge 0$, the solution ${\bar{\psi}^\epsilon_\kappa} $ to \eqref{eq3} satisfies 
		$$\|{\partial^m_{x_2}\bar{\psi}^\epsilon_\kappa} \|_{L^\infty([0,T]\times S)}
		\le 2\|\partial^m_{x_2}g\|_{L^\infty(S)}.$$\end{corollary}\begin{proof}
		By taking $\partial^m_{x_2}$ to the equation \eqref{eq3}, we observe that $\partial^m_{x_2}\bar{\psi}^\epsilon_{\kappa}$ satisfies the same equation with revised initial and boundary conditions:
		\begin{equation}
			\begin{split}
				\partial^m_{x_2}	{\bar{\psi}^\epsilon_\kappa} (0,x_2,\theta)&=\partial^m_{x_2}g(x_2,\theta),\text{ and }\\
				\partial_t(\partial^m_{x_2} {\bar{\psi}^\epsilon_\kappa} ) (t,0,\theta)&=\frac{1}{\kappa}\bigg(\chi_\kappa(\theta)(\partial^m_{x_2} {\bar{\psi}^\epsilon_\kappa} )(t,0,0)+(1-\chi_\kappa(\theta))(\partial^m_{x_2} {\bar{\psi}^\epsilon_\kappa} )(t,0,-\pi)\\&\qquad-(\partial^m_{x_2} {\bar{\psi}^\epsilon_\kappa} )(t,0,\theta)\bigg),\text{ for }t\ge 0\text{ and }\theta\in[-\pi,0].
			\end{split}
		\end{equation}
		Then the corollary follows by Lemma \ref{maxprinciple}.
	\end{proof}
	Now we consider the derivatives with respect to $\theta$. If we take one $\theta$-derivative to the equation \eqref{eq3}, we obtain the following inhomogeneous equation
	$$\mathcal{M}(\partial_\theta \bar{\psi}^\epsilon_\kappa)=\cos\theta \partial_{x_2}\bar{\psi}^\epsilon_\kappa,$$ with the initial-boundary conditions: 
	\begin{equation}
		\notag
		\begin{split}
			\partial_\theta	{\bar{\psi}^\epsilon_\kappa} (0,x_2,\theta)&=\partial_\theta g(x_2,\theta),\text{ and }\\
			\partial_t(\partial_\theta {\bar{\psi}^\epsilon_\kappa} ) (t,0,\theta)&=\frac{1}{\kappa}\bigg(\partial_\theta\chi_\kappa(\theta) {\bar{\psi}^\epsilon_\kappa} (t,0,0)-\partial_\theta\chi_\kappa(\theta) {\bar{\psi}^\epsilon_\kappa} (t,0,-\pi)\\&\qquad-(\partial_\theta {\bar{\psi}^\epsilon_\kappa} )(t,0,\theta)\bigg),\text{ for }t\ge 0\text{ and }\theta\in[-\pi,0].
		\end{split}
	\end{equation}
	Note that the solution to the homogeneous problem with the same initial-boundary conditions satisfies a similar uniform-in-$\epsilon$ bound as in Lemma \ref{maxprinciple}.  
	Then, by Lemma \ref{maxprinciple} and the Duhamel principle, we obtain that for $t\ge 0,$
	\begin{multline*}
		\|\partial_\theta \bar{\psi}^\epsilon_\kappa(t)\|_{L^\infty}\le C_\kappa\| g\|_{C^{(0,1)}_{x_2,\theta}(S)}+t \left(2\| g\|_{C^{(0,1)}_{x_2,\theta}(S)}+ \|\partial_{x_2}\bar{\psi}^\epsilon_\kappa\|_{L^\infty([0,t]\times S)}\right)\\ \le (C_\kappa+2t)\| g\|_{C^{(0,1)}_{x_2,\theta}(S)}+ t\|\partial_{x_2}\bar{\psi}^\epsilon_\kappa\|_{L^\infty([0,t]\times S)}\\
		\le (C_\kappa+2t)\| g\|_{C^{(0,1)}_{x_2,\theta}(S)}+ 2t\|\partial_{x_2}g\|_{L^\infty( S)},
	\end{multline*}by Corollary \ref{x2 derivative estimates} and \eqref{ws0theta}. This provides the uniform-in-$\epsilon$ upper bound for $\partial_\theta \bar{\psi}^\epsilon_\kappa$.
	
	By the same proof, we prove the following bound for $\partial_{x_2}\partial_\theta \bar{\psi}^\epsilon_\kappa$ as it satisfies the same equation for $\partial_\theta \bar{\psi}^\epsilon_\kappa$ with an additional $\partial_{x_2}$ derivatives applied to the initial-boundary conditions:
	\begin{equation}\label{1dx2theta upper}
		\|\partial_{x_2}\partial_\theta \bar{\psi}^\epsilon_\kappa(t)\|_{L^\infty(S)}
		\le (C_\kappa+2t)\| g\|_{C^{(1,1)}_{x_2,\theta}(S)}+ 2t\|\partial^2_{x_2}g\|_{L^\infty( S)}.
	\end{equation}for any $t\ge 0$. Finally, we take the second derivative $\partial_\theta^2$ to \eqref{eq3} and obtain the inhomogeneous equation
	$$\mathcal{M}(\partial^2_\theta \bar{\psi}^\epsilon_\kappa)=2\cos\theta \partial_{x_2}\partial_\theta \bar{\psi}^\epsilon_\kappa -\sin \theta \partial_{x_2} \bar{\psi}^\epsilon_\kappa,$$ with the initial-boundary conditions
	\begin{equation}
		\notag
		\begin{split}
			\partial_\theta^2	{\bar{\psi}^\epsilon_\kappa} (0,x_2,\theta)&=\partial_\theta^2 g(x_2,\theta),\text{ and }\\
			\partial_t(\partial^2_\theta {\bar{\psi}^\epsilon_\kappa} ) (t,0,\theta)&=\frac{1}{\kappa}\bigg(\partial^2_\theta\chi_\kappa(\theta) {\bar{\psi}^\epsilon_\kappa} (t,0,0)-\partial^2_\theta\chi_\kappa(\theta) {\bar{\psi}^\epsilon_\kappa} (t,0,-\pi)\\&\qquad-(\partial_\theta^2 {\bar{\psi}^\epsilon_\kappa} )(t,0,\theta)\bigg),\text{ for }t\ge 0\text{ and }\theta\in[-\pi,0].
		\end{split}
	\end{equation} By the Duhamel principle with the different inhomogeneity, we obtain that for any $t\ge0$,
	\begin{multline*}
		\|\partial^2_\theta \bar{\psi}^\epsilon_\kappa(t)\|_{L^\infty(S)}
		\le (C'_\kappa+2t)\| g\|_{C^{(0,2)}_{x_2,\theta}(S)}\\+ t\left(2\left((C_\kappa+2t)\| g\|_{C^{(1,1)}_{x_2,\theta}(S)}+ 2t\|\partial^2_{x_2}g\|_{L^\infty( S)}\right)+\|\partial_{x_2}g\|_{L^\infty( S)}\right)\le C_{\kappa,t} \|g\|_{C^2(S)},
	\end{multline*}by \eqref{ws0theta}, \eqref{1dx2theta upper}, and Corollary \ref{x2 derivative estimates} where $C_{\kappa,t}>0$ is a second-order polynomial in $t$ for each fixed $\kappa>0$.  Thus, we obtain the following proposition:
	\begin{proposition}[Uniform-in-$\epsilon$ estimates for the derivatives]\label{1d derivative estimates}
		For $t\ge 0$ and $\kappa>0$, the solution ${\bar{\psi}^\epsilon_\kappa} $ to \eqref{eq3} satisfies 
		$$	\left\| \bar{\psi}^\epsilon_\kappa(t)\right\|_{C^{(1,2)}_{x_2,\theta}( S)}\le C_{\kappa,t} \|g\|_{C^2(S)},$$where $C_t>0$ is a second-order polynomial in $t$ for each fixed $\kappa>0$.\end{proposition}

	\subsection{Passing to the limit \texorpdfstring{$\epsilon\to 0$}{}}
\label{sec:limit epsilon to zero}	
	Now we recover the solution to \eqref{eq2} via the definition
	$$u_\kappa^\epsilon({x_2,\theta})=\int_0^\infty e^{-t}{\bar{\psi}^\epsilon_\kappa} (\lambda t , {x_2,\theta})\ dt,$$ which is well-defined for any $\epsilon>0$.
	Then we pass to the limit $\epsilon \rightarrow 0$ and obtain the existence of a unique global solution to \eqref{eq1}. We obtain the limit of the approximating sequence $\left\{u_\kappa^\epsilon\right\}$ as a candidate for a solution by the compactness (Banach-Alaoglu theorem), which is ensured by the \emph{uniform} estimates of the approximate solutions established in the previous subsections. From the uniform estimate (Proposition~\ref{1d derivative estimates}) and from taking the limit as  $\epsilon \rightarrow 0$, we obtain that a sequence of  $\left\{u_\kappa^\epsilon\right\}$ converges to ${u_\kappa}$  in $  C^{(1,2)}_{x_2,\theta} (S )$. 
	Again it follows from Proposition~\ref{1d derivative estimates} that ${u_\kappa}$ also satisfies the bound
	\begin{equation} \label{L^infty-bound}
		\left\| u_\kappa\right\|_{C^{(1,2)}_{x_2,\theta}( S)}\le C_\kappa(1+\lambda^2) \|g\|_{C^2(S)},
	\end{equation}for some constant $C_\kappa>0$.
	\subsection{Hille-Yosida theorem and the global wellposedness of the adjoint problem}Now we are ready to prove the following proposition on the well-posedness of the 1-dimensional adjoint problem \eqref{FP modified adjoint eq}- \eqref{modified periodic}.
	\begin{proposition}[Well-posedness of the 1-dimensional adjoint problem]
	  \label{Thm.sol.adjoint0} 
	    Suppose that $\varphi_{in}=\varphi_{in}(x_2,\theta)\in C([0,\infty)\times [-\pi,\pi))$ satisfies \eqref{forward adjoint initial0}. Then there exists a unique solution $\varphi \in C(S)$ which solves the 1-dimensional adjoint problem \eqref{FP modified adjoint eq}- \eqref{modified periodic}.
	\end{proposition} 
\begin{proof}	 We first note that we obtain Proposition \ref{hille1}, which states that there exists a unique ${u_\kappa}\in C(S)$ which solves \eqref{eq1} for any $g\in C(S)$ by the standard density argument. Then this provides the sufficient condition for Theorem \ref{markovgen}, which states that $\mathcal{L}$ is the \textit{Markov} generator. Now, via the Hille-Yosida theorem (Theorem \ref{hilleyosidathm}), we obtain the corresponding \textit{Markov} semigroup $S_\kappa(t)$ and the solution $\varphi_\kappa(t,\cdot)=S_\kappa(t)u_\kappa(\cdot)$ to the adjoint equation \eqref{FP modified adjoint eq}, \eqref{modified adjoint initial}, and \eqref{modified periodic} with the regularized boundary condition \eqref{regularized modified adjoint boundarycon} for a small $\kappa>0$. We note that the semigroup $S_\kappa$ is continuous and $u_\kappa$ is differentiable. More importantly, we have the maximum principle on $\varphi_\kappa$ as in Lemma \ref{maxprinciple} and we have the uniform-in-$\kappa$ $L^\infty$ estimate on $\varphi_\kappa$; namely, we obtain $$\|\varphi_\kappa\|_{L^\infty([0,T]\times S)}\le 2\max\{1, \|\varphi_{in}\|_{L^\infty(S)}\},$$ since the maximum occurs only at $\{t=0\}$ or $\{x_2=0\}\cap \{\theta\le0\}$ and the boundary values at $x_2=0$ and $\theta\le0$ is bounded by 1 from above. 
	This allows us to pass to the limit $\kappa \to 0$ and obtain a weak solution $\varphi$ as a weak-$\ast$ limit of $\varphi_\kappa(t,x_2,\theta). $   
	Note that the limiting system for its weak-$\ast$ limit $\varphi$ as $\kappa\rightarrow 0$ is the same system as that of $\varphi_\kappa$ except for the difference in the boundary conditions \eqref{modified adjoint boundarycon} and \eqref{regularized modified adjoint boundarycon}. In order to study the limit of the boundary condition for $\varphi_\kappa$, we solve the ODE \eqref{regularized modified adjoint boundarycon} and observe that the boundary condition gives 
	\begin{multline*}
		\varphi_\kappa(t,0,\theta)=\bigg(\chi_\kappa(\theta)\varphi_\kappa(0,0,0)+(1-\chi_\kappa(\theta))\varphi_\kappa(0,0,-\pi)\bigg)\\+ e^{-\frac{t}{\kappa}}\left(\varphi_\kappa(0,0,\theta)-(\chi_\kappa(\theta)\varphi_\kappa(0,0,0)+(1-\chi_\kappa(\theta))\varphi_\kappa(0,0,-\pi))\right),
	\end{multline*}for $\theta\in[-\pi,0].$ Then we also have the bound for the values at the boundary
	$$\|\varphi_\kappa(t,0,\theta)\|_{L^\infty([0,T]\times [-\pi,0])} \le 2 \|\varphi_{in} (0,\theta)\|_{L^\infty([-\pi,\pi])},$$ where $\varphi_{in}$ is the initial profile of \eqref{modified adjoint initial}.
	Moreover, for $t\ge 0$ and a sequence of small $\delta_\kappa>0$ such that $\delta_\kappa\gg \kappa$ and $\delta_\kappa \to 0$ as $\kappa \to 0,$ we have the uniform-in-$\kappa$ limit for $\theta \in [-\pi,-\pi/2-\delta_\kappa]\cup[-\pi/2+\delta_\kappa,0]$ as $\kappa\to 0$
	\begin{multline*}
		\left\|\varphi_\kappa(t,0,\theta)-\bigg(\chi_\kappa(\theta)\varphi_\kappa(0,0,0)+(1-\chi_\kappa(\theta))\varphi_\kappa(0,0,-\pi)\bigg)\right\|\\\le e^{-\frac{t}{\kappa}}\left|\varphi_\kappa(0,0,\theta)-(\chi_\kappa(\theta)\varphi_\kappa(0,0,0)+(1-\chi_\kappa(\theta))\varphi_\kappa(0,0,-\pi))\right|\rightarrow 0.\end{multline*}  Thus, we have that the weak-$\ast$ limit $\varphi$ of $\varphi_\kappa$ solves the limiting system \eqref{FP modified adjoint eq}, \eqref{modified adjoint initial}, and \eqref{modified periodic} with the boundary condition\begin{equation*} \begin{split}
			\varphi(t,0,\theta)&=\varphi(0,0,-\pi)\ \text{if}\ -\pi\le \theta<-\frac{\pi}{2},\ \text{and}\\
			\varphi(t,0,\theta)&=\varphi(0,0,0)\ \text{if}\ -\frac{\pi}{2}<\theta\le 0,
	\end{split}\end{equation*}which is the limit of the regularized boundary condition $$\varphi_\kappa(t,0,\theta)=\bigg(\chi_\kappa(\theta)\varphi_\kappa(0,0,0)+(1-\chi_\kappa(\theta))\varphi_\kappa(0,0,-\pi)\bigg).$$ 
	\end{proof}This completes the proof for the solvability of the adjoint problem \eqref{FP modified adjoint eq}- \eqref{modified periodic}. In the next section, we generalize our method to study the adjoint problem \eqref{FP adjoint eq}, \eqref{adjoint initial}, \eqref{adjoint boundarycon}, and \eqref{adjoint periodic} in the 2-dimensional half-plane.
	
	\section{The generalized adjoint problem for the distribution \texorpdfstring{$\lowercase{f}$}{}}\label{sec:2d section start}
	In this section, we generalize the method and study the full 2-dimensional dual adjoint problem \eqref{FP adjoint eq}, \eqref{adjoint initial}, \eqref{adjoint boundarycon}, and \eqref{adjoint periodic} for the test function $\phi(t,x_1,x_2,\theta)$. This problem is more complicated than the previous problem in a reduced dimension, as we also have the additional dynamics of the free transport in $(t,x_1)$ plane whose speed of propagation varies depending on the $\theta$ variable. 
	Since the functions that satisfies the \textit{trapping} boundary condition \eqref{adjoint boundarycon} can be discontinuous at $(x_2,\theta)=(0,-\pi/2)$, we first regularize the boundary condition.
	\subsection{The regularization of the boundary conditions}\label{sec:2dboundaryreg} In this subsection, we introduce the regularization of the \textit{trapping} boundary condition \eqref{adjoint boundarycon} for the adjoint problem \eqref{FP adjoint eq}- \eqref{adjoint periodic}. 
	Recall the boundary conditions \eqref{adjoint boundarycon}, \eqref{FP adjoint eq}$_2$, and \eqref{FP adjoint eq}$_3$ that 
	\begin{equation}
		\label{2dadjoint trapping boundary}
		\begin{split}
			\partial_t\phi(t,x_1,0,0)&=\partial_{x_1}\phi (t,x_1,0,0),\\
			\partial_t\phi(t,x_1,0,-\pi)&=-\partial_{x_1}\phi (t,x_1,0,-\pi),\\
			\phi(t,x_1,0,\theta)&=\phi(t,x_1,0,\pi)\ \text{if}\ -\pi\le \theta<-\frac{\pi}{2},\ \text{and}\\
			\phi(t,x_1,0,\theta)&=\phi(t,x_1,0,0)\ \text{if}\ -\frac{\pi}{2}<\theta\le 0.
		\end{split}
	\end{equation}Since a solution $\phi$ which satisfies the conditions above can have a discontinuity at $(x_2,\theta)=(0,-\pi/2)$, we regularize the boundary conditions as follows. 
	For each fixed small $\kappa>0$, we define the regularized boundary condition for $\theta \in [-\pi,0] $ as the solution $\phi_\kappa(t,x_1,0,\theta)$ of 
	\begin{equation}
		\label{regularized adjoint boundarycon}
		\begin{split} \partial_t&\phi_\kappa(t,x_1,0,\theta)=\cos\theta \partial_{x_1}\phi_\kappa(t,x_1,0,\theta)\\&+\frac{1}{\kappa}\bigg(\chi_\kappa(\theta)\phi_\kappa(t,x_1,0,0)+(1-\chi_\kappa(\theta))\phi_\kappa(t,x_1,0,-\pi)-\phi_\kappa(t,x_1,0,\theta)\bigg).\end{split}
	\end{equation}
	Here the smooth function $\chi_\kappa$ is given by \eqref{smoothchi}. Then note that if $\theta=0$ or $\theta=-\pi,$ we have\begin{equation}\notag
		\begin{split}
			\partial_t\phi_\kappa(t,x_1,0,0)&=\partial_{x_1}\phi_\kappa (t,x_1,0,0),\\
			\partial_t\phi_\kappa(t,x_1,0,-\pi)&=-\partial_{x_1}\phi_\kappa (t,x_1,0,-\pi).
		\end{split}
	\end{equation} Therefore, $\phi_\kappa (t,x_1,0,0)$ and $\phi_\kappa (t,x_1,0,0)$ are given by the solutions of the free transport equations as
	\begin{equation}\notag
		\begin{split}
			\phi_\kappa(t,x_1,0,0)&=\phi_\kappa (0,x_1+t,0,0),\\
			\phi_\kappa(t,x_1,0,-\pi)&=\phi_\kappa (0,x_1-t,0,-\pi).
		\end{split}
	\end{equation} Here we let $\phi_\kappa(0,x_1,x_2,\theta)=\phi_{in}(x_1,x_2,\theta).$
	\begin{figure}[ht]
		\begin{center}
			\includegraphics[scale=0.3]{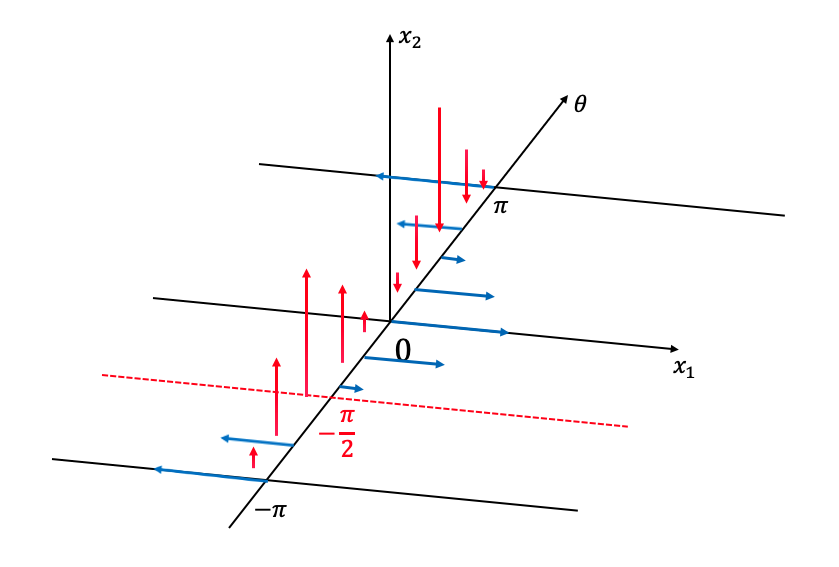}
		\end{center}
		\caption{Characteristics for the 2-dimensional case }
	\end{figure}

	In the rest of this section and in the next section, we will observe that the solvability of the adjoint problem \eqref{FP adjoint eq}-\eqref{adjoint periodic} with \eqref{adjoint boundarycon} replaced by the regularized boundary condition \eqref{regularized adjoint boundarycon} corresponds to the construction of a Markov generator of the semigroup via the Hille-Yosida theorem. 
	\subsection{The Hille-Yosida theorem and the plan} Note that we want to prove the solvability of the 2-dimensional adjoint problem \eqref{FP adjoint eq} together with the initial-boundary conditions \eqref{adjoint initial}, \eqref{regularized adjoint boundarycon} and \eqref{adjoint periodic}. As in the previous sections, our main goal is to prove that the range of the operator $I-\lambda \mathcal{A}$ is equal to $C(X)$:
	\begin{equation}\label{wholerange}\mathcal{R}(I-\lambda \mathcal{A})=C(X),\end{equation} for any small $\lambda>0$, where the linear operator $\mathcal{A}$ is now defined as 
	\begin{equation}
		\label{mathcalA}
		\mathcal{A}\eqdef (\cos\theta,\sin\theta)\cdot \nabla_x + \partial_\theta^2.
	\end{equation}Then, since the adjoint problem is equivalent to
	$$\partial_t \phi_\kappa = \mathcal{A}\phi_\kappa, \ t\in [0,T],\ \phi_\kappa(t,\cdot)\in \mathcal{D}(\mathcal{A}),\text{ if } t\ge 0,\ \phi_\kappa(0,x_1,x_2,\theta)=\upsilon_\kappa(x_1,x_2,\theta),$$ the Hille-Yosida theorem guarantees the one-to-one correspondence of the Markov generator and the Markov semigroup $S_\kappa(t)$, which provides the solvability of the 2-dimensional adjoint problem. We follow the Hille-Yosida theory that is introduced in \cite{Liggett}.

	Observe that \eqref{wholerange} is equivalent to the claim that for any small $ \lambda>0$ and a given function $g\in C(X)$, it is possible to solve the elliptic problem
	\begin{equation}\label{eq4}
		\lambda \mathcal{A}\upsilon_\kappa=\upsilon_\kappa-g.
	\end{equation} 
	Here we define the domain $\mathcal{D}(\mathcal{A})$ of the operator $\mathcal{A}$ which involves the information of the regularized boundary condition \eqref{regularized adjoint boundarycon} as 
	\begin{multline}\label{domainDA}\mathcal{D}(\mathcal{A})=\bigg\{\upsilon_\kappa, \mathcal{A}\upsilon_\kappa\in C^1(X):\text{ for }x_1\in \mathbb{R}\text{ and }\theta \in[-\pi,0]\\ (\mathcal{L}\upsilon_\kappa)(x_1,0,\theta)=\frac{1}{\kappa}\bigg(\chi_\kappa(\theta)\upsilon_\kappa(x_1,0,0)+(1-\chi_\kappa(\theta))\upsilon_\kappa(x_1,0,-\pi)-\upsilon_\kappa(x_1,0,\theta)\bigg)\bigg\},\end{multline}where the operator $\mathcal{L}$ is defined in \eqref{operator L}. \begin{remark}\label{simpletransport}Here we remark that the definition $\mathcal{D}(\mathcal{A})$ is written using the operator $\mathcal{L}$ instead of $\mathcal{A}$, as the leftover dynamics in $t$ and $x_1$ is a simple free-transport on the plane. \end{remark}Then we have the following proposition:
	\begin{proposition}\label{pregenerator proposition2}
		The operator $\mathcal{A}$ is a Markov pregenerator.
	\end{proposition}Note that the boundary condition at $x_2=0$ in \eqref{domainDA} is the same as the one given in \eqref{domainD} and the operator is independent of $x_1$. Thus, for the proof of Proposition \ref{pregenerator proposition2}, there is no need of taking an additional cutoff in $x_1$ variable for the construction of $\lambda$ in the proof of Proposition \ref{pregenerator proposition}. Therefore, the proof is almost the same as the one for Proposition \ref{pregenerator proposition}, and we omit it. As mentioned in Remark \ref{simpletransport}, we observe that the dynamics in $t$ and $x_1$ variables is the free-transport.
	
	Then our goal is to prove that the operator $\mathcal{A}$ is indeed the Markov generator:
	\begin{theorem}\label{markovgen2} The operator $\mathcal{A}$ is a \textit{Markov} generator.
	\end{theorem}
	To this end, it suffices to prove that the operator $\mathcal{A}$ is bounded by Proposition 2.8 (a) of \cite{Liggett}. Equivalently, we will prove the following proposition:
	\begin{proposition}\label{hille2}
		For any $g\in C(X)$ there exists  $\upsilon_\kappa\in \mathcal{D}(\mathcal{A})$ which solves \eqref{eq4}.
	\end{proposition}
	Then Proposition 2.8 (b) of \cite{Liggett} will imply that Theorem \ref{markovgen2} holds.

	\subsubsection{The regularized equation}
	As before, we regularize the Laplace-Beltrami operator as $Q^\epsilon$ defined in \eqref{discretized} with \eqref{zeta property} and obtain the following regularized operator\begin{equation}\label{regularizedmathcalA}
		\mathcal{A}^\epsilon\eqdef (\cos\theta,\sin\theta)\cdot \nabla_x + Q^\epsilon[\ \cdot \ ] .
	\end{equation} Then the regularized problem that we consider is 
	\begin{equation}\label{eq5}
		\lambda \mathcal{A}^\epsilon\upsilon_\kappa^\epsilon=\upsilon_\kappa^\epsilon-g,
	\end{equation}
	where $g\in C(X)$ for small $\lambda>0$.
	
	\section{The solvability of the adjoint problem for \texorpdfstring{$\lowercase{f}$}{}}\label{sec:2d wellposedness}\subsection{A \texorpdfstring{$\lowercase{t}$}{}-dependent problem} The solution $\upsilon_\kappa^\epsilon$ to the elliptic equation \eqref{eq5} can correspond to the solution ${\psi^\epsilon_\kappa}$ of the following $t$-dependent problem: 
	\begin{multline}
		\label{eq6}\qquad\partial_t {\psi^\epsilon_\kappa}  = \mathcal{A}^\epsilon{\psi^\epsilon_\kappa}, \text{ for }t\ge 0 \text{ and } \xi \eqdef (x_1,x_2,\theta)\in X,\\
		{\psi^\epsilon_\kappa} (0,\xi)=\psi_{\kappa,in}(\xi)=g(\xi),\\
	{\psi^\epsilon_\kappa} (t,x_1,x_2,-\pi)={\psi^\epsilon_\kappa} (t,x_1,x_2,\pi),\text { for }t\ge 0, \ x_1\in\mathbb{R},\ x_2\ge 0,	\text{ and }\\
		(\partial_t-\cos\theta\partial_{x_1}) {\psi^\epsilon_\kappa}  (t,x_1,0,\theta)=\frac{1}{\kappa}\bigg(\chi_\kappa(\theta)\psi^\epsilon_\kappa(t,x_1,0,0)\\+(1-\chi_\kappa(\theta))\psi^\epsilon_\kappa(t,x_1,0,-\pi)-\psi^\epsilon_\kappa(t,x_1,0,\theta)\bigg),
		\text{ for }x_1\in\mathbb{R}\text{ and }\theta\in[-\pi,0],
	\end{multline}  via the definition	
	$$\upsilon_\kappa^\epsilon(\xi)=\int_0^\infty e^{-t}{\psi^\epsilon_\kappa} (\lambda t , \xi)\ dt.$$ 
	Note that if we plug $\theta=0$ or $\theta=-\pi$ into \eqref{eq6}$_3$ we have\begin{equation}\notag
		\begin{split}
			\partial_t\psi^\epsilon_\kappa(t,x_1,0,0)&=\partial_{x_1}\psi^\epsilon_\kappa (t,x_1,0,0),\\
			\partial_t\psi^\epsilon_\kappa(t,x_1,0,-\pi)&=-\partial_{x_1}\psi^\epsilon_\kappa (t,x_1,0,-\pi).
		\end{split}
	\end{equation} Therefore, $\psi^\epsilon_\kappa (t,x_1,0,0)$ and $\psi^\epsilon_\kappa (t,x_1,0,0)$ are given by the solutions of the transport equations as
	\begin{equation}\notag
		\begin{split}
			\psi^\epsilon_\kappa(t,x_1,0,0)&=\psi^\epsilon_\kappa (0,x_1+t,0,0),\\
			\psi^\epsilon_\kappa(t,x_1,0,-\pi)&=\psi^\epsilon_\kappa (0,x_1-t,0,-\pi).
		\end{split}
	\end{equation} 
	Without loss of generality, we will assume that $g$ is chosen from a dense subset $C^\infty(X)$  of $C(X)$.
	%

	We first solve the boundary equation \eqref{eq6}$_3$
	\begin{multline}\label{lemma5.1.1}
		(\partial_t-\cos\theta\partial_{x_1}) \psi^\epsilon_\kappa (t,x_1,0,\theta)=\frac{1}{\kappa}\bigg(\chi_\kappa(\theta)\psi^\epsilon_\kappa(0,x_1+t,0,0)\\+(1-\chi_\kappa(\theta))\psi^\epsilon_\kappa(0,x_1-t,0,-\pi)-\psi^\epsilon_\kappa(t,x_1,0,\theta)\bigg),
		\text{ for }x_1\in\mathbb{R}\text{ and }\theta\in[-\pi,0],
	\end{multline}
	as$$ \psi^\epsilon_\kappa(t,\xi)= \psi^\epsilon_\kappa\left(t+\frac{x_2}{\sin\theta},x_1-\frac{x_2\cos\theta}{\sin\theta},0,\theta\right),$$ if $ x_2+t\sin\theta\le 0\text{  and }\theta\in[-\pi,0].$ Note that $\chi_\kappa\in [0,1]$ and $\psi^\epsilon_\kappa(0,\xi)=g(\xi)$. For each fixed $\theta\in[-\pi,0]$, we have the following inhomogeneous transport equation
	\begin{multline}\label{nonhomo transport for Upsilon}
		(\partial_t-\cos\theta\partial_{x_1}) \psi^\epsilon_\kappa (t,x_1,0,\theta)=\frac{1}{\kappa}\bigg(\chi_\kappa(\theta)g(x_1+t,0,0)\\+(1-\chi_\kappa(\theta))g(x_1-t,0,-\pi)-\psi^\epsilon_\kappa(t,x_1,0,\theta)\bigg).
	\end{multline}	
	Define $$
	\psi^\epsilon_{\kappa,g} (t,x_1,0,\theta)\eqdef \psi^\epsilon_\kappa(t,x_1,0,\theta) -\chi_\kappa(\theta)g(x_1+t,0,0)-(1-\chi_\kappa(\theta))g(x_1-t,0,-\pi).
	$$ Then $ \psi^\epsilon_{\kappa,g}$ satisfies 
	\begin{multline}
		(\partial_t-\cos\theta\partial_{x_1}) \psi^\epsilon_{\kappa,g} (t,x_1,0,\theta)=-\frac{1}{\kappa}\psi^\epsilon_{\kappa,g}(t,x_1,0,\theta)\\
		-(\partial_t -\cos\theta\partial_{x_1}) \bigg(\chi_\kappa(\theta)g(x_1+t,0,0)+(1-\chi_\kappa(\theta))g(x_1-t,0,-\pi)\bigg)\\
		=-\frac{1}{\kappa}\psi^\epsilon_{\kappa,g}(t,x_1,0,\theta)
		-\chi_\kappa(\theta)(1-\cos\theta)\partial_{x_1}g(x_1+t,0,0)\\+(1-\chi_\kappa(\theta))(1+\cos\theta)\partial_{x_1}g(x_1-t,0,-\pi).
	\end{multline}
	Then we use Duhamel's principle and obtain that
	\begin{multline}\label{boundary duhamel 2d}
		\psi^\epsilon_{\kappa,g}(t,x_1,0,\theta)= e^{-\frac{t}{\kappa}}	\psi^\epsilon_{\kappa,g}(0,x_1+t\cos\theta ,0,\theta)\\+\int_0^te^{-\frac{t-s}{\kappa}}\bigg(	-\chi_\kappa(\theta)(1-\cos\theta)\partial_{x_1}g(x_1-(t-s)\cos\theta+s,0,0)\\+ (1-\chi_\kappa(\theta))(1+\cos\theta)\partial_{x_1}g(x_1+(t-s)\cos\theta-s,0,-\pi)\bigg)ds.
	\end{multline}We remark that this formula depends on the first derivative of $g$ and this is one of the main differences from the 1-dimensional case discussed in Section \ref{sec: 1d wellposedness}. Since $$\psi^\epsilon_{\kappa,g}(0,x_1,0,\theta)=g(x_1,0,\theta)-\chi_\kappa(\theta)g(x_1,0,0)-(1-\chi_\kappa(\theta))g(x_1,0,-\pi),$$ we have
	\begin{multline}\label{boundary final duhamel 2d}\|\psi^\epsilon_{\kappa,g}(t,\cdot,0,\cdot)\|_{L^\infty(\mathbb{R}\times [-\pi,\pi])}\le 2e^{-\frac{t}{\kappa}}\|g\|_{L^\infty(X)}+2\|\partial_{x_1}g\|_{L^\infty(X)}\int_0^t e^{-\frac{t-s}{\kappa}}ds\\=2e^{-\frac{t}{\kappa}}\|g\|_{L^\infty(X)}+2\kappa\|\partial_{x_1}g\|_{L^\infty(X)} .\end{multline}Thus, we have 
	$$\|\psi^\epsilon_{\kappa,g}(t,\cdot,0,\cdot)\|_{L^\infty(\mathbb{R}\times [-\pi,\pi])}\to 0,$$ as $\kappa\to 0$ for $t>0$.
	
	\subsection{Construction of a mild solution}In this section, we will construct a mild solution to the $t$-dependent problem via Duhamel's principle
	\subsubsection{Homogeneous problem}We first consider a solution $\Upsilon=\Upsilon(t,\xi)$ to the following free transport equation:
	\begin{multline}\label{2dhomo}
		\qquad \partial_t \Upsilon =(\cos\theta,\sin\theta)\cdot\nabla_x \Upsilon ,\ \ 
		\Upsilon(0,\xi)=g(    \xi),\text{ for }\xi\in X,\\
	\Upsilon (t,x_1,x_2,-\pi)=\Upsilon (t,x_1,x_2,\pi),\text { for }t\ge 0, \ x_1\in\mathbb{R},\ x_2\ge 0,	\text{ and }\\
		(\partial_t-\cos\theta\partial_{x_1}) \Upsilon (t,x_1,0,\theta)=\frac{1}{\kappa}\bigg(\chi_\kappa(\theta)\Upsilon(0,x_1+t,0,0)\\+(1-\chi_\kappa(\theta))\Upsilon(0,x_1-t,0,-\pi)-\Upsilon(t,x_1,0,\theta)\bigg),
		\text{ for }x_1\in\mathbb{R}\text{ and }\theta\in[-\pi,0].
	\end{multline}
	Then the solution $\Upsilon$ is given by
	\begin{equation}
		\Upsilon(t,x,\theta)=\begin{cases}g(x_1+t\cos\theta,x_2+t\sin\theta, \theta),\ &\text{ if } x_2+t\sin\theta>0\text{ or }\theta\in[0,\pi]\\ \Upsilon\left(t+\frac{x_2}{\sin\theta},x_1-\frac{x_2\cos\theta}{\sin\theta},0,\theta\right),&\text{ if } x_2+t\sin\theta\le 0 \text { and } \theta \in\left(-\pi,0\right).
		\end{cases}
	\end{equation}
	
	\subsubsection{Inhomogeneous problem}\label{duhamel2}
	Then, by the variation of parameters, the solution to \eqref{eq6} can be obtained. 
	Let us denote $x=(x_1,x_2)$ below.
	\begin{enumerate}
		\item  If $x_2+t\sin\theta >0\text{ or }\theta\in[0,\pi]$, we have
		$${\psi^\epsilon_\kappa} (t,x,\theta)=\Upsilon(t,x,\theta)+\int_0^t Q^\epsilon[{\psi^\epsilon_\kappa} ](s,U(t-s)(\xi) )ds,$$ where the semigroup $U(t)$ is given by
		$$U(t)\xi=U(t)(x,\theta)\eqdef (x+t(\cos\theta,\sin\theta),\theta).$$
		Therefore, we have
		$$\qquad{\psi^\epsilon_\kappa} (t,x,\theta)=g(x+t(\cos\theta,\sin\theta),\theta)+\int_0^t Q^\epsilon[{\psi^\epsilon_\kappa} ](s,x+(t-s)(\cos\theta,\sin\theta),\theta )ds.$$
		
		\item  On the other hand, if $x_2+t\sin\theta \le0\text{ and }\theta\in[-\pi,0]$, we have 
		\begin{multline*}\qquad{\psi^\epsilon_\kappa} (t,x,\theta)=\Upsilon\left(t+\frac{x_2}{\sin\theta},x_1-\frac{x_2\cos\theta}{\sin\theta},0,\theta\right)\\
			+\int_{t+\frac{x_2}{\sin\theta}}^t Q^\epsilon[{\psi^\epsilon_\kappa} ](s,x+(t-s)(\cos\theta,\sin\theta),\theta )ds.\end{multline*}
	\end{enumerate}
	\subsection{Solvability} 
	Define an operator $\mathcal{T}^\epsilon$ as
	$$\mathcal{T}^\epsilon[{\psi^\epsilon_\kappa} ]\eqdef \Upsilon(t,\xi)+\int_{\max\left\{0,\ t+\frac{x_2}{\sin\theta}\right\}} ^t ds \ Q^\epsilon[{\psi^\epsilon_\kappa} ](s,U(t-s)\xi),$$ where 
	\begin{equation}
		\Upsilon(t,\xi)= \begin{cases}
			& g(x+t(\cos\theta,\sin\theta),\theta),\text{ if }x_2+t\sin\theta >0\text{ or }\theta\in[0,\pi]\\
			& \Upsilon\left(t+\frac{x_2}{\sin\theta},x_1-\frac{x_2\cos\theta}{\sin\theta},0,\theta\right), \text{ if }x_2+t\sin\theta\le 0\text{  and }\theta\in[-\pi,0].\end{cases}
	\end{equation}
	Define a set $\mathrm{E}(X)$ as
	$$\mathrm{E}(X)=\{\psi \in C(X) \ : \ \|\psi(t)\|_{C^{(1,1,2)}(X)}\le C\|g\|_{C^{3}(X)}\text{ for any }t\in[0,T]\},$$
	for some $C>54$.

	%
	%
	%

	Then, we obtain the following lemma.
	\begin{lemma}[Contraction mapping]\label{contraction2d}If ${\psi^\epsilon_\kappa} \in \mathrm{E}(X)$, then we have $\mathcal{T}^\epsilon[{\psi^\epsilon_\kappa} ]\in \mathrm{E}(X).$ Moreover,
		$\mathcal{T}^\epsilon$ is a contraction in $\mathrm{E}(X)$, for $T=T_1$ sufficiently small.
	\end{lemma}
	\begin{proof}

		Now we show that  $\mathcal{T}^\epsilon$  maps  $E(X)$  to itself and is a contraction, if  $T_1=T_1(\epsilon) >0$  is sufficiently small. 
		Note that for any $\psi\in E(X)$ and any multi-index $\alpha \le (1,1,2)$ in the partial order, we have
		\begin{equation*}
			\big|Q^\epsilon[\partial^\alpha\psi](t,x_1,x_2,\theta)\big|  \leq  
			\frac{4}{\epsilon^2}  \|\psi(t)\|_{ C^{(1,1,2)}_{x_1,x_2,\theta}}\! \int_{(-\pi,\pi)}\!\zeta(\nu) d\nu  \leq 
			\frac{4C}{\epsilon^2} \|g\|_{C^{3}},
		\end{equation*}
		by the definition of $Q^\epsilon[\psi]$. We now take the derivatives $\partial^\alpha$ for $\alpha\le (1,1,2)$ on the solutions in the mild form.
		
		If $x_2+t\sin\theta>0$ or $\theta\in [0,\pi]$, we observe that
		\begin{multline*}
			\qquad\partial^{m,l}_{x_1,x_2}{\psi^\epsilon_\kappa} (t,x_1,x_2,\theta)=\partial^{m,l}_{x_1,x_2}g(x_1+t\cos\theta,x_2+t\sin\theta,\theta)\\
			+ Q^\epsilon[{\partial^{m,l}_{x_1,x_2}\psi^\epsilon_\kappa} ](s,U(t-s)(x_1,x_2,\theta) )ds,
		\end{multline*} for any non-negative integers $m$ and $l$, and
		\begin{multline*}\partial_{\theta}\psi^\epsilon_\kappa (t,x_1,x_2,\theta)=-t\sin\theta\partial_{x_1}g(x_1+t\cos\theta,x_2+t\sin\theta,\theta)\\+t\cos\theta\partial_{x_2}g(x_1+t\cos\theta,x_2+t\sin\theta,\theta)+\partial_{\theta}g(x_1+t\cos\theta,x_2+t\sin\theta,\theta)\\+\int_0^t Q^\epsilon[{(-(t-s)\sin\theta\partial_{x_1}+(t-s)\cos\theta\partial_{x_2}+\partial_{\theta})\psi^\epsilon_\kappa} ]\\
			(s,x_1+(t-s)\cos\theta,x_2+(t-s)\sin\theta,\theta)ds,\end{multline*}
		and
		\begin{multline*}\partial^2_{\theta}\psi^\epsilon_\kappa (t,x_1,x_2,\theta)
			=-t\cos\theta\partial_{x_1}g(x_1+t\cos\theta,x_2+t\sin\theta,\theta)\\+t^2\sin^2\theta\partial^2_{x_1}g(x_1+t\cos\theta,x_2+t\sin\theta,\theta)-t\sin\theta\partial_{x_2}g(x_1+t\cos\theta,x_2+t\sin\theta,\theta)\\+t^2\cos^2\theta\partial^2_{x_2}g(x_1+t\cos\theta,x_2+t\sin\theta,\theta)-2t\sin\theta\partial^{(1,0,1)}_{x_1,x_2,\theta}g(x_1+t\cos\theta,x_2+t\sin\theta,\theta)\\+2t\cos\theta\partial^{(0,1,1)}_{x_1,x_2,\theta}g(x_1+t\cos\theta,x_2+t\sin\theta,\theta)+\partial^2_{\theta}g(x_1+t\cos\theta,x_2+t\sin\theta,\theta)\\+\int_0^t Q^\epsilon\bigg[\bigg(-(t-s)\cos\theta\partial_{x_1}+(t-s)^2\sin^2\theta\partial^2_{x_1}-(t-s)\sin\theta\partial_{x_2}+(t-s)^2\cos\theta^2\partial_{x_2}^2\\-2(t-s)\sin\theta\partial^{(1,0,1)}_{x_1,x_2,\theta}+2(t-s)\cos\theta\partial^{(0,1,1)}_{x_1,x_2,\theta}+\partial^2_\theta\bigg)\psi^\epsilon_\kappa\bigg](s,x_2+(t-s)\sin\theta,\theta)ds.\end{multline*}On the other hand, if $x_2+t\sin\theta\le 0$ and $\theta \in (-\pi,0)$, we have 	 $$\left|\partial^\alpha_{x_1,x_2} \Upsilon\left(t+\frac{x_2}{\sin\theta},x_1-\frac{x_2\cos\theta}{\sin\theta},0,\theta\right)\right|\le 2\|\partial^\alpha_{x_1,x_2} g\|_{L^\infty(X)}+2\|\partial^\alpha_{x_1,x_2}\partial_{x_1}g\|_{L^\infty(X)},$$ by \eqref{boundary duhamel 2d} and \eqref{2dhomo}$_2$ for any $x$-derivatives. Also, regarding the $\theta$-derivatives, we have $$\left|\partial^\beta_{\theta} \Upsilon\left(t+\frac{x_2}{\sin\theta},x_1-\frac{x_2\cos\theta}{\sin\theta},0,\theta\right)\right|\le 2^{\beta+1}\|\partial^\beta_{\theta} g\|_{L^\infty(X)}+2^{\beta+1}\|\partial^\beta_{\theta}\partial_{x_1}g\|_{L^\infty(X)},$$ by \eqref{boundary duhamel 2d}.
		Therefore, for any  $t\in[0,T_1]$ and $\alpha\le (1,1,2)$ with $|\alpha|\le 2$, we have
		\begin{align*}
			\partial^\alpha_{x_1,x_2,\theta}\mathcal{T}^\epsilon[\psi](t)
			&  \leq  2\bigg (t+\frac{3}{2}\bigg)^2\|g\|_{C^3} +\bigg(\frac{2t^3}{3}+3t^2+t\bigg) \big|Q^\epsilon[\partial^\alpha\psi]\big| \\[2pt]
			&  \leq\left(  2\bigg(T_1+\frac{3}{2}\bigg)^2+	\frac{4C}{\epsilon^2}\bigg(\frac{2T_1^3}{3}+3T_1^2+T_1\bigg) \right)\|g\|_{C^3}   \leq  \frac{C}{12}\|g\|_{C^3},
		\end{align*} 
		provided  $T_1$  is chosen sufficiently small such that $ \frac{4}{\epsilon^2}\bigg(\frac{2T_1^3}{3}+3T_1^2+T_1\bigg) \ll1$ and $$2\bigg(T_1+\frac{3}{2}\bigg)^2+	\frac{4C}{\epsilon^2}\bigg(\frac{2T_1^3}{3}+3T_1^2+T_1\bigg) \le \frac{C}{12} ,$$ for given $C>54.$ 
		Then we have
		\begin{equation*}
			\sup_{t\in [0,T_1]}\|\mathcal{T}^\epsilon[\psi](t)\|_{ C^{(1,1,2)}_{x_1,x_2,\theta}} \leq  C\|g\|_{C^{3}(X)} .
		\end{equation*}
		All these above imply that  $\mathcal{T}^\epsilon[\psi]\in E(X)$  and hence  $\mathcal{T}^\epsilon$  maps  $E(X)$  into  $E(X)$.
		In addition, for any two $\psi_1, \ \psi_2 \in E(X)$  and  $t\in[0,T_1]$, similar arguments yield that
		\begin{multline*}
			\|\mathcal{T}^\epsilon[\psi_1](t) - \mathcal{T}^\epsilon[\psi_2](t)\|_{C^{(1,1,2)}_{x_1,x_2,\theta}}
			=  \|\mathcal{T}^\epsilon[\psi_1\!-\psi_2](t)\|_{C^{(1,1,2)}_{x_1,x_2,\theta}} \\[5pt]
			\leq \bigg(\frac{2T_1^3}{3}+3T_1^2+T_1\bigg) \frac{4}{\epsilon^2}  \sup_{t\in [0,T_1]}\|\psi_1(t)-\psi_2(t)\|_{C^{(1,1,2)}_{x_1,x_2,\theta}}.
		\end{multline*}
		Since  $\bigg(\frac{2T_1^3}{3}+3T_1^2+T_1\bigg) \frac{4}{\epsilon^2}  <1$ with our choice of  $T_1$  above, we conclude that  $\mathcal{T}^\epsilon$  is a contraction.
	\end{proof}	\begin{corollary}There exists a unique global mild solution ${\psi^\epsilon_\kappa}$ in $\mathrm{E}(X)$ to \eqref{eq6}.
	\end{corollary}\begin{proof}
		Therefore, by the Schauder-type fixed-point theorem, there exists a unique mild solution in $E(X)$  on the interval  $[0,T_1]$  for such  $T_1=T_1(\epsilon)$  that $\bigg(\frac{2T_1^3}{3}+3T_1^2+T_1\bigg) \frac{4}{\epsilon^2}   <1$ is satisfied.
		For the global existence, note that $T_1=T_1(\epsilon)$ does not depend on the initial data $g$. Then, by a continuation argument, we can extend the existence interval to an arbitrary  $T>0$ independent of  $\epsilon$.
	\end{proof}
	Therefore, we obtain the global wellposedness in $\mathrm{E}(X)$ for the regularized problem \eqref{eq6}. In the next subsection, we will establish the uniform-in-$\epsilon$ estimates for $\psi^\epsilon_\kappa$ and its derivatives.

	\subsection{Uniform-in-\texorpdfstring{$\epsilon$}{} estimates}
	In this section, we will provide a uniform-in-$\epsilon$ $L^\infty$ estimate for the solutions to \eqref{eq6}.
	\begin{lemma}[maximum principle]\label{maxprinciple2}
	Define 	\begin{equation}\label{M2}
		\mathcal{M} h\eqdef \left(\partial_t-(\cos\theta,\sin\theta)\cdot \nabla_x-Q^\epsilon \right)h, 
	\end{equation} and let a $2\pi$-periodic(-in-$\theta$) function $h\in C^{1,1,1,2}_{t,x_1,x_2,\theta}$ solve $\mathcal{M} h\le 0$. Then $h$ attains its maximum only when $t=0$ or when $x_2=0$ and $\theta\in[-\pi,0]$. Therefore, a solution ${\psi^\epsilon_\kappa} $ to \eqref{eq6} satisfies 
		$$\|{\psi^\epsilon_\kappa} \|_{L^\infty([0,T]\times X)}\le 3\|g\|_{L^\infty(X)}+2\kappa\|\partial_{x_1}g\|_{L^\infty(X)} .$$\end{lemma}
	\begin{proof}
Note that the derivatives $\partial_t{\psi^\epsilon_\kappa}$ and $\nabla_x{\psi^\epsilon_\kappa}$ are well-defined due to the mild form of the solution in Section \ref{duhamel2} and that $g$ is sufficiently smooth.
		
		First of all, we suppose that a $2\pi$-periodic-in-$\theta$ function ${\psi^\epsilon_\kappa}$ satisfies $\mathcal{M}{\psi^\epsilon_\kappa}<0$. \begin{itemize}
			\item If ${\psi^\epsilon_\kappa}$ also attains
			its maximum either at an interior point $(t,x,\theta) \in (0,T) \times\mathbb{R}\times (0,\infty)\times
			(-\pi,\pi)$. Then we have $\partial_{t}
			{\psi^\epsilon_\kappa} =\partial_{x_1}{\psi^\epsilon_\kappa}=\partial_{x_2} {\psi^\epsilon_\kappa} =0$ while $Q^\epsilon {\psi^\epsilon_\kappa} \le0$ at the maximum. Thus $\mathcal{M} {\psi^\epsilon_\kappa}(t,x,\theta) \ge0$ and this contradicts the assumption.
			
			\item If the maximum occurs on $(T,x,\theta)$ with $(x,\theta) \in\mathbb{R}\times (0,\infty)\times
			(-\pi,\pi)$,  then $\partial_{t}
			{\psi^\epsilon_\kappa}\ge 0$ and $ \partial_{x_1}{\psi^\epsilon_\kappa}=\partial_{x_2} {\psi^\epsilon_\kappa} =0$ while $Q^\epsilon {\psi^\epsilon_\kappa} \le0$ at the maximum. Thus $\mathcal{M} {\psi^\epsilon_\kappa}(t,x,\theta) \ge0$ and this contradicts the assumption.
			
			\item	If the maximum occurs on $(t,x_1,0,\theta)$ with $t\in (0,T)$, $x_1\in\mathbb{R}$ and $\theta \in [0,\pi)$, then $\partial_{t}
			{\psi^\epsilon_\kappa}= \partial_{x_1}{\psi^\epsilon_\kappa}=0$ and $-\sin\theta \partial_{x_2} {\psi^\epsilon_\kappa} \ge 0$ while $Q^\epsilon {\psi^\epsilon_\kappa} \le0$ at the maximum. Thus $\mathcal{M} {\psi^\epsilon_\kappa}(t,x,\theta) \ge0$ and this contradicts the assumption.
			
			\item If the maximum occurs on $(t,x,\pm \pi)$ with $t\in (0,T)$ and $x\in\mathbb{R}\times (0,\infty)$, then $\partial_{t}
			{\psi^\epsilon_\kappa}=\partial_{x_1} {\psi^\epsilon_\kappa}=\partial_{x_2} {\psi^\epsilon_\kappa} =0$ while $Q^\epsilon {\psi^\epsilon_\kappa} \le0$ at the maximum. Thus $\mathcal{M} {\psi^\epsilon_\kappa}(t,x,\theta) \ge0$ and this contradicts the assumption.
			\item  Now suppose that the maximum occurs at $|x_1|=\infty$. Without loss of generality, suppose that it occurs at $x_1=+\infty$. Then there exists a sufficiently small $\varepsilon>0$ such that $\mathcal{M}{\psi^\epsilon_\kappa}(\infty)+\varepsilon\le 0.$ 
			Then by the continuity of $\mathcal{M}{\psi^\epsilon_\kappa}$,  there exists a sufficiently large constant $R>0$ such that if $x_1\ge  R$, then $\mathcal{M}{\psi^\epsilon_\kappa}(t,x_1,x_2,\theta)<0,$ for any $(t,x_2,\theta) \in [0,T]\times[0,\infty]\times  [-\pi,\pi].$ Also, since ${\psi^\epsilon_\kappa}(\infty)=\max_{\zeta\in S} {\psi^\epsilon_\kappa}(\zeta)$, there exist a sufficiently large $x_1^*>2R$ and a small constant $\delta>0$ such that 
			the local maximum of ${\psi^\epsilon_\kappa}(t,x_1,x_2,\theta)$ on the neighborhood $N_\delta\eqdef (T/2-\delta,T/2+\delta)\times (x_1^*-\delta,x_1^*+\delta)\times (R-\delta,R+\delta)\times  (-3\pi/4-\delta,-3\pi/4+\delta)$ occurs at the point $(s,y_1,y_2,\psi)$ on the upper-in-$x_1$ boundary of $N_\delta$ with  $y_1=x_1^*+\delta$, $s=(T/2-\delta,T/2+\delta)$, $y_2\in(R-\delta,R+\delta)$  and $\psi \in (-3\pi/4-\delta,-3\pi/4+\delta).$ 
			Then note that
			\begin{multline*}
				\qquad \qquad	0>\mathcal{M}{\psi^\epsilon_\kappa}(s,y_1,y_2,\psi)\\\qquad \quad = \partial_t {\psi^\epsilon_\kappa}-\cos\psi\partial_{x_1}{\psi^\epsilon_\kappa}(s,y_1,y_2,\psi)-\sin\psi\partial_{x_2}{\psi^\epsilon_\kappa}(s,y_1,y_2,\psi)-\partial_\theta^2 {\psi^\epsilon_\kappa}(s,y_1,y_2,\psi)  \\\ge -\cos\psi\partial_{x_1}{\psi^\epsilon_\kappa}(s,y_1,y_2,\psi)\ge 0,
			\end{multline*}which leads to the contradiction. 
			\item  Finally, if the  maximum occurs at $x_2=\infty$, there exists a sufficiently small $\varepsilon>0$ such that $\mathcal{M}{\psi^\epsilon_\kappa}(\infty)+\varepsilon\le 0.$ 
			Then by the continuity of $\mathcal{M}{\psi^\epsilon_\kappa}$,  there exists a sufficiently large constant $R>0$ such that if $x_2\ge R$, then $\mathcal{M}{\psi^\epsilon_\kappa}(t,x_1,x_2,\theta)<0,$ for any $(t,x_1,\theta) \in [0,T]\times\mathbb{R}\times   [-\pi,\pi].$ Also, since ${\psi^\epsilon_\kappa}(\infty)=\max_{\zeta\in S} {\psi^\epsilon_\kappa}(\zeta)$, there exists a sufficiently large $x_2^*>2R$ and a small constant $\delta>0$ such that 
			the local maximum of ${\psi^\epsilon_\kappa}(t,x_1,x_2,\theta)$ on the neighborhood $N_\delta\eqdef (T/2-\delta,T/2+\delta)\times (-\delta,\delta)\times (x_2^*-\delta,x_2^*+\delta)\times (-\pi/2-\delta,-\pi/2+\delta)$ occurs at the point $(s,y_1,y_2,\psi)$ on the upper-in-$x_2$ boundary of $N_\delta$ with  $y_2=x_2^*+\delta$, $s=(T/2-\delta,T/2+\delta)$, $y\in (-\delta,\delta)$ and $\psi \in (-\pi/2-\delta,-\pi/2+\delta).$ 
			Then note that
			\begin{multline*}
				\qquad \qquad	0>\mathcal{M}{\psi^\epsilon_\kappa}(s,y_1,y_2,\psi)\\\qquad \quad = \partial_t {\psi^\epsilon_\kappa}-\cos\psi\partial_{x_1}{\psi^\epsilon_\kappa}(s,y_1,y_2,\psi)-\sin\psi\partial_{x_2}{\psi^\epsilon_\kappa}(s,y_1,y_2,\psi)-\partial_\theta^2 {\psi^\epsilon_\kappa}(s,y_1,y_2,\psi)  \\\ge -\sin\psi\partial_{x_2}{\psi^\epsilon_\kappa}(s,y_1,y_2,\psi)\ge 0,
			\end{multline*}which leads to the contradiction. 
		\end{itemize}	
		Therefore, if some function ${\psi^\epsilon_\kappa}$ satisfies $\mathcal{M}{\psi^\epsilon_\kappa}<0$, then the maximum occurs only when $t=0$ or when $x_2=0$ and $\theta\in [-\pi,0]$. 
		
		Now, if $\mathcal{M} {\psi^\epsilon_\kappa}$ is just $ \le0$ then we define a new function  $\psi^{\epsilon, k}_\kappa := {\psi^\epsilon_\kappa}-kt$ for some
		$k>0$ so that $\mathcal{M} \psi^{\epsilon, k} <0$. Then we have
		$$
		\sup_{(t,x,\theta)\in[0,T]\times X } \psi^{\epsilon,k}_\kappa(t,x,\theta) \\
		=
		\sup_{\{t=0\}\ \text{or} \ \{x_2=0 \text{ and }\theta\in [-\pi,0]\} } \psi^{\epsilon,k}_\kappa(t,x,\theta).
		$$
		We take $k \rightarrow0$ and observe the estimate of $\psi_\kappa^\epsilon$ on $x_2=0$ and $\theta \in [-\pi,0]$ from \eqref{boundary final duhamel 2d} that 
		$$\|\psi_\kappa^\epsilon\|_{L^\infty([0,T]\times X)}\le3\|g\|_{L^\infty(X)}+2\kappa\|\partial_{x_1}g\|_{L^\infty(X)}.$$
		This completes the proof.
	\end{proof}
	
	\begin{remark}As a corollary, we obtain the uniqueness of the solution via the maximum principle.
	\end{remark}
	Now we establish the uniform estimates on the derivatives. We first consider $x$-derivatives which does not change the form of the equation:
	\begin{corollary}[Estimates for $x$ derivatives]\label{2d x derivative estimates}
		For any integers $m,l\ge 0$, the solution ${\psi^\epsilon_\kappa} $ to \eqref{eq6} satisfies 
		$$\|{\partial^m_{x_1}\partial^l_{x_2}\psi^\epsilon_\kappa} \|_{L^\infty([0,T]\times X)}
		\le 3\|\partial^m_{x_1}\partial^l_{x_2}g\|_{L^\infty(X)}+2\kappa\|\partial^{m+1}_{x_1}\partial^l_{x_2}g\|_{L^\infty(X)}.$$\end{corollary}\begin{proof}
		By taking $\partial^m_{x_1}\partial^l_{x_2}$ to the equation \eqref{eq6}, we observe that $\partial^m_{x_1}\partial^l_{x_2}\psi^\epsilon_{\kappa}$ satisfies the same equation with revised initial and boundary conditions in the same kind:
		\begin{multline}
			\notag\qquad\partial^m_{x_1}\partial^l_{x_2}\partial_t {\psi^\epsilon_\kappa}  = \mathcal{A}^\epsilon{\partial^m_{x_1}\partial^l_{x_2}\psi^\epsilon_\kappa}, \text{ for }t\ge 0 \text{ and } \xi \eqdef (x_1,x_2,\theta)\in X,\\
			\partial^m_{x_1}\partial^l_{x_2}{\psi^\epsilon_\kappa} (0,\xi)=\partial^m_{x_1}\partial^l_{x_2}g(\xi),\text{ and }\\
			(\partial_t-\cos\theta\partial_{x_1}) {\partial^m_{x_1}\partial^l_{x_2}\psi^\epsilon_\kappa}  (t,x_1,0,\theta)=\frac{1}{\kappa}\bigg(\chi_\kappa(\theta)\partial^m_{x_1}\partial^l_{x_2}\psi^\epsilon_\kappa(t,x_1,0,0)\\+(1-\chi_\kappa(\theta))\partial^m_{x_1}\partial^l_{x_2}\psi^\epsilon_\kappa(t,x_1,0,-\pi)-\partial^m_{x_1}\partial^l_{x_2}\psi^\epsilon_\kappa(t,x_1,0,\theta)\bigg),\\
			\text{ for }x_1\in\mathbb{R}\text{ and }\theta\in[-\pi,0],
		\end{multline} 
		Then the corollary follows by Lemma \ref{maxprinciple2}.
	\end{proof}
	Now we consider the derivatives with respect to $\theta$. If we take one $\theta$-derivative to the equation \eqref{eq6}, we obtain the following inhomogeneous equation
	\begin{equation}\label{eq7}\mathcal{M}(\partial_\theta \psi^\epsilon_\kappa)=-\sin\theta \partial_{x_1}\psi^\epsilon_\kappa+\cos\theta \partial_{x_2}\psi^\epsilon_\kappa,\end{equation} for $\mathcal{M}$ from \eqref{M2} with the initial-boundary conditions: 
	\begin{equation}\label{eq8}\begin{split}
		\partial_\theta	{\psi^\epsilon_\kappa} (0,\xi)=\partial_\theta g(\xi),\text{ and }&\\
		(\partial_t-\cos\theta\partial_{x_1}) {	\partial_\theta	\psi^\epsilon_\kappa}  (t,x_1,0,\theta)&=-\sin\theta\partial_{x_1}\psi^\epsilon_\kappa (t,x_1,0,\theta)\\+\frac{1}{\kappa}\bigg(	\partial_\theta\chi_\kappa(\theta)		\psi^\epsilon_\kappa(t,x_1,0,0)&-	\partial_\theta\chi_\kappa(\theta)	\psi^\epsilon_\kappa(t,x_1,0,-\pi)-	\partial_\theta	\psi^\epsilon_\kappa(t,x_1,0,\theta)\bigg),\\
	&	\text{ for }x_1\in\mathbb{R}\text{ and }\theta\in[-\pi,0].
	\end{split}\end{equation}
	Note that we have an additional inhomogeity $-\sin\theta\partial_{x_1}\psi^\epsilon_\kappa (t,x_1,0,\theta)$ in both the equation \eqref{eq7} for $\partial_\theta\psi^\epsilon_\kappa$ and the boundary equation \eqref{eq8}$_2$. We first study the homogeneous solution to the homogeneous equation $\mathcal{M}(\partial_\theta \psi^\epsilon_\kappa)=0.$ We first define 
	\begin{multline}\label{eq5.13}\partial_\theta	{\psi^\epsilon_{\kappa,g}}(t,x_1,0,\theta)\eqdef \\	-\partial_\theta\chi_\kappa(\theta)		\psi^\epsilon_\kappa(t,x_1,0,0)+	\partial_\theta\chi_\kappa(\theta)	\psi^\epsilon_\kappa(t,x_1,0,-\pi)+	\partial_\theta	\psi^\epsilon_\kappa(t,x_1,0,\theta).\end{multline}
	Then the boundary equation \eqref{eq8}$_2$ is equivalent to\begin{multline}
		(\partial_t-\cos\theta\partial_{x_1}) \partial_\theta {\psi^\epsilon_{\kappa,g}}(t,x_1,0,\theta)\\=-\sin\theta\partial_{x_1}\psi^\epsilon_\kappa (t,x_1,0,\theta)-\partial_\theta\chi_\kappa(\theta)		(\partial_t-\cos\theta\partial_{x_1})\psi^\epsilon_\kappa(t,x_1,0,0)\\+	\partial_\theta\chi_\kappa(\theta)	(\partial_t-\cos\theta\partial_{x_1})\psi^\epsilon_\kappa(t,x_1,0,-\pi)-\frac{1}{\kappa}\partial_\theta{\psi^\epsilon_{\kappa,g}}(t,x_1,0,\theta)\\
		=-\sin\theta\partial_{x_1}\psi^\epsilon_\kappa (t,x_1,0,\theta)-\partial_\theta\chi_\kappa(\theta)	(\partial_t-\cos\theta\partial_{x_1})g(x_1+t,0,0)\\+	\partial_\theta\chi_\kappa(\theta)	(\partial_t-\cos\theta\partial_{x_1})g(x_1-t,0,-\pi)-\frac{1}{\kappa}\partial_\theta{\psi^\epsilon_{\kappa,g}}(t,x_1,0,\theta)\\
		=-\sin\theta\partial_{x_1}\psi^\epsilon_\kappa (t,x_1,0,\theta)-\partial_\theta\chi_\kappa(\theta)	(1-\cos\theta)\partial_{x_1}g(x_1+t,0,0)\\-	\partial_\theta\chi_\kappa(\theta)	(1+\cos\theta)\partial_{x_1}g(x_1-t,0,-\pi)-\frac{1}{\kappa}\partial_\theta{\psi^\epsilon_{\kappa,g}}(t,x_1,0,\theta).
	\end{multline}
	Then we use the Duhamel principle to this equation similarly as in \eqref{boundary duhamel 2d} and obtain that
	\begin{multline}\notag
		\partial_\theta	{\psi^\epsilon_{\kappa,g}}(t,x_1,0,\theta)= e^{-\frac{t}{\kappa}}	\partial_\theta	{\psi^\epsilon_{\kappa,g}}(0,x_1+t\cos\theta ,0,\theta)\\+\int_0^te^{-\frac{t-s}{\kappa}}\bigg(-\sin\theta\partial_{x_1}\psi^\epsilon_\kappa (s,x_1+(t-s)\cos\theta,0,\theta)\\	-\partial_\theta\chi_\kappa(\theta)(1-\cos\theta)\partial_{x_1}g(x_1+s-(t-s)\cos\theta,0,0)\\ -\partial_\theta\chi_\kappa(\theta)(1+\cos\theta)\partial_{x_1}g(x_1-s+(t-s)\cos\theta,0,-\pi)\bigg)ds.
	\end{multline} Since \eqref{eq5.13} implies \begin{multline*}\partial_\theta \psi^\epsilon_{\kappa,g}(0,x_1,0,\theta)\\=	-\partial_\theta\chi_\kappa(\theta)		\psi^\epsilon_\kappa(0,x_1,0,0)+	\partial_\theta\chi_\kappa(\theta)	\psi^\epsilon_\kappa(0,x_1,0,-\pi)+	\partial_\theta	\psi^\epsilon_\kappa(0,x_1,0,\theta)\\=	-\partial_\theta\chi_\kappa(\theta)		g(x_1,0,0)+	\partial_\theta\chi_\kappa(\theta)	g(x_1,0,-\pi)+	\partial_\theta	g(x_1,0,\theta),\end{multline*} we have for some $C_\kappa>0$
	\begin{multline}\notag\|\partial_\theta		\psi^\epsilon_{\kappa,g}(t,\cdot,0,\cdot)\|_{L^\infty(\mathbb{R}\times [-\pi,0])}\\\le C_\kappa\bigg( e^{-\frac{t}{\kappa}}(\|g\|_{L^\infty(X)}+\|\partial_\theta g\|_{L^\infty(X)})
		+(\|\partial_{x_1}\psi^\epsilon_\kappa\|_{L^\infty(X)} +\|\partial_{x_1}g\|_{L^\infty(X)})\int_0^t e^{-\frac{t-s}{\kappa}}ds\bigg)\\\le C_\kappa\bigg( e^{-\frac{t}{\kappa}}(\|g\|_{L^\infty(X)}+\|\partial_\theta g\|_{L^\infty(X)})\\
		+\kappa(3\|\partial_{x_1}g\|_{L^\infty(X)}+2\kappa\|\partial^{2}_{x_1}g\|_{L^\infty(X)} +2\|\partial_{x_1}g\|_{L^\infty(X)})\bigg) ,\end{multline}
	by Corollary \ref{2d x derivative estimates}. Thus, we recover the bound for $\partial_\theta \psi^\epsilon_\kappa$ by \eqref{eq5.13} and Lemma \ref{maxprinciple2} as
	\begin{multline}\notag\|\partial_\theta		\psi^\epsilon_{\kappa}(t,\cdot,0,\cdot)\|_{L^\infty(\mathbb{R}\times [-\pi,0])}\le2C_\kappa (3\|g\|_{L^\infty(X)}+2\kappa\|\partial_{x_1}g\|_{L^\infty(X)} )\\
		+ C_\kappa e^{-\frac{t}{\kappa}}(\|g\|_{L^\infty(X)}+\|\partial_\theta g\|_{L^\infty(X)})
		\\	+C_\kappa \kappa(3\|\partial_{x_1}g\|_{L^\infty(X)}+2\kappa\|\partial^{2}_{x_1}g\|_{L^\infty(X)} +2\|\partial_{x_1}g\|_{L^\infty(X)}) \\\lesssim_\kappa \|g\|_{L^\infty(X)}+\|\partial_\theta g\|_{L^\infty(X)}+\|\partial_{x_1}g\|_{L^\infty(X)} +\|\partial^{2}_{x_1}g\|_{L^\infty(X)} .
	\end{multline}
	Now, using this boundary value at $x_2=0$, we use Lemma 3.3 for the homogeneous solution $(\partial_\theta \psi^\epsilon_\kappa)_{homo}$ of \eqref{eq7} and \eqref{eq8} to observe that  $(\partial_\theta \psi^\epsilon_\kappa)_{homo}$
	has the following uniform-in-$\epsilon$ bound
$$
	\|(\partial_\theta \psi^\epsilon_\kappa)_{homo}(t)\|_{L^\infty( X)}\\ \lesssim_\kappa \|g\|_{L^\infty(X)}+\|\partial_\theta g\|_{L^\infty(X)}+\|\partial_{x_1}g\|_{L^\infty(X)} +\|\partial^{2}_{x_1}g\|_{L^\infty(X)}.
$$
	Then, the Duhamel principle, we obtain a uniform-in-$\epsilon$ bound for the solution $\partial_\theta \psi^\epsilon_\kappa$ of \eqref{eq7} and \eqref{eq8} as
	\begin{multline*}
		\|\partial_\theta \psi^\epsilon_\kappa(t)\|_{L^\infty}\lesssim_\kappa  \| g\|_{C^{(2,0,1)}_{x_1,x_2,\theta}(X)}+t \left(\|\partial_{x_1}\psi^\epsilon_\kappa\|_{L^\infty([0,t]\times X)}+ \|\partial_{x_2}\psi^\epsilon_\kappa\|_{L^\infty([0,t]\times X)}\right)\\ \lesssim_\kappa \| g\|_{C^{(2,0,1)}_{x_1,x_2,\theta}(X)}+ t\| g\|_{C^{(2,1,0)}_{x_1,x_2,\theta}(X)},
	\end{multline*}by Corollary \ref{2d x derivative estimates}. This provides the uniform-in-$\epsilon$ upper bound for $\partial_\theta \psi^\epsilon_\kappa$. 
	
	By the same proof, we prove the following bound for $\partial_{x_1}^l\partial_{x_2}^m\partial_\theta \psi^\epsilon_\kappa$ as it satisfies the same equation for $\partial_\theta \psi^\epsilon_\kappa$ with an additional $\partial_{x_1}^l\partial_{x_2}^m$ derivatives applied to the initial-boundary conditions:
	\begin{equation}\label{2d x theta derivative}
		\|\partial_{x_1}^l\partial_{x_2}^m\partial_\theta \psi^\epsilon_\kappa(t)\|_{L^\infty(X)}
		\lesssim_\kappa \| g\|_{C^{(l+2,m,1)}_{x_1,x_2,\theta}(X)}+ t\| g\|_{C^{(l+2,m+1,0)}_{x_1,x_2,\theta}(X)}.
	\end{equation} Finally, we take one more $\theta$-derivative to \eqref{eq7} and \eqref{eq8} and obtain the system for $\partial^2_\theta \psi^\epsilon_\kappa$ as \begin{equation}\label{eq9}\mathcal{M}(\partial^2_\theta \psi^\epsilon_\kappa)=-\cos\theta \partial_{x_1}\psi^\epsilon_\kappa-2\sin\theta \partial_{x_1}\partial_\theta\psi^\epsilon_\kappa-\sin\theta \partial_{x_2}\psi^\epsilon_\kappa+2\cos\theta \partial_{x_2}\partial_\theta\psi^\epsilon_\kappa,\end{equation} for $\mathcal{M}$ from \eqref{M2} with the initial-boundary conditions: 
	\begin{multline}\label{eq10}
		\partial^2_\theta	{\psi^\epsilon_\kappa} (0,\xi)=\partial_\theta^2 g(\xi),\text{ and }\\
		(\partial_t-\cos\theta\partial_{x_1}) {	\partial^2_\theta	\psi^\epsilon_\kappa}  (t,x_1,0,\theta)=(-\cos\theta\partial_{x_1}-2\sin\theta\partial_{x_1}\partial_\theta)\psi^\epsilon_\kappa (t,x_1,0,\theta)\\+\frac{1}{\kappa}\bigg(	\partial^2_\theta\chi_\kappa(\theta)		\psi^\epsilon_\kappa(t,x_1,0,0)-	\partial^2_\theta\chi_\kappa(\theta)	\psi^\epsilon_\kappa(t,x_1,0,-\pi)-	\partial^2_\theta	\psi^\epsilon_\kappa(t,x_1,0,\theta)\bigg),\\
		\text{ for }x_1\in\mathbb{R}\text{ and }\theta\in[-\pi,0].
	\end{multline}
	As we solved for \eqref{eq7} and \eqref{eq8}, we similarly use the Duhamel principle with the different inhomogeneity and obtain that 
	\begin{multline*}
		\|\partial^2_\theta \psi^\epsilon_\kappa(t)\|_{L^\infty(X)}
		\lesssim_\kappa (\|g\|_{L^\infty}+\|\partial_{x_1}g\|_{L^\infty}+\|\partial_{x_1}\partial_{\theta}g\|_{L^\infty})\\
		+t \left(\|\partial_{x_1}\psi^\epsilon_\kappa\|_{L^\infty}+\| \partial_{x_1}\partial_\theta\psi^\epsilon_\kappa\|_{L^\infty}+ \|\partial_{x_2}\psi^\epsilon_\kappa\|_{L^\infty}+ \|\partial_{x_2}\partial_\theta\psi^\epsilon_\kappa\|_{L^\infty}\right)\\
		\lesssim_\kappa \|g\|_{C^{(1,0,1)}_{x_1,x_2,\theta}(X)}
		+ t\bigg(\|g\|_{C^{(1,0,0)}_{x_1,x_2,\theta}(X)}+\|g\|_{C^{(2,0,0)}_{x_1,x_2,\theta}(X)}+\| g\|_{C^{(3,0,1)}_{x_1,x_2,\theta}(X)}+ t\| g\|_{C^{(3,1,0)}_{x_1,x_2,\theta}(X)}\\+|g\|_{C^{(0,1,0)}_{x_1,x_2,\theta}(X)}+\|g\|_{C^{(1,1,0)}_{x_1,x_2,\theta}(X)}+| g\|_{C^{(2,1,1)}_{x_1,x_2,\theta}(X)}+ t\| g\|_{C^{(2,2,0)}_{x_1,x_2,\theta}(X)}\bigg)
		\lesssim_\kappa (1+t^2)\|g\|_{C^4(X)},
	\end{multline*}by \eqref{2d x theta derivative} and Corollary \ref{2d x derivative estimates}. Thus, we obtain the following proposition:
	\begin{proposition}[Uniform-in-$\epsilon$ estimates for the derivatives]\label{2d derivative estimates}
		For $t\ge 0$ and $\kappa>0$, the solution ${\psi^\epsilon_\kappa} $ to \eqref{eq6} satisfies 
		$$	\left\| \psi^\epsilon_\kappa(t)\right\|_{C^{(1,1,2)}_{x_1,x_2,\theta}( X)}\lesssim_\kappa (1+t^2) \|g\|_{C^4(X)}.$$\end{proposition}
	This completes the uniform-in-$\epsilon$ estimates for the derivatives of $\psi^\epsilon_\kappa$. In the next subsection, we will pass it to the limit $\epsilon \to 0.$

	\subsection{Passing to the limit \texorpdfstring{$\epsilon \to 0$}{}}
	Now we recover the solution to \eqref{eq5} via the definition
	$$\upsilon_\kappa^\epsilon(\xi)=\int_0^\infty e^{-t}{\psi^\epsilon_\kappa} (\lambda t , \xi)\ dt,$$ which is well-defined for any $\epsilon>0$.
	Then we pass to the limit $\epsilon \rightarrow 0$ and obtain the existence of a unique global solution to \eqref{eq4}. We obtain the limit of the approximating sequence $\left\{\upsilon_\kappa^\epsilon\right\}$ as a candidate for a solution by the  compactness (Banach-Alaoglu theorem), which is ensured by the \emph{uniform} estimates of the approximate solutions established in the previous subsections. By the uniform estimate (Proposition~\ref{2d derivative estimates}) and by taking the limit as  $\epsilon \rightarrow 0$, we obtain that a sequence of  $\left\{\upsilon_\kappa^\epsilon\right\}$ converges to $\upsilon_\kappa$  in $  C^{(1,1,2)}_{x_1,x_2,\theta} (X)$. 
	It follows from Proposition~\ref{2d derivative estimates} that $\upsilon_\kappa$ also satisfies the bound
	\begin{equation} \label{L^infty-bound2}
		\|\upsilon_\kappa\|_{ C^{(1,1,2)}_{x_1,x_2,\theta} (X)}  \leq  C_\kappa(1+\lambda^2)\|g\|_{C^4 (X )},
	\end{equation}for some constant $C_\kappa>0$. Therefore, this completes the proof for Proposition \ref{hille2}, which states that there exists a unique $\upsilon_\kappa\in \mathcal{D}(\mathcal{A})$ which solves \eqref{eq4} for any $g\in C(X)$.
	\subsection{Hille-Yosida theorem and the global wellposedness of the adjoint problem}\label{finalhilleyosida}Now we are ready to prove the following proposition on the well-posedness of the adjoint problem \eqref{FP adjoint eq} -\eqref{adjoint periodic}.  
	\begin{proposition}[Well-posedness of the adjoint problem]
	  \label{Thm.sol.adjoint} 
	    Suppose that $\phi_{in}=\phi_{in}(x_1,x_2,\theta)\in C(X)$ satisfies \eqref{forward adjoint initial}. Then there exists a unique solution $\phi \in C(X)$ which solves the adjoint problem \eqref{FP adjoint eq}-\eqref{adjoint periodic}.
	\end{proposition} 
	\begin{remark}\label{remark hypo}
	    It can further be shown that this solution $\phi$ further satisfies the regularity condition $\phi \in C^{1,1,1,2}_{t,x_1,x_2,\theta}(\tilde{\Omega})$ with $\tilde{\Omega}\eqdef  (0,T)\times (-\infty,\infty)\times \{(0,\infty)\times (-\pi,\pi) \cup \{x_2=0\}\times \{(-\pi,-\pi/2)\cup(-\pi/2,0)\cup (0,\pi)\}\}.$  This property is by the hypoellipticity of the operator and it will be proved in Section \ref{sec:hypoellipticity}.
	\end{remark}
\begin{proof}We first observe that Proposition \ref{hille2} provides the sufficient condition for Theorem \ref{markovgen2}, which states that $\mathcal{A}$ is the \textit{Markov} generator. Now, via the Hille-Yosida theorem (Theorem \ref{hilleyosidathm}), we obtain the corresponding \textit{Markov} semigroup $S_\kappa(t)$ to the Markov generator $\mathcal{A}$ for the adjoint equation \eqref{FP adjoint eq} -\eqref{adjoint initial}, and \eqref{adjoint periodic} with the regularized boundary condition  \eqref{regularized adjoint boundarycon}. Now, recall that 
	the solution $\phi_\kappa$ can be written as
	$\phi_\kappa(t,\cdot)=S_\kappa(t) \upsilon_\kappa(\cdot).$ Here note that the semigroup $S_\kappa$ is continuous and $\upsilon_\kappa$ is differentiable. More importantly, we have the following  uniform-in-$\kappa$ estimate for $\phi_\kappa$ by Lemma \ref{maxprinciple2}, since the maximum of $\phi_\kappa$ occurs only at either $\{t=0\}$ or at $\{x_2=0\text { and }\theta\le 0\}$;
	$$\|\phi_\kappa\|_{L^\infty}\le\max\{1, 3\|\phi_{in}\|_{L^\infty(X)}+2\|\partial_{x_1}\phi_{in}\|_{L^\infty(X)}\},$$ since the boundary values at $x_2=0$ and $\theta\le 0$ is bounded by 1 from above. Then, we pass to the weak-$\ast$ limit $\phi_\kappa\stackrel{\ast}{\rightharpoonup} \phi$ as $\kappa\rightarrow 0$ and obtain the solvability of the original adjoint problem \eqref{FP adjoint eq}-\eqref{adjoint periodic}. 
	We recall that the limiting system for the weak-$\ast$ limit $\phi$ as $\kappa\rightarrow 0$ is the same system as that of $\phi_\kappa$ except for the difference in the boundary conditions \eqref{adjoint boundarycon} and \eqref{regularized adjoint boundarycon}. 
	In order to study the limit of the regularized boundary condition for $\phi_\kappa$, we first define  $$\psi_{\kappa,g}(t,x_1,0,\theta)\eqdef \psi_{\kappa}(t,x_1,0,\theta)-\chi_\kappa(\theta)\psi_\kappa(t,x_1,0,0)-(1-\chi_\kappa(\theta))\psi_\kappa(t,x_1,0,-\pi).$$
	Then, by the same $L^\infty$ estimate \eqref{boundary final duhamel 2d}, we obtain that the function $\psi_{\kappa,g}(0,x_1,0,\theta)$ satisfies the bound 
	\begin{equation}\label{final psi kappa g}\notag\|\psi_{\kappa,g}(t,\cdot,0,\cdot)\|_{L^\infty(\mathbb{R}\times [-\pi,\pi])}\le 2e^{-\frac{t}{\kappa}}\|g\|_{L^\infty(X)}+2\kappa\|\partial_{x_1}g\|_{L^\infty(X)} .\end{equation} Therefore, we obtain the bound $$\notag\|\psi_{\kappa}(t,\cdot,0,\cdot)\|_{L^\infty(\mathbb{R}\times [-\pi,\pi])}\le 8\|g\|_{L^\infty(X)}+6\kappa\|\partial_{x_1}g\|_{L^\infty(X)} .$$Also, by taking the limit $\kappa \to 0$ to \eqref{final psi kappa g} we obtain
	$$\|\psi_{\kappa,g}(t,\cdot,0,\cdot)\|_{L^\infty(\mathbb{R}\times [-\pi,\pi])}\to 0,$$ as $\kappa\to 0$ for $t\ge 0$.  Furthermore, for a sequence of small $\delta_\kappa>0$ such that $\delta_\kappa\gg \kappa$ and $\delta_\kappa \to 0$ as $\kappa \to 0,$ we have the limit $$\|\psi_{\kappa,g}(t,\cdot,0,\cdot)\|_{L^\infty(\mathbb{R}\times \{[-\pi,-\pi/2-\delta_\kappa]\cup[-\pi/2+\delta_\kappa,0]\})}\to 0,$$ as $\kappa\to 0$ for $t\ge 0$.  
	This is equivalent to write
$$	\bigg\|\chi_\kappa(\cdot)\psi_\kappa(t,x_1,0,0)+(1-\chi_\kappa(\cdot))\psi_\kappa(t,x_1,0,-\pi)-\psi_{\kappa}(t,\cdot,0,\cdot)\bigg\|_{L^\infty_{\delta_\kappa}}
			\to 0,
$$uniformly, where $L^\infty_{\delta_\kappa}$ stands for $L^\infty(\mathbb{R}\times \{[-\pi,-\pi/2-\delta_\kappa]\cup[-\pi/2+\delta_\kappa,0]\})$ here.
	Thus, we obtain that the weak-$\ast$ limit $\phi$ of $\phi_\kappa$ solves the limiting system \eqref{FP adjoint eq}, \eqref{adjoint initial}, and \eqref{adjoint periodic} with the boundary condition \begin{equation}\notag
		\begin{split}
			\phi(t,x_1,0,\theta)&=\phi(t,x_1,0,-\pi)\ \text{if}\ -\pi\le \theta<-\frac{\pi}{2},\ \text{and}\\
			\phi(t,x_1,0,\theta)&=\phi(t,x_1,0,0)\ \text{if}\ -\frac{\pi}{2}<\theta\le 0.
		\end{split}
	\end{equation}which is the limit of the regularized boundary condition $$\phi_\kappa(t,x_1,0,\theta)=\bigg(\chi_\kappa(\theta)\phi_\kappa(t,x_1,0,0)+(1-\chi_\kappa(\theta))\phi_\kappa(t,x_1,0,-\pi)\bigg).$$ This completes the proof for the solvability of the adjoint problem \eqref{FP adjoint eq}- \eqref{adjoint periodic}.
	\end{proof}
	This completes Proposition \ref{Thm.sol.adjoint} on the proof of the solvability of the adjoint problem \eqref{FP adjoint eq}- \eqref{adjoint periodic}. In the next section, we will use Proposition \ref{Thm.sol.adjoint} to construct a measure-valued solution that corresponds to the weak solution of the original problem \eqref{FP eq}-\eqref{periodic} in the sense of Definition \ref{weaksoldef}. Furthermore, we will also prove that the solution is unique.
	
	\section{A unique weak solution to the original problem by duality}\label{sec: uniqueness}
	We are now ready to prove the main existence of a unique weak solution to the original problem \eqref{FP eq}-\eqref{periodic}. In this section, we use the well-posedness of the adjoint problem from the previous section and construct a unique weak measure-valued solution of the original problem \eqref{FP eq}-\eqref{periodic} via the duality argument which also coincides with the weak measure-valued solution in the sense of Definition \ref{weaksoldef}.
	
	We now prove our main existence and uniqueness theorem (Theorem \ref{mainexistencethm}).\begin{proof}[Proof of Theorem \ref{mainexistencethm}]
		For the proof of the existence of a weak solution, we construct a measure-valued function $f\in C([0,T]; \mathcal{M}_+(X))$ using the unique solution to the adjoint problem. The construction is motivated by Liggett \cite[Definition 1.6]{Liggett}. By Proposition \ref{Thm.sol.adjoint},  for any test function $\psi\in C([0,T];C(X))$, there exists a semi-group $S(\tau)$ such that $S(\tau)\psi(t)$ solves the \textit{backward-in-time} adjoint problem \eqref{backwardintimeproblem}-\eqref{backward initial} with respect to the \textit{backward} initial data $\psi(t)$ for any fixed $t\in[0,T]$.  Recall that we proved the global well-posedness in Proposition \ref{Thm.sol.adjoint} for the \textit{forward-in-time} adjoint problem  \eqref{FP adjoint eq} with the initial-boundary conditions \eqref{adjoint initial}, \eqref{adjoint boundarycon} and \eqref{adjoint periodic} in the previous section in which we reverse the variable $t\mapsto T-t$  on the \textit{backward-in-time} system \eqref{backwardintimeproblem}-\eqref{backward boundary}. By Remark \ref{remark hypo}, $S(\tau)\psi(t)$ is sufficiently regular.  Now, for any given initial data $f_{in}\in\mathcal{M}_+(X)$ of the original problem and any given $\tau\in [0,T]$, we define $f_\tau=f_\tau(\cdot,\cdot,\cdot)\eqdef f(\tau,\cdot,\cdot,\cdot)$ such that, for any $\psi(t)\in C(X)$ and the corresponding unique solution $S(\tau)\psi(t)$ to the \textit{backward} adjoint problem \eqref{backwardintimeproblem}-\eqref{backward initial}, the following dual-identity holds:  \begin{equation}\label{duality}\int \ \psi(t) df_\tau=\int \ S(\tau)\psi(t) df_{in},\ \text{for any fixed }t,\tau\in[0,T].\end{equation} The function $f$ in this definition exists by means of the semi-group  \cite[Definition 1.6]{Liggett}. 
		  Then, by Definition \ref{weaksoldef}, we have that the solution above coincides with the weak solution in the sense of Definition \ref{weaksoldef}, since
		\begin{multline*}0=\iiint_{\mathbb{R}^2_+\times [-\pi,\pi]} dx_1dx_2d\theta \ f(T)\psi(T)-\iiint_{\mathbb{R}^2_+\times [-\pi,\pi]} dx_1dx_2d\theta \ f_{in}\psi_{in}\\=\int_0^T dt\iiint_{\mathbb{R}^2_+\times [-\pi,\pi]} dx_1dx_2d\theta \left[\partial_t \psi +(\cos\theta,\sin\theta)\cdot \nabla_x \psi + \partial_\theta^2 \psi\right]f_r\\
			+\int_0^T dt\int_{\mathbb{R}}dx_1 \ [\partial_t\psi(t,x_1,0,0)+\partial_{x_1}\psi (t,x_1,0,0)]\rho_+(t,x_1)\\+\int_0^T dt\int_{\mathbb{R}}dx_1 \ [\partial_t\psi(t,x_1,0,-\pi)-\partial_{x_1}\psi (t,x_1,0,-\pi)]\rho_-(t,x_1).\end{multline*} 
	This completes the proof of the existence.
		
		For the proof of the uniqueness, we assume that there are two measure solutions $f_1$ and $f_2$ in the sense of Definition \ref{weaksoldef} with the same initial distribution $f_{in}$. Then $w\eqdef f_1-f_2$ is also a solution with the zero initial data in the sense of Definition \ref{weaksoldef} for any test functions $\phi(t)\in C^1([0,T]; C^1(\mathbb{R}^2_+;C^2( [-\pi,\pi])))$. If $w$ is not identically zero, then there exists some $T>0$ such that $w(T)$ is nonzero. Choose any distribution $\phi\in C^1([0,T]; C^1(\mathbb{R}^2_+;C^2( [-\pi,\pi])))$ such that
		$$	\iiint_{\mathbb{R}^2_+\times [-\pi,\pi]} dx_1dx_2d\theta \ w(T)\phi(T)\neq 0.$$ Then by the wellposedness of the adjoint problem of Section \ref{adjointproblem section} which is guaranteed by the Hille-Yosida theorem and the argument in Section \ref{finalhilleyosida}, there exists a solution $\phi$ of the backward-in-$t$ adjoint problem \eqref{backwardintimeproblem}-\eqref{backward boundary}, which has the initial value $\phi(T)$. Then the duality identity \eqref{duality} leads to  $$\iiint_{\mathbb{R}^2_+\times [-\pi,\pi]} dx_1dx_2d\theta \ w(T)\phi(T)=\iiint_{\mathbb{R}^2_+\times [-\pi,\pi]} dx_1dx_2d\theta \ w_0\phi_{in}=0,$$ which is a contradiction. 
	\end{proof}
	This concludes the discussion on the existence of a unique measure solution to the kinetic model for semiflexible polymers. In this next section, we will further study its long-chain asymptotic behavior.

	\section{Hypoellipticity} \label{sec:hypoellipticity}
	In this section, we will discuss the a posteriori regularity of a weak solution that we obtained from the previous section. Our main goal is to prove Theorem \ref{maintheorem}$_{ (\ref{hypoellipticitymaintheorem})}$ in this section.
	The linear Fokker-Planck operator with the standard Laplacian in the velocity variable $\Delta_v$ is known as a hypoelliptic operator \cite{MR3897919,Hormander, Pascucci, Rothschild, MR3237885, MR3436235}. Here we recall H\"ormander's definition of the hypoellipticity:
	\begin{definition}[{\cite[page 477]{MR126071}}]
		A differential operator $P(x,D)$ with coefficients in $C^\infty$ is called \textit{hypoelliptic} if the equation $$P(x,D)u=h$$ only has solutions $u\in C^\infty$ when $h\in C^\infty.$
	\end{definition}
	
	In our case, the standard Laplacian operator is now replaced by the Laplace-Beltrami operator $\Delta_{\frac{v}{|v|}}=\Delta_n$ where $n \in \mathbb{S}^1.$ Using the new parametrization of the variables $\theta \in [-\pi,\pi],$ we obtain that the equation \eqref{FP eq} in terms of the variables $(t,x,\theta)\in [0,T]\times \mathbb{R}^2_+\times [-\pi,\pi]$. In this section, we show that a weak solution $f=f(t,x,\theta)$ in the sense of Definition \ref{weaksoldef} is indeed smooth in the interior of the domain and at the boundary with $x_2=0$ and $\theta \in [-\pi,0).$
	For this, we will prove that the commutator of the vector fields of the differential operator in the sense of H\"ormander's hypoelliptic theory on the second-order differential operator \cite{Hormander} that we will introduce below is non-zero and hence there is a regularity mixing. Therefore, we can conclude that the weak solution is indeed smooth.  
	
	We first introduce H\"ormander's theorem on a sufficient condition for the hypoellipticity:
	\begin{theorem}[{\cite[Theorem 1.1]{Hormander}}]\label{hormandertheorem}
		Let $P$ be written in the form 
		$$ P=\sum_{1}^{r}X_j^2+X_0+c,$$ where $X_0,...,X_r$ denote first order homogeneous differential operators in an open set of $\Omega\in \mathbb{R}^n$ with $C^\infty$ coefficients, and $c\in C^\infty(\Omega).$ Assume that among the operators $X_{j_1},$ $[X_{j_1},X_{j_2}]$, $[X_{j_1},[X_{j_2},X_{j_3}]]$,...,$[X_{j_1},[X_{j_2},[X_{j_3},..., X_{j_k}]]]$,... where $j_i=0,1,...,r$, there exist $n$ which are linearly independent at any given point in $\Omega.$ Then it follows that $P$ is hypoelliptic.
	\end{theorem}
	
	Now we will check if it is eligible for us to use this theorem to prove Theorem \ref{maintheorem}$_{ (\ref{hypoellipticitymaintheorem})}$.
	\begin{proof}[Proof of Theorem \ref{maintheorem}$_{ (\ref{hypoellipticitymaintheorem})}$]
	We first define the interior $\Omega$ of the domain of $[0,T]\times \mathbb{R}^2_+\times [-\pi,\pi]$ and let $\Omega= (0,T)\times (-\infty,\infty)\times (0,\infty)\times (-\pi,\pi).$ 
	We can write the equation of our interest \eqref{FP eq} as 
	$$Pf=(X_1^2+X_0)f=0,$$ where $$X_1\eqdef i\partial_\theta\text{ and }X_0=\partial_t+\cos\theta \partial_{x_1}+\sin\theta \partial_{x_2}.$$
	Then we first observe that the commutator of the vector fields of $X_1 $ and $X_0$ is 
	$$[X_1,X_0]
	=i\partial_\theta[(\cos\theta,\sin\theta)\cdot \nabla_x+\partial_t]-[(\cos\theta,\sin\theta)\cdot \nabla_x+\partial_t]i\partial_\theta
	=i(-\sin\theta,\cos\theta)\cdot \nabla_x.
	$$
	Using this, we further have that
	$$[X_1,[X_1,X_0]]
	=i\partial_\theta[i(-\sin\theta,\cos\theta)\cdot \nabla_x]-[i(-\sin\theta,\cos\theta)\cdot \nabla_x]i\partial_\theta
	=(\cos\theta,\sin\theta)\cdot \nabla_x.
	$$
	Therefore, we observe that
	$$\{X_1, X_0,[X_1,X_0],[X_1,[X_1,X_0]]\}$$forms a linearly independent set at any given point in $\Omega$. By Theorem \ref{hormandertheorem}, we conclude that $P$ is hypoelliptic. 
	
	Now, for any $(t_0,x_0,\theta_0)\in \Omega$, we define a smooth cutoff function $\chi=\chi(t,x,\theta)$ in the whole space $[0,T]\times \mathbb{R}^2\times [-\pi,\pi]$ for a sufficiently small $\epsilon>0$ such that $B_{2\epsilon}(t_0,x_0,\theta_0)$ is included in the interior of $\Omega,$ 
	$\chi(t,x,\theta) \equiv 1$ if $|(t,x,\theta)-(t_0,x_0,\theta_0)|\le \epsilon$, and $\chi(t,x,\theta) \equiv 0$ if $|(t,x,\theta)-(t_0,x_0,\theta_0)|>2\epsilon$.
	Then $\chi f\equiv 0$ if $|(t,x,\theta)-(t_0,x_0,\theta_0)|>2\epsilon$, and by Theorem \ref{hormandertheorem} and the hypoellipticity of $P$ we conclude that $\chi f$ is smooth in $B_{2\epsilon}(t_0,x_0,\theta_0)$. Since $\chi f \equiv f$ in $B_{\epsilon}(t_0,x_0,\theta_0)$, we conclude that for any $(t_0,x_0,\theta_0)\in \Omega$, there exists a sufficiently small $\epsilon>0$ such that $f$ is smooth in $B_{\epsilon}(t_0,x_0,\theta_0)$.

	Now we consider the hypoellipticity on the nonsingular boundary on $x_2=0$. Define the non-singular boundary  $S_b$ of $\Omega$ as
	$$ S_b \eqdef (0,T)\times (-\infty,\infty)\times \{x_2=0\}\times \{(-\pi,-\pi/2)\cup (-\pi/2,0)\cup (0,\pi)\}.$$
	Choose and fix any $(t_0,x_0,\theta_0) \in S_b.$
	Note that near the boundary, $f$ satisfies the equation $$\partial_t +(\cos\theta,\sin\theta)\cdot \nabla_x f =\partial_\theta^2 f.$$ We first consider $\theta_0<0.$ Choose a sufficiently small $\epsilon>0$ such that $[t_0-\epsilon,t_0+\epsilon]\in [0,T]$ and $[\theta_0-\epsilon, \theta_0+\epsilon ]\in (-\pi,-\pi/2)\cup (-\pi/2,0)$. Since $f$ has no boundary restriction for $\theta<0$, we extend the domain of $\mathbb{R}\times \{x_2\ge 0\}\times \{(-\pi,-\pi/2)\cup (-\pi/2,0)\}$ to 
	$\mathbb{R}\times \{x_2\ge -2\epsilon\}\times \{(-\pi,-\pi/2)\cup (-\pi/2,0)\}$. Then $B_\epsilon(t_0,x_0,\theta_0) $ is included in the extended domain. Then by the hypoellipticity of the operator $P$, we obtain that $f$ is locally smooth in $ B_\epsilon(t_0,x_0,\theta_0) $.
	On the other hand, if $\theta_0>0$, then we recall the boundary condition \eqref{FP eq tu} that $f=0$ if $x_2=0$. Thus we extend the solution $f$ to $f_{ext}$ in the whole space such that $f_{ext}\equiv 0$ if $x_2\le 0$. Then consider a sufficiently small $\epsilon>0$ such that $[\theta_0-\epsilon,\theta_0+\epsilon]\in(0,\pi)$ and a local subset $B_\epsilon(t_0,x_0,\theta_0)\subset [0,T]\times \mathbb{R}^2\times (0,\pi)$ near $(t_0,x_0,\theta_0)$. Then $f_{ext}$ is the solution to the same equation in the whole space $[0,T]\times \mathbb{R}^2\times (0,\pi)$, and by the hypoellipticity (Theorem \ref{hormandertheorem}) in the whole space. Thus $f_{ext}$ is locally smooth in $B_\epsilon$ and hence $f$ is smooth in $B_\epsilon\cap \bar{\Omega}.$
	This proves Theorem \ref{maintheorem}$_{ (\ref{hypoellipticitymaintheorem})}$.\end{proof}
	This completes the proof of the hypoellipticity away from the singular boundary and the pathological set. In the next section, we will prove that the solution is indeed H\"older continuous in the domain including the singular boundary.
	
	\section{H\"older continuity near the singular boundary}\label{sec:holder cont}
	In this section, we prove Theorem \ref{maintheorem}$_{ (\ref{holdercontinuity})}$ on the local H\"older continuity of the weak solution $f$ in $x_2$ and $\theta$ variables  in the region including the singular boundary $ (x_2,\theta)=(0,0)$ or $(0,-\pi)$. By the boundary hypoellipticity away from the singular points from the previous section, it suffices to consider the region near the singular points $ (x_2,\theta)=(0,0)$ or $(0,-\pi)$.
	We construct barriers using self-similar type solutions to a stationary problem and derive the maximum principle via the comparison.
	Without loss of generality, we consider the case near the singular point with $x_2=0$ and $\theta=0$ only. The other case $\theta=-\pi$ is similar.

	We start with constructing a supersolution of the equation \eqref{FP eq} near $(x_2,\theta)=(0,0)$. 
	We start with defining the definition of a supersolution to \eqref{FP eq}.
	\begin{definition}\label{supersolutiondef}
		A function $f\in C^1_t([0,T];C^1_{x_1,x_2}(\mathbb{R}\times (0,\infty);C^2_\theta((-\pi,0)\cup(0,\pi))))$ is a supersolution to \eqref{FP eq} near the singular boundary $(x_2,\theta)=(0,0)$ if it satisfies
		$$(\partial_t+(\cos\theta,\sin\theta)\cdot\nabla_x-\partial_\theta^2)f\ge 0.$$
	\end{definition}
	We first introduce preliminary lemmas on the solution to a steady equation, which will be used to construct a supersolution:
	\begin{lemma}[{\cite[Lemma 21]{MR3237885}}]\label{f*0lemma}
		There exists a positive function $f^*_0=f^*_0(x_2,\theta)>0$ on $(0,\infty)\times (-\pi,\pi)$ which solves
		\begin{equation}\label{originalsteadyFP} \theta \partial_{x_2}f^*_0 = \partial_\theta^2 f^*_0,\end{equation} and satisfies $$\lim_{R\rightarrow 0,\ x_2+|\theta|^3=R} f^*_0(x_2,\theta)=\infty.$$ Also, for some $\alpha<0$ with $|\alpha|$ being sufficiently small, $$f^*_0(x_2,\theta)\approx x_2^{\alpha},$$ if $x_2\approx |\theta|^3$ and $x_2+|\theta|^3<\delta$ for some sufficiently small $\delta>0.$ 
	\end{lemma} The function $f^*_0$ in this lemma is constructed as
	\begin{equation}\label{f*0 form}f^*_0(x_2,\theta)=x_2^\alpha M\left(-\alpha; \frac{2}{3};-\frac{\theta^3}{9x_2}\right),\end{equation} where $M=M(a;b;z)$ is the confluent hypergeometric function of the first kind.  Namely, this lemma implies that there exists a positive steady solution $f^*_0$ to \eqref{originalsteadyFP} which blows up at the singular boundary $(x_2,\theta)=(0,0)$. Now we modify the steady solution $f^*_0$ by considering an additional remainder term $R$ such that it is a sub/supersolution of the steady Fokker-Planck equation for the semiflexible polymers \eqref{steadyFP} as below: \begin{lemma}\label{f*lemma}
		There exists a sufficiently small function $R^*(x_2,\theta)$ such that
		$$|R^*|\ll f^*_0 \text{ if }x_2\approx |\theta|^3 \text{ and } x_2+|\theta|^3<\delta,$$ and that $f^*=f^*(x_2,\theta)=f^*_0+R^*$ satisfies\begin{equation}\label{steadyFP}\sin \theta \partial_{x_2}f^* - \partial_\theta^2 f^* =O(x_2^\alpha),\end{equation}and satisfies $$\lim_{R\rightarrow 0,\ x_2+|\theta|^3=R} f^*(x_2,\theta)=\infty,$$ for the same $\alpha<0$ of that in \eqref{f*0 form} with $|\alpha|$ being sufficiently small, and $$f^*(x_2,\theta)\approx x_2^{\alpha},$$  if $x_2\approx |\theta|^3$ and $x_2+|\theta|^3<\delta,$ where $f^*_0$ is given in Lemma \ref{f*0lemma}.
	\end{lemma}
	\begin{proof}We will construct $R$ such that 
		$$|R^*|\ll f^*_0 \text{ if }x_2\approx |\theta|^3 \text{ and } x_2+|\theta|^3<\delta,$$ and 
		$$(\sin\theta-\theta)\partial_{x_2}f^*_0+(\sin\theta \partial_{x_2}-\partial^2_\theta)R^*\ge 0.$$
		Recall that $$f^*_0(x_2,\theta)=x_2^\alpha M\left(-\alpha; \frac{2}{3};-\frac{\theta^3}{9x_2}\right),$$ where $M=M(a;b;z)$ is the confluent hypergeometric function of the first kind. $f^*_0$ is the solution of $$(\theta\partial_{x_2}-\partial^2_\theta)f^*_0=0,$$ by Lemma \ref{f*0lemma}.
		Note that for $x_2\approx |\theta|^3$ and $x_2+|\theta|^3<\delta$, we have  $\sin\theta -\theta =O(\theta^3)\approx O(x_2)$ and $\partial_{x_2}f^*_0\approx x_2^{\alpha-1}$. 
		Thus, $$(\sin\theta-\theta)\partial_{x_2}f^*_0\approx O(x_2^{\alpha}),$$ for $x_2\approx |\theta|^3$ and $x_2+|\theta|^3<\delta.$
		We want that this order $O(x_2^\alpha)$ can absorb all the remainder contribution in terms of $R$.   Define $R$ as $$R^*(x_2,\theta)= x_2^{\alpha+2/3} Q\left(\frac{\theta}{x_2^{1/3}}\right),$$ for some $Q$.
		Then observe that
		\begin{multline*}(\sin\theta \partial_{x_2}-\partial^2_\theta)R^*
			=(\alpha+2/3) \sin \theta x_2^{\alpha+2/3-1}Q\left(\frac{\theta}{x_2^{1/3}}\right) \\-\theta x_2^{\alpha+2/3-4/3} \sin\theta Q'\left(\frac{\theta}{x_2^{1/3}}\right)-x_2^{\alpha+2/3-2/3}Q''\left(\frac{\theta}{x_2^{1/3}}\right)\\
			=(\alpha+2/3) (\theta+O(\theta^3)) x_2^{\alpha+2/3-1}Q -\theta x_2^{\alpha+2/3-4/3}(\theta+O(\theta^3))Q' -x_2^{\alpha+2/3-2/3}Q'' \\
			=-x_2^{\alpha}(Q''+z^2Q' -(\alpha+2/3) z Q +O(\theta^2)(Q+Q')), 
		\end{multline*}for $z=\frac{\theta}{x_2^{1/3}}$ for $x_2\approx |\theta|^3$ and $x_2+|\theta|^3<\delta.$ Then we have
		\begin{multline*}\sin \theta \partial_{x_2}f^* - \partial_\theta^2 f^*=(\sin\theta-\theta)\partial_{x_2}f^*_0+(\sin\theta \partial_{x_2}-\partial^2_\theta)R^*\\
			=-x_2^{\alpha}(x_2^{-\alpha}(\sin\theta-\theta)\partial_{x_2}f^*_0+Q''+z^2Q' -(\alpha+2/3) z Q +O(\theta^2)(Q+Q')).
		\end{multline*}Now we define $Q$ as the Kummer function which is the solution to $$Q''(z)+z^2Q'(z) -(\alpha+2/3) z Q(z)=0.$$
		Note that $O(\theta^2)(Q(z)+Q'(z))$ is sufficiently small if $\theta$ is small. Then, since $x_2^{-\alpha}(\sin\theta-\theta)\partial_{x_2}f^*_0$ is of order 1, $O(\theta^2)(Q(z)+Q'(z))$  is in the smaller scale than $x_2^{-\alpha}(\sin\theta-\theta)\partial_{x_2}f^*_0$ and is being absorbed. Thus, we can conclude that $f^*=f^*_0+R^*$  satisfies  $$\sin \theta \partial_{x_2}f^* - \partial_\theta^2 f^* = O(x_2^\alpha).$$ This completes the proof.
	\end{proof}

	On the other hand, we would also like to construct a part of the supersolution which behaves like a polynomial near the singular boundary. We first introduce the following lemma from \cite{MR3237885} on the existence of a function $\Lambda$ which will be used for the construction of a positive function $Z^*_0$ behaves like a power-law near the singular boundary and satisfies \eqref{Z*equation}.
	\begin{lemma}[{\cite[Claim 3.3.1]{MR3237885}}]\label{claim3.7}
		For any $\alpha \in (0,1/6),$ there exists a solution $\Lambda(\zeta)$ of \begin{equation}\label{3.29}\Lambda''(\zeta)+3\zeta^2\Lambda'(\zeta)-9\alpha \zeta \Lambda(\zeta)=0,\end{equation} with the form 
		$$\Lambda(\zeta)= U(-\alpha,2/3,-\zeta^3),\ \zeta\in \mathbb{R},$$ where we denote as $U(a,b,z)$ the Tricomi confluent hypergeometric function. The function $\Lambda(\zeta)$ has the following properties:
		\begin{itemize}
			\item $\Lambda(\zeta)>0$ for any $\zeta \in \mathbb{R}.$
			\item There exists a positive constant $K_+>0$ such that 
			\begin{equation*}
				\Lambda(\zeta)\sim \begin{cases}
					&K_+|\zeta|^{3\alpha},\ \zeta\rightarrow \infty,\\
					&|\zeta|^{3\alpha},\ \zeta\rightarrow-\infty.
				\end{cases}
			\end{equation*}
			\item The function $\Lambda(\zeta),$ up to a multiplicative constant, is the only solution of \eqref{3.29} which is polynomially bounded for large $|\zeta|$. 	
		\end{itemize} 
	\end{lemma}
		Using the lemma above, we will now construct a positive function $Z^*_0=Z^*_0(y,\zeta)$ which behaves like a power-law and further satisfies \eqref{Z*equation}.
	\begin{lemma}\label{Z0*lemma}For $\alpha \in (0,1/6),$
		there exists a sufficiently small $R_0(y,\zeta)$ such that \begin{itemize}
			\item $|R_0|\ll F_0$ for $|\zeta|^3+|y|\ll 1$ and that \item $Z^*_0(y,\zeta)\eqdef F_0(y,\zeta)+R_0(y,\zeta)$ satisfies \begin{equation}\label{Z*equation}(\partial_t+(\cos\theta,\sin\theta)\cdot\nabla_x-\partial_\theta^2)Z^*_0\left(\frac{x_2}{t^{3/2}},\frac{\theta}{t^{1/2}}\right)\ge 0,\end{equation} for $|\theta|^3+x_2\ll t^{3/2}$, where $F_0$ is given by $$F_0(y,\zeta)=y^\alpha\Lambda(\zeta/(9y)^{1/3}).$$
		\end{itemize}Notice that $Z^*_0>0.$
	\end{lemma}
	\begin{proof}We prove this locally on some interval $t\in [t_0, t_0+1]$ so that $t -t_0\le 1$ and locally on the domain with  $|\theta|^3+x_2\ll t^{3/2}$. Without loss of generality, we can assume that $t_0=0$ and let $t\le 1$. 	Since $\Lambda$ satisfies \eqref{3.29}, the function $F_0$ with the definition $F_0(y,\zeta)=y^\alpha\Lambda(\zeta/(9y)^{1/3})$ satisfies the following identity:
	$$(\partial^2_\zeta -\zeta\partial_y)F_0=0.$$Therefore, we further note that
\begin{equation}\label{F0identity}
	  \left(\partial^2_\zeta+\frac{1}{2}\zeta\partial_\zeta+\left(\frac{3}{2}y -\zeta\right)\partial_y\right)F_0
	  =\left(\frac{1}{2}\zeta\partial_\zeta+\frac{3}{2}y \partial_y\right)F_0
	  =\frac{3}{2}\alpha F_0.
\end{equation}On the other hand, we also notice that, for $\zeta=\theta t^{-1/2}$ and $y=x_2t^{-3/2},$
		\begin{multline}\label{eq64}
			(\partial_t+(\cos\theta,\sin\theta)\cdot\nabla_x-\partial_\theta^2)Z^*_0\left(\frac{x_2}{t^{3/2}},\frac{\theta}{t^{1/2}}\right)	=(\partial_t+\sin\theta\partial_{x_2}-\partial_\theta^2)Z^*_0\left(\frac{x_2}{t^{3/2}},\frac{\theta}{t^{1/2}}\right)\\
			=\left(-\frac{3}{2}\frac{x_2}{t^{5/2}}\partial_y-\frac{1}{2}\frac{\theta}{t^{3/2}}\partial_\zeta+\sin\theta t^{-3/2}\partial_y-t^{-1}\partial_\zeta^2\right) Z^*_0\left(\frac{x_2}{t^{3/2}},\frac{\theta}{t^{1/2}}\right)\\
			=t^{-1}\left(-\frac{3}{2}y\partial_y-\frac{1}{2}\zeta\partial_\zeta+\sin\theta t^{-1/2}\partial_y-\partial_\zeta^2\right)  Z^*_0\left(y,\zeta\right)\\
			=-t^{-1}\left(\partial_\zeta^2+\frac{1}{2}\zeta\partial_\zeta+\left(\frac{3}{2}y-\zeta\right)\partial_y\right)  Z^*_0+t^{-3/2}\left(-\theta+\sin\theta \right)\partial_{y}Z^*_0\\
			=-t^{-1}\left[\frac{3}{2}\alpha F_0+\left(\partial_\zeta^2+\frac{1}{2}\zeta\partial_\zeta+\left(\frac{3}{2}y-\zeta\right)\partial_y\right)  R_0\right]+t^{-3/2}\left(-\theta+\sin\theta \right)\partial_{y}(F_0+R_0)\\
			=-t^{-1}\left[\left(\partial_\zeta^2+\frac{1}{2}\zeta\partial_\zeta+\left(\frac{3}{2}y-\frac{\sin\theta}{\theta}\zeta\right)\partial_y\right)  R_0+\frac{3}{2}\alpha F_0-\frac{\sin\theta}{\theta}\zeta F_0\right],
		\end{multline}by \eqref{F0identity} where $Z^*_0\eqdef F_0+R_0$. Therefore, if we find a function $R_0$ which satisfies
		\begin{equation}\label{R0ineq}\left(\partial_\zeta^2+\frac{1}{2}\zeta\partial_\zeta+\left(\frac{3}{2}y-\frac{\sin\theta}{\theta}\zeta\right)\partial_y\right)  R_0+\frac{3}{2}\alpha F_0-\frac{\sin\theta}{\theta}\zeta F_0\le 0,\end{equation} then $Z^*_0\eqdef F_0+R_0$ will then satisfy \eqref{Z*equation} by \eqref{eq64}. Hence, our goal is to construct $R_0$ that satisfies \eqref{R0ineq} for the given function $F_0(y,\zeta)=y^\alpha\Lambda(\zeta/(9y)^{1/3})$, and the rest of the proof is devoted to the construction.
		
		Now we put an additional ansatz that $R_0(y,\zeta)=y^{\frac{2}{3}+\alpha}\varphi\left(-\frac{\zeta^3}{9y}\right)=y^{\frac{2}{3}+\alpha}\varphi(z)$ for some function $\varphi$ such that $|\varphi(z)|\ll y^{2/3}\Lambda(-z^{1/3})$ for $y\ll 1$ with $z=-\frac{\zeta^3}{9y}$. Then we have
		\begin{multline*}
		    \left(\partial_\zeta^2+\frac{1}{2}\zeta\partial_\zeta+\left(\frac{3}{2}y-\frac{\sin\theta}{\theta}\zeta\right)\partial_y\right)  R_0\\
		    =-\zeta y^{\alpha-\frac{1}{3}}\left(z\frac{d^2}{dz^2}\varphi+\left(\frac{2}{3}-\frac{\sin\theta}{\theta}z\right)\frac{d}{dz}\varphi+\left(\frac{2}{3}+\alpha\right)\frac{\sin\theta}{\theta}\varphi\right)\\+\frac{3}{2}\left(\frac{2}{3}+\alpha\right)y^{\frac{2}{3}+\alpha}\varphi.
		\end{multline*}
		Thus, it suffices to find $\varphi(z)$ such that 
			\begin{multline*}
		 z^{1/3}\left(z\frac{d^2}{dz^2}\varphi+\left(\frac{2}{3}-\frac{\sin\theta}{\theta}z\right)\frac{d}{dz}\varphi+\left(\frac{2}{3}+\alpha\right)\frac{\sin\theta}{\theta}\varphi\right)+\left(1+\frac{3}{2}\alpha\right)y^{\frac{2}{3}}\varphi\\\le \left(-\frac{3}{2}\alpha +\frac{\sin\theta}{\theta}\zeta\right) \Lambda,
		\end{multline*}since $F_0=y^\alpha \Lambda.$ Then this is equivalent to	\begin{multline*}
		 z^{1/3}\left(z\frac{d^2}{dz^2}\varphi+\left(\frac{2}{3}-z\right)\frac{d}{dz}\varphi+\left(\frac{2}{3}+\alpha\right)\varphi\right)+\left(1+\frac{3}{2}\alpha\right)y^{\frac{2}{3}}\varphi\\\le \left(-\frac{3}{2}\alpha +\frac{\sin\theta}{\theta}\zeta\right) \Lambda-z^{1/3}\left(1-\frac{\sin\theta}{\theta}\right)z\frac{d}{dz}\varphi +z^{1/3}\left(1-\frac{\sin\theta}{\theta}\right)\varphi.
		\end{multline*}Here we note that 
		\begin{multline*}\left|z^{1/3}\left(1-\frac{\sin\theta}{\theta}\right)z\frac{d}{dz}\varphi\right|\approx \left|z^{1/3}\left(1-\frac{\sin\theta}{\theta}\right)\varphi\right|\approx \left|z^{1/3}\theta^2\varphi\right|\approx \left|\frac{\theta^3}{x_2^{1/3}}\varphi\right|\\
		\ll \left|\frac{\theta^3}{x_2^{1/3}}y^{2/3}\Lambda\right|\approx \left|\zeta y t \Lambda\right|\ll |\Lambda|,\end{multline*}
		for $y, |\zeta|^3\ll 1$ and $|\theta|\ll t^{1/2}\le 1$. Therefore, it suffices to find $\varphi(z)$ such that $$z\frac{d^2}{dz^2}\varphi+\left(\frac{2}{3}-z\right)\frac{d}{dz}\varphi+\left(\frac{2}{3}+\alpha\right)\varphi\le -2z^{-1/3}\alpha Q.$$ The existence of such a solution $\varphi$ to the ODE has already been studied in \cite{MR3237885} in the variation form \cite[(3.42)]{MR3237885} involving the Kummer and the Tricomi hypergeometric functions $M(a,b,z)$ and $U(a,b,z).$ Furthermore, it was also shown at \cite[(3.43)]{MR3237885} that the solution satisfies the ansatz that $|\varphi(z)|\ll y^{2/3}\Lambda (-z^{1/3})$ for $|\zeta|^3+y\ll 1.$ This completes the proof of the lemma. 
	\end{proof}
	Finally, we prove the H\"older continuity near the singular boundary:
	\begin{proof}[Proof of Theorem \ref{maintheorem}$_{(\ref{holdercontinuity})}$. ]
		Now using the function $Z^*_0$ defined in Lemma \ref{Z0*lemma}, we define
		\begin{equation}\label{eq fhat0}\hat{f}_0(t,x_1,x_2,\theta)=Z^*_0\left(\frac{x_2}{t^{3/2}},\frac{\theta}{t^{1/2}}\right),\end{equation}
		where $Z^*_0$ is given by Lemma \ref{Z0*lemma}. Then define 
		$$\bar{f}_0^\epsilon (t,x_1,x_2,\theta)=\|f_{in}\|_\infty \hat{f}_0(t,x_1,x_2,\theta)+\epsilon f^*(x_2,\theta),$$ where $f^*$ is given by Lemma \ref{f*lemma} which is a singular self-similar profile near the singular boundary $x_2=0$ and $\theta=0$. Then by Lemma \ref{Z0*lemma}, we have
		$$(\partial_t+(\cos\theta,\sin\theta)\cdot\nabla_x-\partial_\theta^2)\hat{f}_0(t,x_1,x_2,\theta)\ge 0.$$
		Therefore, $\bar{f}_0^\epsilon$ is a supersolution of the equation \eqref{FP eq}. Then since $f^*$ satisfies \eqref{steadyFP}, we have $\bar{f}_0^\epsilon$ is a supersolution.   By the asymptotic behavior in Lemma \ref{claim3.7}, note that we also have
		$$f(0,x,\theta)\le \bar{f}_0^\epsilon(0,x,\theta),$$
		for $(x,\theta)\in \mathbb{R}^2_+\times (-\pi,\pi).$
		Define $h=f-\bar{f}^\epsilon_0$. Then $h$ satisfies that
		$$(\partial_t+(\cos\theta,\sin\theta)\cdot\nabla_x-\partial_\theta^2)h\le 0,$$ with the initial condition $h(0,x,\theta)\le 0.$ Then by the maximum principle (cf. Lemma \ref{maxprinciple2}), we obtain that $h(t,x,\theta)\le 0.$   Then by taking $\epsilon \rightarrow 0,$ we obtain that
		$$f(t,x,\theta)\le \|f_{in}\|_\infty \hat{f}_0(t,x,\theta),$$ for all $t>0$ and $(x,\theta)\in \mathbb{R}^2_+\times (-\pi,\pi).$ 
		This implies that 
		$$f(t,x,\theta)\le C(x_2^\alpha + |\theta|^{3\alpha}),$$ where $\alpha$ is the one given in Lemma \ref{claim3.7} and $C$ depends only on $\|f_{in}\|_\infty$.
		
		Now choose any two points $(x_2,\theta)$ and $(z_2,\vartheta)$ both in $\mathbb{R}_+\times [-\pi,\pi].$ Without loss of generality, we assume that 
		$0<\delta_2\eqdef(z_2+|\vartheta|^3)^{1/3}<\delta_1\eqdef(x_2+|\theta|^3)^{1/3}<\delta_0$ for a small constant $\delta_0>0.$ Define
		$$\rho\eqdef (|z_2-x_2|+|\vartheta-\theta|^3)^{1/3}.$$ 
		
		If $\rho^3\ge \frac{1}{100}(\delta_1^3+\delta_2^3),$ then
		we have
		\begin{multline*}
			|f(t,x_1,x_2,\theta)-f(t,x_1,z_2,\vartheta)|\le 	|f(t,x_1,x_2,\theta)|+|f(t,x_1,z_2,\vartheta)|\\\lesssim x_2^\alpha+|\theta|^{3\alpha}+z_2^\alpha+|\vartheta|^{3\alpha}\lesssim (\delta_1^{3\alpha}+\delta_2^{3\alpha})\rho^{3\alpha}, 
		\end{multline*}which results in 
		$$
		|f(t,x_1,x_2,\theta)-f(t,x_1,z_2,\vartheta)|\le C (|x_2-z_2|^\alpha+|\theta-\vartheta|^{3\alpha}).
		$$
		Therefore, we have Theorem \ref{maintheorem}$_{(\ref{holdercontinuity})}$ when $\rho^3\ge \frac{1}{100}(\delta_1^3+\delta_2^3).$
		
		On the other hand, if $\rho^3< \frac{1}{100}(\delta_1^3+\delta_2^3),$ then we have $\rho\ll \delta_1$ and $\delta_1\sim \delta_2$. Without loss of generality, let $\rho\le \frac{1}{10}\delta_2$ and $\delta_1\le 2\delta_2$ and consider the following rescaling:
		$$ f(t,x,\theta)= \delta_1^{3\alpha}  \bar{f}(\tau,X,\Theta),\ x=\delta_2^3X,\ \theta=\delta_2\Theta,\ t=\delta_2^2\tau.$$ Then, $\bar{f}$ solves 	$$ 
		\left(	\partial_\tau+\frac{\cos(\delta_2\Theta)}{\delta_2}\partial_{X_1}\right)\bar{f} + \frac{\sin(\delta_2\Theta)}{\delta_2}\partial_{X_2}\bar{f} - \partial_\Theta^2 \bar{f}=0.
		$$
		For a sufficiently small $\theta =\delta_2\Theta$, we have
		$$ 
		\left(	\partial_\tau+\frac{1}{\delta_2}\partial_{X_1}\right)\bar{f} + \frac{\sin(\delta_2\Theta)}{\delta_2}\partial_{X_2}\bar{f}  - \partial_\Theta^2 \bar{f}=\frac{1-\cos(\delta_2\Theta)}{\delta_2}\partial_{X_1}\bar{f} ,
		$$as $\cos\theta=1+O(\theta^2)$.

		We further make a change of variables $$(\tau,X_1,X_2,\Theta)\mapsto (\tau,X_1'=X_1+\frac{1}{\delta_2}\tau, X_2,\Theta),$$ such that we obtain
		$$ 
		\partial_\tau\bar{f} + \frac{\sin(\delta_2\Theta)}{\delta_2}\partial_{X_2}\bar{f} - \partial_\Theta^2 \bar{f}=\frac{1-\cos(\delta_2\Theta)}{\delta_2}\partial_{X'_1}\bar{f}=\frac{O(\delta_2^2\Theta^2)}{\delta_2}\partial_{X'_1}\bar{f}=\Theta O(\theta)\partial_{X_1'}\bar{f}.
		$$ Note that the coefficient of $\partial_{X_1'}\bar{f}$ is uniformly bounded for $\Theta \approx 1$ and $\theta \ll 1.$  Also, the coefficient of $\partial_{X_2}\bar{f}$ is uniformly bounded  for $\Theta \approx 1$ and $\theta \ll 1$ as 
		$$\frac{\sin(\delta_2\Theta)}{\delta_2}= \Theta +\frac{O(\delta_2^3\Theta^3)}{\delta_2}=\Theta(1+O(\theta^2)).$$
		For the given $\delta_2>0,$ 
		the operator is uniformly-in-$\delta_2$ hypoelliptic in the variables $(\tau, X_1',X_2,\Theta)$ by the H\"ormander's hypoelliptic theorem \cite{Hormander}. This can be shown similarly to Section \ref{sec:hypoellipticity} by choosing
		$$X_0=i\partial_\Theta\text{ and }X_1=\partial_{\tau}+\Theta O(\theta)\partial_{X_1'}+	\Theta(1+O(\theta^2))\partial_{X_2}.$$
		Therefore, for any sufficiently small $$|\Theta|^3+X_2 \le \delta_1^3/\delta_2^3,$$  we have $C>0$
		such that
		$$|\partial_\tau \bar{f}|+|\partial_{X_1'}\bar{f}|+|\partial_{X_2}\bar{f}|+|\partial_\Theta \bar{f}|\le C.$$
		In terms of the original variables, we have
		$$\delta_2^2|\partial_t f|+\delta_2^3 |\partial_{(t+x_1)} f|+\delta_2^3|\partial_{x_2}f|+\delta_2|\partial_\theta f|\le C\delta_1^{3\alpha}.$$
		Finally, we observe that
		\begin{multline*}
			|f(t,x_1,x_2,\theta)-f(t,x_1,z_2,\vartheta)|\\
			 \le |\partial_{x_2} f||x_2-z_2|+|\partial_\theta f||\theta-\vartheta|
			\le C\delta_1^{3\alpha}\left(\frac{|x_2-z_2|}{\delta_2^3}+\frac{|\theta-\vartheta|}{\delta_2}\right)\\
			\le C\delta_1^{3\alpha}\left(\frac{\rho^3}{\delta_2^3}+\frac{\rho}{\delta_2}\right)
			\le C\rho^{3\alpha}\left(\frac{\rho^{3-3\alpha}\delta_1^{3\alpha}}{\delta_2^3}+\frac{\rho^{1-3\alpha}\delta_1^{3\alpha}}{\delta_2}\right)
			\lesssim \rho^{3\alpha},
		\end{multline*} for $\alpha\in(0,1/6)$, since we have $\rho^3< \frac{1}{100}(\delta_1^3+\delta_2^3)$ and $\delta_1\sim \delta_2$. Therefore, we have
		$$
		|f(t,x_1,x_2,\theta)-f(t,x_1,z_2,\vartheta)|\le C (|x_2-z_2|^\alpha+|\theta-\vartheta|^{3\alpha}).
		$$
		Therefore, this completes the proof for Theorem \ref{maintheorem}$_{(\ref{holdercontinuity})}$.
	\end{proof}	This completes the proof of the H\"older continuity of the solution in the domain including the singular boundary. In the next section, we prove that the mass is being accumulated on the boundary $x_2=0$ and the total mass is conserved.
	
	\section{Balance of Mass}\label{sec:mass conservation}
	In this section, we introduce one property of the conservation of total mass that the solution $f$ to the problem enjoys. Namely, we will derive an identity on the balance of mass and will prove Theorem \ref{maintheorem}$_{\eqref{mainconmass1},\eqref{mainconmass2}}$ in this section. 
	
	We introduce the following lemma on the balance of mass identity:
	\begin{lemma}[Balance of mass]\label{lemma.balancemass}
		The solution $f$ to the problem \eqref{FP eq}-\eqref{periodic} and \eqref{boundarycon} in the sense of Definition \ref{weaksoldef} satisfies that
		\begin{equation}\label{balanceof mass eq}\frac{d}{dt}\iiint_{\mathbb{R}^2_+\times [-\pi,\pi]}  f_r(t,x_1,x_2,\theta)\ dx_1dx_2d\theta
			+\int_{\mathbb{R}} \left(\rho_+(t,x_1)+\rho_-(t,x_1)\right)dx_1 =0,\end{equation}
		for any $t\ge 0.$
	\end{lemma}
	\begin{proof}
		We prove the identity using a special test function in \eqref{eq.weaksol}. Namely, we choose $\phi=h(t)$ where $h(t)\in C^1([0,T])$.
		By \eqref{eq.weaksol}, we then have
		\begin{multline}\label{weaksolspecialtest}\int_0^T dt\ \partial_t h\left[\iiint_{\mathbb{R}^2_+\times [-\pi,\pi]} dx_1dx_2d\theta\  f_r
			+ \int_{\mathbb{R}}dx_1 \  \partial_th \left(\rho_+(t,x_1)  +\rho_-(t,x_1)\right)\right]\\=\iiint_{\mathbb{R}^2_+\times [-\pi,\pi]} dx_1dx_2d\theta \ f(T)h(T) -\iiint_{\mathbb{R}^2_+\times [-\pi,\pi]} dx_1dx_2d\theta \ f_{in}h(0) .\end{multline} 
		We therefore obtain from \eqref{weaksolspecialtest} that \eqref{balanceof mass eq} holds in the sense of distributions on $[0,T]$.
	\end{proof}
	As a direct consequence, we can obtain Theorem \ref{maintheorem}$_{\eqref{mainconmass1},\eqref{mainconmass2}}$:
	\begin{proof}[Proof of Theorem \ref{maintheorem}$_{\eqref{mainconmass1},\eqref{mainconmass2}}$]
		Note that Lemma \ref{lemma.balancemass} directly implies Theorem \ref{maintheorem}$_{\eqref{mainconmass2}}$ as
		$$
		f(t,x_1,x_2,\theta)=\rho_+(t,x_1)\delta(x_2)\delta(\theta)+\rho_-(t,x_1)\delta(x_2)\delta(\theta+\pi)+f_r(t,x_1,x_2,\theta),$$ by Definition \ref{weaksoldef}.
		
		In addition, note that Lemma  \ref{lemma.balancemass} implies that for any $t \ge 0$,
		\begin{equation}\notag
			\frac{d}{dt}\iiint_{\mathbb{R}^2_+\times [-\pi,\pi]}  f_r(t,x_1,x_2,\theta)\ dx_1dx_2d\theta
			\le 0,\end{equation} in the sense of distributions, since $\rho_+(t),\rho_-(t)\in\mathcal{M}_+(\mathbb{R})$.
		Therefore, for any $0\le t_1<t_2<\infty$, we have
	\begin{multline*}\int_{\mathbb{R}}dx_1\int_{\{x_2>0\}}dx_2 \int_{-\pi}^\pi d\theta \ f_r(t_1,x_1,x_2,\theta)\\
			\ge \int_{\mathbb{R}}dx_1\int_{\{x_2>0\}}dx_2 \int_{-\pi}^\pi d\theta \ f_r(t_2,x_1,x_2,\theta).\end{multline*}
		This proves Theorem \ref{maintheorem}$_{\eqref{mainconmass1}}.$
	\end{proof}This completes the proof on the mass accumulation on the boundary and the total mass conservation. In the next section, we will prove Theorem \ref{maintheorem}$_{ \eqref{longtimeasymptoticmain}}$ on the long-chain asymptotics as $t\to \infty.$
	
	\section{Long-chain asymptotics}\label{sec:longtime asymptotics}Provided the existence of a unique weak solution, we now study the long-chain asymptotics of the measure solution. In this section, we prove Theorem \ref{maintheorem}$_{ \eqref{longtimeasymptoticmain}}$.
	Suppose that we are given an initial measure $f_{in}\in \mathcal{M}_+(X)$.
	Then we claim that $f_r$ satisfies
	$$\iiint_{\mathbb{R}^2_+\times [-\pi,\pi]} f_r(t,x,\theta)\ dx_1dx_2d\theta \rightarrow 0,$$ as $t\rightarrow \infty.$ In addition, we will also observe that the long-chain asymptotics
	$\rho^+_\infty(x_1),\ \rho^-_\infty(x_1)$ of $\rho^+(t,x_1)$,  $\rho^-(t,x_1)$ satisfy
	$$\int_{\mathbb{R}}dx_1(\rho^+_\infty(x_1-t)+\rho^-_\infty(x_1+t))=\iint_{\mathbb{R}^2_+\times [-\pi,\pi]} dxd\theta f_{in}(x,\theta),$$ 
	for any $f_{in}=f_{in}(x,\theta)\in \mathcal{M}_+(X)$.
	
	The analysis on the long-chain asymptotics is closely related to that of the steady-state solution to a stationary problem. It is crucial to obtain the existence and the uniqueness of a solution to the stationary problem associated to the $t$-dependent problem. 
	In order to guarantee the uniqueness of the steady-state, it is crucial to study the behavior of the stationary solutions for large values of $x_2>0$. In order to understand where the mass of solution concentrates as $t\to \infty$, it is also crucial to guarantee that the mass is not escaping for large values of $x_2$. To this end, we will use the 1-dimensional integrated problem for the mass density $\rho_1$ and its adjoint problem that we have obtained by integrating the system with respect to $x_1$. 
	
	In the next subsection, we will first consider a stationary problem and observe that the mass does not escape for large values of $x_2>0$. 
	\subsection{Asymptotics of the stationary solutions with zero boundary}\label{x2 asymptotics section}As we mentioned in Section \ref{x2infty intro},
	it is crucial to obtain the asymptotics of the solutions $\bar{\varphi}_\infty$ to the stationary equation: 
	\begin{equation}\label{stationary} -\sin\theta \partial_{x_2}\bar{\varphi}=\partial^2_\theta \bar{\varphi}, \text { for } x_2>0,\ \theta \in [-\pi,\pi],\end{equation} with the periodic condition in $\theta$: \begin{equation}\label{periodic stationary} \bar{\varphi}(x_2,-\pi)=\bar{\varphi}(x_2,\pi)\text{ and }\partial_\theta\bar{\varphi}(x_2,-\pi)=\partial_\theta\bar{\varphi}(x_2,\pi), \text{ for }x_2\ge 0.\end{equation}We consider the following boundary condition for the stationary solution $\bar{\varphi}_\infty$ 
	\begin{equation}\notag
		\bar{\varphi}(0,\theta)=1 \text{ for }\theta \in [-\pi,-\pi/2)\cup(-\pi/2,0],
	\end{equation} which is obtained by rescaling each of the boundary conditions \eqref{psiminusboundary} and \eqref{psiplusboundary} without loss of generality and adding them together. Then our goal in this section is to prove that $\bar{\varphi}\equiv 1$ for any $x_2>0$ and $\theta \in [-\pi,\pi]$. This result will be used to prove that no mass is escaping for large values of $x_2$ via the duality argument.  In addition, this will be crucially used to prove the uniqueness of solutions to the stationary problem in the next section, which has the boundary values of \eqref{psiminusboundary} or \eqref{psiplusboundary}.
	
	By redefining the function $\bar{\varphi}$ as $1-\bar{\varphi}$, we can consider the same problem with the \textit{zero boundary} condition \begin{equation}\label{zeroboundary}
		\bar{\varphi}(0,\theta)=0 \text{ for }\theta \in [-\pi,-\pi/2)\cup(-\pi/2,0].
	\end{equation} In this case, our goal is to show that the stationary solution $\bar{\varphi}_\infty$ with the zero boundary condition on $x_2=0$ and $\theta\in [-\pi,-\pi/2)\cup(-\pi/2,0]$ is zero everywhere on $x_2>0$ and $\theta\in[-\pi,\pi]$. 
	The first step for the proof is to prove the non-negativity of a supersolution if the boundary value at $x_2=0$ is non-negative.
	\subsubsection{Non-negativity of supersolutions}\label{sec:nonnegativity}\begin{definition}\label{supersolutionstationary}
		We call $\varsup\in C^2((0,\infty)\times [-\pi,\pi])$ a supersolution to \eqref{stationary}-\eqref{zeroboundary} if $\varsup(0,\theta)\ge \bar{\varphi}(0,\theta)$ for $\theta\in [-\pi,-\pi/2)\cup(-\pi/2,0]$ and $\varsup$ satisfies
		\begin{equation}\label{supersol eq}-\sin\theta \partial_{x_2}\bar{\varphi}\ge \partial^2_\theta \bar{\varphi}.\end{equation}\end{definition}
	Then we claim the following lemma on the non-negativity of a supersolution.
	
	\begin{lemma}
		Suppose $\varsup\in C^2((0,\infty)\times [-\pi,\pi])$ is a supersolution to the stationary equation \eqref{stationary}-\eqref{zeroboundary} in the sense of Definition \ref{supersolutionstationary}.
		Suppose that $\varsup$ is non-negative at $x_2=0$ and $\theta\in [-\pi,-\pi/2)\cup(-\pi/2,0]$. Then $$\varsup(x_2,\theta)\ge 0,$$ for any $x_2>0$ and $\theta\in[-\pi,\pi].$
	\end{lemma}\begin{proof}
		Suppose on the contrary that $\varsup<0$ for some region in $x_2>0$ and $\theta\in[-\pi,\pi].$ Suppose that a point $(x_{2,0}, \theta_0)$ is the local minimum of $\varsup$ in the region. Firstly, if $\theta_0=0$, then we have
		$\partial_\theta^2\varsup (x_{2,0}, \theta_0)\le 0,$ which is a contradiction. If $\theta_0< 0,$ define $\varsup^*(x_2,\theta)=\varsup(x_2,\theta)+\varepsilon x_2$ for a sufficiently small $\varepsilon>0$ such that $\varsup^*$ is still negative at a local minimum $(x_{2,1}, \theta_1)$ of $\varsup^*$ where $\theta_1$ is still negative. Then, we have that $x_{2,1}>0$ and that
	\begin{multline*}
	-\sin\theta_1\partial_{x_2}\varsup^*(x_{2,1}, \theta_1)=-\sin\theta_1(\partial_{x_2}\varsup(x_{2,1}, \theta_1)+\varepsilon)\\
	\ge \partial_\theta^2\varsup -\sin\theta_1\varepsilon\ge -\sin\theta_1\varepsilon>0,
	\end{multline*}
which leads to a contradiction.
		On the other hand, if $\theta_0>0$, then define $\varsup^*(x_2,\theta)=\varsup(x_2,\theta)-\varepsilon' x_2$ for a sufficiently small $\varepsilon'>0$ such that $\varsup^*$ attains its local minimum at $(x_{2,2},\theta_2)$ where $\theta_2>0$. Then we observe that
\begin{multline*}
-\sin\theta_2\partial_{x_2}\varsup^*(x_{2,2}, \theta_2)=-\sin\theta_2(\partial_{x_2}\varsup(x_{2,2}, \theta_2)-\varepsilon')\\
\ge \partial_\theta^2\varsup +\sin\theta_2\varepsilon'\ge \sin\theta_2\varepsilon'>0,
\end{multline*}
		which leads to  a contradiction. 
	\end{proof}
	\subsubsection{Construction of a super-solution}\label{sec.construction.sub} We now construct a super-solution $F_\lambda$ to \eqref{stationary}-\eqref{zeroboundary} as the following:
	\begin{lemma}\label{consubsol}
		Define $F_\lambda=F_\lambda(x_2,\theta)$ as
		$$ F_\lambda(x_2,\theta)= e^{-\lambda x_2}\left(1-\lambda \sin\theta -\frac{\lambda^2}{8}\cos(2\theta)\right),$$ for a sufficiently small $\lambda>0.$ Then $(1-F_\lambda)$ is a super-solution to \eqref{stationary}-\eqref{zeroboundary} in the sense of Definition \ref{supersolutionstationary}.
	\end{lemma}
	\begin{proof}
		Note that
		$$-\sin\theta \partial_{x_2}F_\lambda = \lambda \sin\theta e^{-\lambda x_2}\left(1-\lambda \sin\theta -\frac{\lambda^2}{8}\cos(2\theta)\right),$$ and
		$$\partial_\theta^2 F_\lambda = e^{-\lambda x_2}\left(\lambda\sin\theta +\frac{\lambda^2}{2} \cos(2\theta)\right)=e^{-\lambda x_2}\left(\lambda\sin\theta +\frac{\lambda^2}{2} (1-2\sin^2\theta)\right).$$ Therefore,
		$$-\sin\theta \partial_{x_2}F_\lambda-\partial_\theta^2 F_\lambda=e^{-\lambda x_2}\left(-\frac{\lambda^2}{2}-\frac{\lambda^3}{8}\sin\theta\cos(2\theta)\right)\le 0,$$ if $\lambda>0$ is sufficiently small. Also, note that 
		$$ (1-F_\lambda)(0,\theta) \ge 0=\varphi(0,\theta),\text{ if }\theta\in [-\pi,-\pi/2)\cup(-\pi/2,0].$$This completes the proof.
	\end{proof}
	
	Modifying this super-solution, we will now construct a generalized supersolution as follows:
	\begin{lemma}\label{finalsupsol}
		For a sufficiently small $\eta$ and $\lambda>0$, define 
		$$\varsup(x_2,\theta)= 1+\eta -e^{-\lambda (x_2-\delta)}\left(1-\lambda \sin\theta -\frac{\lambda^2}{8}\cos(2\theta)\right),$$ for $\delta>0.$ Then $\varsup$ satisfies \eqref{supersol eq} and $\varsup(\delta,\theta)\ge \frac{\eta}{2},$ for a sufficiently small $\lambda>0$.
	\end{lemma}The proof is straightforward by Lemma \ref{consubsol} and we omit it.
	\subsubsection{Regularity}\label{sec:longtimeholder}Our final ingredient is to prove that a stationary solution to \eqref{stationary}-\eqref{stationary periodic} is locally H\"older continuous near $(0,0)$ and $(0,-\pi)$ and is smooth elsewhere. The proof for the smoothness would involve the local hypoellipticity of the stationary operator as in Section \ref{sec:hypoellipticity}. In this section, we provide the proof for the H\"older continuity near the singular points, which would involve the construction of a supersolution using a self-similar type profile.
	\begin{lemma}[Local H\"older continuity near the singular points]\label{holdercont stationary}
		Suppose that $\varinf$ is a solution to \eqref{stationary} with the zero boundary condition \eqref{zeroboundary} and the periodicity \eqref{periodic stationary}. Then $\varinf\in C^{0,\alpha}_{x_2, loc}(\mathbb{R}_+; C^{0,3
	\alpha}_{\theta,loc}[-\pi,\pi]),$  for any $\alpha \in (0,1/6)$.
	\end{lemma}
	\begin{proof}Without loss of generality, we prove the H\"older continuity near the singular point $(0,0)$ only. The proof for the other point $(0,-\pi)$ is similar. 
		The proof consists of the constructions of two power-law type profiles, where one decays fast near $(0,0)$ and the other blows up fast near $(0,0)$. We denote the former as $\hat{\varphi}$ and the latter $\tilde{\varphi}$. Note that $\hat{\varphi}$ and  $\tilde{\varphi}$ have already been constructed in \eqref{eq fhat0} and Lemma \ref{f*lemma} as $\hat{f}_0$ and $f^*$, respectively.
		Therefore, the $\varsup$ in the following definition
		$\varsup = \hat{\varphi}+\varepsilon\tilde{\varphi},$ is a supersolution and we can do the comparison and conclude that 
		$\varinf\le \varsup.$ By letting $\varepsilon\rightarrow 0$, we obtain the H\"older continuity and the rest of the proof follows similarly to the proof of Theorem \ref{maintheorem}$_{\eqref{holdercontinuity}}$.
	\end{proof}
	
	\subsubsection{Main lemma}\label{sec.main.stationaryasymp}
	Now we are ready to prove our main lemma:
	\begin{lemma}\label{lemma.eventually zero}
		Suppose that $\varinf$ is a solution to \eqref{stationary} with the zero boundary condition \eqref{zeroboundary} and the periodicity \eqref{periodic stationary}. Then we have
		$$\varinf(x_2,\theta)=0,\text{ for any }x_2>0,\ \theta\in[-\pi,\pi].$$
	\end{lemma}
	\begin{proof}
		By Lemma \ref{holdercont stationary}, we have that $\varinf$ is H\"older continuous near $(x_2,\theta)=(0,0)$ and $(0,-\pi)$. Suppose the zero boundary condition \eqref{zeroboundary}. Then for a fixed $\varepsilon>0$, there exists a sufficiently small $\delta>0$ such that $$|\varinf(\delta,\theta)|\le \varepsilon, \text{ for any} \ \theta \in [-\pi,\pi].$$
		Now observe that the supersolution $\varsup$ defined in Lemma \ref{finalsupsol} as 	$$\varsup(x_2,\theta)= 1+\eta -e^{-\lambda (x_2-\delta)}\left(1-\lambda \sin\theta -\frac{\lambda^2}{8}\cos(2\theta)\right)$$ is a periodic supersolution for \eqref{stationary} in the rectangular domain $x_2\in [\delta,R]$ for a sufficiently large $R>0$ by Lemma \ref{finalsupsol} and that
		$$\varsup(R,\theta)\ge 1+\frac{\eta}{2}\ge 1 \ge  \varinf(R,\theta).$$ Then we can do the comparison in the bounded rectangular region $x_2\in [\delta,R]$ and obtain that
		$$\varinf\le \varsup,\text{ for }x_2\in [\delta,R],\ \theta\in[-\pi,\pi],$$ and $R$ can be arbitrarily large. Since $\lambda$ and $\eta$ can be arbitrarily small, we obtain that as $\lambda, \eta \rightarrow 0^+$, 
		$$\varinf \le  1+\eta -e^{-\lambda (x_2-\delta)}\left(1-\lambda \sin\theta -\frac{\lambda^2}{8}\cos(2\theta)\right)\rightarrow 0.$$ Thus, $\varinf=0$ everywhere.
	\end{proof}
	By reverting back $\bar{\varphi}\mapsto 1-\bar{\varphi}$,
	the discussion with the stationary equation is complete. In the next subsection, we will discuss the $t$-dependent adjoint problems.
	\subsection{Long-chain asymptotics for the \texorpdfstring{$\lowercase{t}$}{}-dependent problems} Equipped with the knowledge on the solutions to the stationary solution $\bar{\varphi}_\infty$ with the zero boundary values \eqref{zeroboundary}, we will now consider the stationary problem with the general boundary condition \eqref{psiminusboundary} and \eqref{psiplusboundary}, whose solutions will then be the asymptotics of the $t$-dependent adjoint problems as $t\to \infty$. 
	
	By integrating the 2-dimensional adjoint equation \eqref{FP adjoint eq} with respect to $x_1$ variable, we obtain the adjoint equation \eqref{FP modified adjoint eq} for the mass density $\rho_1$ which depends only on $t,x_2,\text{and }\theta$ variables. In this subsection, we consider the proof of Theorem \ref{maintheorem}$_{ \eqref{longtimeasymptoticmain}}$ in terms of the mass density $\rho_1$ and its adjoint $\varphi$, as it suffices to show that
	$$\int_0^\infty \int_{-\pi}^{\pi}\rho_1(t,x_2,\theta)\ dx_2d\theta \to 0.$$ 
	As in \eqref{FP modified adjoint eq}, we consider the forward-in-$t$ adjoint problem for a test function $\varphi=\varphi(t,x_2,\theta)$ which is a solution to
	\begin{equation}\label{timedependentadjointeq again}\partial_t \varphi -\sin\theta \partial_{x_2}\varphi =\partial_{\theta}^2\varphi,\end{equation} with \begin{equation}\notag \varphi(t,0,\theta)=\begin{cases}&\varphi(0,0,-\pi)=\alpha, \ \theta\in[-\pi,-\pi/2)\\
			&\varphi(0,0,0)=\beta,\  \theta\in(-\pi/2,0],\end{cases}\end{equation} for some non-negative $\alpha$ and $\beta$ for any $t\ge 0$ with the periodic boundary condition \eqref{modified periodic}. Here note that we already inverted the direction of $t$ via $t\mapsto T-t$. 
	
	We write a solution $\varphi(t,x_2,\theta)$ in the form of
	\begin{equation}\label{longtimedecomp}\varphi(t,x_2,\theta)=\alpha \psi_-(t,x_2,\theta)+\beta \psi_+(t,x_2,\theta)+\psi_0(t,x_2,\theta),\end{equation}
	where $\psi_\pm$ and $\psi_0$ are continuous and each solves \eqref{timedependentadjointeq again} with the boundary conditions
	\begin{equation}\label{psiminusboundary}\psi_-(t,0,\theta)=\begin{cases}&1, \ \theta\in[-\pi,-\pi/2),\\
			&0,\  \theta\in(-\pi/2,0],\end{cases}\end{equation}
	\begin{equation}\label{psiplusboundary}\psi_+(t,0,\theta)=\begin{cases}&0, \ \theta\in[-\pi,-\pi/2),\\
			&1,\  \theta\in(-\pi/2,0],\end{cases}\end{equation}
	and \begin{equation}\label{psizeroboundary}\psi_0(t,0,\theta)=0, \ \theta\in[-\pi,-\pi/2)\cup(-\pi/2,0],
	\end{equation} and each satisfies the periodic boundary condition \eqref{modified periodic}. 
	Then we claim that the following proposition holds:
	\begin{proposition}\label{mainlongtimelemma}
		Suppose $\varphi$ solves \eqref{timedependentadjointeq again}-\eqref{psizeroboundary} with the periodic boundary condition \eqref{modified periodic}. Then we have 
			$\psi_\pm(t,x_2,\theta)\to \psi^\pm_\infty$ and $\psi_0(t,x_2,\theta)\to \psi^0_\infty\equiv 0$  as $t\rightarrow \infty$, where $\psi^\pm_\infty=\psi^\pm_\infty(x_2,\theta)$ are the stationary solutions to \eqref{stationary} with the boundary conditions \begin{equation}\label{stationary psiminusboundary}\psi^-_\infty(0,\theta)=\begin{cases}&1, \ \theta\in[-\pi,-\pi/2),\\
				&0,\  \theta\in(-\pi/2,0],\end{cases}\end{equation} and \begin{equation}\label{stationary psiplusboundary}\psi^+_\infty(0,\theta)=\begin{cases}&0, \ \theta\in[-\pi,-\pi/2),\\
				&1,\  \theta\in(-\pi/2,0],\end{cases}\end{equation}with the periodic boundary condition \eqref{periodic stationary}. \end{proposition}
	Then using the duality identity \eqref{duality}, Proposition \ref{mainlongtimelemma} implies the following corollary:
	\begin{corollary}\label{mainlongtimecorollary}
		For all $f_{in}=f_{in}(x,\theta)\in \mathcal{M}_+(X)$, we have the convergence 
		$$\iiint_{\mathbb{R}^2_+\times [-\pi,\pi]}  f_r(t,x,\theta)\ dxd\theta=\iint_{(0,\infty)\times [-\pi,\pi]} \rho_1(t,x_2,\theta)\ dx_2d\theta\rightharpoonup 0,$$ as $t\rightarrow \infty.$
	\end{corollary}As a direct result of Corollary \ref{mainlongtimecorollary}, we obtain Theorem \ref{maintheorem}\eqref{longtimeasymptoticmain}. Therefore, our goal in the next subsections is to prove Proposition \ref{mainlongtimelemma}. Without loss of generality, we consider the proof for $\psi_+$ with the boundary condition \eqref{psiplusboundary} only. The other cases with \eqref{psiminusboundary}-\eqref{psizeroboundary} can be proved similarly. Especially, we have seen that $\psi^0_\infty$, the limit of $\psi_0$, is zero for any $x_2>0$ and $\theta\in [-\pi,-\pi/2)\cup(-\pi/2,0]$ by Section \ref{x2 asymptotics section}. The proof will depend on the analysis of the stationary problem associated to the time dependent problem \eqref{timedependentadjointeq again}-\eqref{psizeroboundary} with \eqref{modified periodic}.

	\subsection{Solutions to the stationary problems}
	In this subsection, we prove the existence and the uniqueness of the solutions $\psi^\pm_\infty$ to the stationary equation \eqref{stationary} with \eqref{stationary psiminusboundary}-\eqref{stationary psiplusboundary} and \eqref{periodic stationary}. 
	
	Suppose that $\psi^+_\infty$ solves \eqref{stationary} with \eqref{stationary psiplusboundary} and \eqref{periodic stationary}. Then we first remark that a sufficiently regular solution $\psi^+_\infty$ satisfies the following a priori bounds:
	\begin{lemma}[A priori bounds]\label{apriorilemma}
		Suppose $\psi^+_\infty\in C^2((0,\infty)\times [-\pi,\pi])$ solves \eqref{stationary} with \eqref{stationary psiplusboundary} and \eqref{periodic stationary}. Then $\psi^+_\infty$ satisfies
		$0\le \psi^+_\infty\le 1,$ for any $x_2>0$ and $\theta \in [-\pi,\pi]$.
	\end{lemma}
	\begin{proof}
		We first prove that $\psi^+_\infty\ge 0$ for any $x_2>0$ and $\theta \in [-\pi,\pi]$. 	Suppose on the contrary that $\psi^+_\infty<0$ for some region. Suppose that a point $(x_{2,0}, \theta_0)$ is the local minimum of $\psi^+_\infty$. Firstly, if $\theta_0=0$, then we have
		$\partial_\theta^2\psi^+_\infty (x_{2,0}, \theta_0)\le 0,$ which is a contradiction. If $\theta_0< 0,$ define $\psi^{+,*}_\infty(x_2,\theta)=\psi^+_\infty(x_2,\theta)+\varepsilon x_2$ for a sufficiently small $\varepsilon>0$ such that $\psi^{+,*}_\infty$ is still negative at a local minimum $(x_{2,1}, \theta_1)$ of $\psi^{+,*}_\infty$ where $\theta_1$ is still negative. Then, we have that $x_{2,	1}>0$ and that
\begin{multline*}-\sin\theta_1\partial_{x_2}\psi^{+,*}_\infty(x_{2,1}, \theta_1)=-\sin\theta_1(\partial_{x_2}\psi^+_\infty(x_{2,1}, \theta_1)+\varepsilon)\\
			\ge \partial_\theta^2\psi^+_\infty -\sin\theta_1\varepsilon\ge -\sin\theta_1\varepsilon>0,\end{multline*}which leads to a contradiction.
		On the other hand, if $\theta_0>0$, then define $$\psi^{+,*}_\infty(x_2,\theta)=\psi^+_\infty(x_2,\theta)-\varepsilon' x_2$$ for a sufficiently small $\varepsilon'>0$ such that $\psi^{+,*}_\infty$ attains its local minimum at $(x_{2,2},\theta_2)$ where $\theta_2>0$. Then we observe that
	\begin{multline*}-\sin\theta_2\partial_{x_2}\psi^{+,*}_\infty(x_{2,2}, \theta_2)=-\sin\theta_2(\partial_{x_2}\psi^+_\infty(x_{2,2}, \theta_2)-\varepsilon')\\
			\ge \partial_\theta^2\psi^+_\infty +\sin\theta_2\varepsilon'\ge \sin\theta_2\varepsilon'>0,\end{multline*}which leads to a contradiction. Therefore, we obtain that $\psi^{+}_\infty\ge 0$. 
		
		On the other hand, we define $\bar{\psi}^+_\infty\eqdef 1- \psi^+_\infty.$ Then $\bar{\psi}^+_\infty$ also solves the same stationary equation \eqref{stationary} with the boundary conditions \eqref{stationary psiminusboundary} and \eqref{periodic stationary}. Then by the same proof, we obtain that $\bar{\psi}^+_\infty\ge 0$, which implies that $\psi^+_\infty\le 1 $ for any $x_2>0$ and $\theta \in [-\pi,\pi]$. This completes the proof.
	\end{proof}
	In addition, we can obtain the uniqueness of the stationary solution $\psi^+_\infty$:
	\begin{lemma}[Uniqueness]\label{uniquestationary}
		Suppose $\psi^+_\infty\in C^2((0,\infty)\times [-\pi,\pi])$ solves \eqref{stationary} with \eqref{stationary psiplusboundary} and \eqref{periodic stationary}. Then $\psi^+_\infty$ is unique.
	\end{lemma}
	\begin{proof}
		Suppose that there are two solutions $\psi^{+,1}_\infty$ and $\psi^{+,2}_\infty$ to \eqref{stationary} with \eqref{stationary psiplusboundary} and \eqref{periodic stationary}. Then define the difference $$\psi^{+,d}_\infty\eqdef \psi^{+,1}_\infty-\psi^{+,2}_\infty.$$ Then $\psi^{+,d}_\infty(0,\theta)=0$ for $\theta \in [-\pi,-\pi/2)\cup (-\pi/2,0]$. Thus, $\psi^{+,d}_\infty$ solves \eqref{stationary} with the zero boundary condition \eqref{zeroboundary} and the periodic boundary \eqref{periodic stationary}. Then by Lemma \ref{lemma.eventually zero}, we have $\psi^{+,d}_\infty=0$ for any $x_2>0$ and $\theta\in [-\pi,\pi]$. This completes the proof.
	\end{proof}
	Finally, we prove the existence of a smooth stationary solution $\psi^+_\infty$. The proof is based on the analysis in Section \ref{sec: 1d wellposedness}. 
	\begin{lemma}[Existence of a locally smooth solution]
		There exists a solution $\psi^+_\infty\in C^\infty_{loc}(\{(0,\infty)\times [-\pi,\pi]\} \cup \{\{x_2=0\}\times \{(-\pi,-\pi/2)\cup(-\pi/2,0)\cup (0,\pi)\}\})$ to the stationary problem \eqref{stationary} with \eqref{stationary psiplusboundary} and \eqref{periodic stationary}. 
	\end{lemma}
	\begin{proof}
		 For the proof of the existence, we regularize the Laplace-Beltrami operator to a discretized operator $Q^\epsilon$ defined in \eqref{discretized}. Define the operator $\mathcal{L}^\epsilon = \sin\theta \partial_{x_2}+Q^\epsilon.$
		In order to solve the equation $(\sin\theta \partial_{x_2}+\partial^2_\theta) \psi^+_\infty=0,$ we consider the regularized equation
		\begin{equation}\label{stationary regularized problem}\mathcal{L}^\epsilon \psi^+_{\infty,\epsilon} =\mu_0 \psi^+_{\infty,\epsilon},
		\end{equation} for some sufficiently small $\mu_0>0$. 
		We pass to the limit $\mu_0\to 0$ later after showing the existence for each fixed $\mu_0>0$. Then in order to prove the existence of a solution to \eqref{stationary regularized problem} with the boundary condition \eqref{stationary psiplusboundary}, it suffices to prove the existence of a solution $\bar{\psi}^+_{\infty,\epsilon}=\bar{\psi}^+_{\infty,\epsilon}(t,x_2,\theta)$ to the following $t$-dependent problem
		\begin{equation}\label{tdependent regularized problem}\partial_t  \bar{\psi}^+_{\infty,\epsilon} = \mathcal{L}^\epsilon \bar{\psi}^+_{\infty,\epsilon}, \end{equation}
		with the initial-boundary condition 
		\begin{align*}
			\bar{\psi}^+_{\infty,\epsilon}(0,x_2,\theta)&=0
			,\text { for }x_2>0,\ \theta \in [-\pi,\pi],\\
			\bar{\psi}^+_{\infty,\epsilon}(t,0,\theta)&=\bar{\psi}^+_{\infty,\epsilon}(0,0,\theta)=\begin{cases}&0, \ \theta\in[-\pi,-\pi/2),\\
				&1,\  \theta\in(-\pi/2,0],\end{cases}
		\end{align*} with the periodic boundary condition \eqref{modified periodic}.
		This is via the relationship
		$$\psi^+_{\infty,\epsilon}(x_2,\theta)\equiv \int_0^\infty e^{-t}\bar{\psi}^+_{\infty,\epsilon}\left(\frac{t}{\mu_0},x_2,\theta\right)dt.$$
		Now, we further regularize the boundary condition as 
		\begin{multline}\label{eq111}\partial_t \bar{\psi}^+_{\infty,\epsilon,\kappa}(t,0,\theta) =\frac{1}{\kappa}\bigg(\chi_\kappa(\theta)\bar{\psi}^+_{\infty,\epsilon,\kappa}(t,0,0)\\+(1-\chi_\kappa(\theta))\bar{\psi}^+_{\infty,\epsilon,\kappa}(t,0,-\pi)-\bar{\psi}^+_{\infty,\epsilon,\kappa}(t,0,\theta)\bigg),\text { for } \theta \in [-\pi,0],\end{multline} where
		a smooth function $\chi_\kappa$ is defined the same as \eqref{smoothchi}. 
		By plugging in $\theta=0$ and $\theta=-\pi$ to \eqref{eq111}, we obtain that 
		$\partial_t \bar{\psi}^+_{\infty,\epsilon,\kappa}(t,0,\theta)=0$  for $\theta = 0$ and $=-\pi.$
		Thus, we have
		\begin{equation}\label{eq113}\partial_t \bar{\psi}^+_{\infty,\epsilon,\kappa}(t,0,\theta) =\frac{1}{\kappa}\left(1-\chi_\kappa(\theta)-\bar{\psi}^+_{\infty,\epsilon,\kappa}(t,0,\theta)\right),\text { for } \theta \in [-\pi,0].\end{equation} Then the existence of the solution $\bar{\psi}^+_{\infty,\kappa}$ in the limit $\epsilon\to 0$ is given by Section \ref{sec:1d intro} - Section \ref{sec:limit epsilon to zero} with $g\equiv 0$. Remark that the maximum of $\bar{\psi}^+_{\infty,\epsilon,\kappa}$ in Lemma \ref{maxprinciple} now occurs only on the boundary values $x_2=0$ with $\theta\in[-\pi,0]$ since the initial condition $g$ is $\equiv 0$. This gives the uniform bound of 
		$$\|\bar{\psi}^+_{\infty,\kappa}\|_{L^\infty([0,T]\times S)}\le 1.$$ Then in the weak-limit as $\kappa\to 0,$ we obtain a weak solution $\psi^+_\infty$ to the stationary solution. 
		Now the local smoothness of the weak solution away from the singular boundary $\{x_2=0\}\times \{\theta=0,-\pi\}$ and the pathological set $\{x_2=0\}\times \{\theta=-\pi/2\}$ is obtained directly from the hypoellipticity of the operator by the same arguments in Section \ref{sec:hypoellipticity}. This completes the proof.
	\end{proof}
	Similarly, we can prove the existence and the uniqueness of the other limits $\psi^-_{\infty}$ and $\psi^0_\infty=0$ using the other boundary conditions \eqref{stationary psiminusboundary} and \eqref{zeroboundary}. 
	Lastly, we prove that $\psi^+_\infty(x_2,\theta)\to \frac{1}{2}$ as $x_2\to \infty.$
	\begin{lemma}\label{stationaryx2infty}The solution $\psi^+_\infty$ satisfies the limit
		$\psi^+_\infty(x_2,\theta)\to \frac{1}{2}$ as $x_2\to \infty,$ for any $\theta\in [-\pi,\pi]$.
	\end{lemma}
	\begin{proof}
		We define $\psi_n(x_2,\theta) = \psi^+_\infty(x_2+n,\theta).$ Note that $0\le \psi_n \le 1$, by Lemma \ref{apriorilemma}, and we have $\psi_n \to \bar{\psi},$ as $n\to \infty$ where $\bar{\psi}$ solves \eqref{stationary} for $x_2\in \mathbb{R}$ and $\theta\in[-\pi,\pi]$ with the periodicity \eqref{periodic stationary}. By Lemma \ref{apriorilemma}, we also have 
		$0\le \bar{\psi}\le 1.$	Define 
		$$\limsup_{x_2\to \infty}\bar{\psi}(x_2,\theta)=m_+\in [0,1]\text{ and }\liminf_{x_2\to \infty}\bar{\psi}(x_2,\theta)=m_-\in [0,1].$$
		We claim that $m_-=m_+$. To this end, we first observe that there exists a subsequence $\{\lambda_k\}$ for $\lambda_k\to \infty$ as $k\to \infty$ such that 
		$$\bar{\psi}(x_2+\lambda_k,\theta)\to v(x_2,\theta),\text { as }k\to \infty,$$ where $$\min _{\theta \in[-\pi,0]} v(x_2,\theta) = m_- \text{ and }\max _{\theta \in[-\pi,0]} v(x_2,\theta) = m_+, \text{ for any }x_2\in\mathbb{R}.$$ Then by the maximum principle in Lemma \ref{maxprinciple}, we obtain that $m_-=m_+$.

		Now we prove that $m_-=m_+=\frac{1}{2}.$ For this, we use the symmetrization.
		Define $$\tilde{\psi}(x_2,\theta) = \psi^+_\infty(x_2,\theta)+\psi^+_\infty(x_2,\pi-\theta).$$ Note that $\psi^+_\infty(x_2,\theta)$ satisfies the boundary condition \eqref{stationary psiplusboundary} and $\psi^{ref}(x_2,\theta)\eqdef \psi^+_\infty(x_2,\pi-\theta)$ satisfies the boundary condition \eqref{stationary psiminusboundary}. Therefore, $\tilde{\psi}(x_2,\theta)$ satisfies the boundary condition that 
		$\tilde{\psi}(x_2,\theta)=1$ if $\theta \in [-\pi,0].$ Then by Lemma \ref{lemma.eventually zero} on $1-\tilde{\psi}$ we obtain that $\tilde{\psi}\to 1=2m_+$. Therefore, we obtain that 
		$\psi^+_\infty\to \frac{1}{2}$ as $x_2\to \infty$. The same proof also holds for $\psi^-_\infty$ and we get $\psi^-_\infty\to \frac{1}{2},$ as $x_2\to \infty$.
	\end{proof}
	
	Another ingredient for the proof of Proposition \ref{mainlongtimelemma} is on the diffusive property of the Fokker-Planck operator. Namely, we prove the following positivity of a solution:
	\begin{lemma}\label{subsolution}Suppose $V=V(t,x_2,\theta)$ is a solution to \eqref{timedependentadjointeq again} and \eqref{modified periodic} with $V(0,x_2,\theta)\ge 1 $ if $(x_2,\theta)\in B_R(x_{2,0},\theta_0)$, for a sufficiently small $R>0$ and some $x_{2,0}\ge 0$ and $\theta_0\in [-\pi,\pi]$. Then there exist $t_*=t_*(R)>0$ and $C_0=C_0(R)>0$ such that $V(t_*,x_2,\theta)\ge C_0>0$ for any $(x_2,\theta)\in B_{2R}(x_{2,0},\theta_0)$.
	\end{lemma}
	\begin{proof}
		For the proof, we construct a subsolution $W=W(t,x_2,\theta)$ which satisfies
		\begin{equation}\label{subsol eq}\partial_t W -\sin\theta \partial_{x_2}W-\partial^2_\theta W \le 0,\end{equation} with $W(0,x_2,\theta)=1$ on $B_R(x_{2,0},\theta_0)$ and $W(t_*,x_2,\theta)\ge C_0$ for any $|x_2-x_{2,0}|\le R/{10^6}$ and $|\theta-\theta_0|\le 2R$. We look for a sub-solution $W$ of the form 
		$$W(t,x_2,\theta)= e^{-\sigma t} V(x_2+t\sin\theta ,\theta),$$ for some sufficiently large $\sigma>0$ and $W\in C^2_{x,\theta}$. Then \eqref{subsol eq} implies 
		$$ -\sigma V \le \partial^2_\theta V -t\sin\theta \partial_{x_2}V +t\cos\theta \partial_{x_2}\partial_\theta V +t^2\cos^2\theta \partial_{x_2}^2 V.$$ Then we define $V$ as
		\begin{equation}\label{eq for V}V(x_2,\theta) \eqdef \frac{1}{2}\cos(\sigma^{1/4} (\theta-\theta_0)) \cos(\sigma^{1/4}(x_2-x_{2,0}))+\frac{1}{2},\end{equation} with $\sigma =\left(\frac{2\pi}{R}\right)^4>0$ sufficiently large (with a sufficiently small $R>0$) so that \eqref{eq for V} holds for any $0\le t \le t_*$ for some $ t_*=t_*(R)>0$. Then by the maximum principle of Lemma \ref{maxprinciple} we have 
		$V(t_*,x_2,\theta)\ge W(t_*,x_2,\theta)\ge C_0,$ for some $C_0=C_0(R)>0$ for any for any $(x_2,\theta)\in B_{2R}(x_{2,0},\theta_0)$.
		This completes the proof.
	\end{proof}
	Finally, we conclude the proof of Proposition \ref{mainlongtimelemma} by introducing the following lemma:
	\begin{lemma}\label{longtimelemma1d}The solution $\psi_+$ to \eqref{timedependentadjointeq again}, \eqref{psiplusboundary}, and \eqref{modified periodic} satisfies the long-chain asymptotic behavior $\psi_+(t,x_2,\theta)\to \psi^+_\infty(x_2,\theta)$ as $t\rightarrow \infty$ where $\psi^+_\infty$ is the solution of the stationary equation \eqref{stationary} with the boundary conditions \eqref{stationary psiplusboundary} and \eqref{periodic stationary}.
	\end{lemma}
	\begin{proof}For the proof of the long-chain asymptotics of $\psi_+(t,x_2,\theta)$, we define a function 
		$$W(t,x_2,\theta)=\psi_+(t,x_2,\theta)-\psi^+_\infty(x_2,\theta),$$ which solves \eqref{timedependentadjointeq again} with the initial condition $W(0,x_2,\theta)=\psi_+(0,x_2,\theta)-\psi^+_\infty(x_2,\theta)$,
		the zero boundary condition \eqref{psizeroboundary}, and the periodic boundary condition \eqref{modified periodic}. Our goal is to prove that $W(t,x_2,\theta) \to 0$ for any $x_2\ge 0$ and  $\theta\in[-\pi,\pi]$ as $t\to \infty$.
		
		We start with observing that $0\le \psi_+\le 1$ by the maximum principle in Lemma \ref{maxprinciple}. 
		Thus we have $|W(0,x_2,\theta)|\le C$ for some $C>0$. 
	Also, note that for any $\varepsilon>0$ there exists $K\ge 0$ and a sufficiently large $R>0$ such that 
	\begin{equation}\label{longtimecomparison}-K(\psi^+_\infty(x_2,\theta)+\varepsilon)\le W(t,x_2,\theta)\le K(\psi^+_\infty(x_2,\theta)+\varepsilon),\text{ for }t\ge R\end{equation}  by the comparison principle, since $\pm K(\psi^+_\infty(x_2,\theta)+\varepsilon)$ is a super- and a sub-solution to the system.  Define $$\lambda(t)\eqdef \sup_{x_2\in [0,\infty),\ \theta \in [-\pi,\pi]}\frac{W(t,x_2,\theta)}{\psi^+_\infty(x_2,\theta)+\varepsilon}.$$ Note that $\psi^+_\infty(x_2,\theta)+\varepsilon\ge \varepsilon>0$ by the non-negativity of $\psi^+_\infty$. Then
	$$\limsup_{t\to \infty}\lambda(t)\le K.$$ If $\limsup_{t\to \infty}\lambda(t)< K$, then we re-define $R$ and $K$ by increasing $R$ and decreasing $K$ such that $\limsup_{t\to \infty}\lambda(t)= K$ and \eqref{longtimecomparison} holds. 
	If $\limsup_{t\to \infty}\lambda(t)= K$, then there exists a sequence $\{t_n\} \to \infty$ such that $\lambda(t_n)\to K$. Then there exists a sequence of tuples $(x_{2,n},\theta_n)\in \bar{\mathbb{R}}_+\times [-\pi,\pi]$ such that 
	$$\frac{W(t_n,x_{2,n},\theta_n)}{\psi^+_\infty(x_{2,n},\theta_n)+\varepsilon} \to K,\text{ as }n\to \infty.$$ Then $\bar{W}$ is also a solution to the same system with the different initial condition $\bar{W}(0,x_2,\theta)=W(t_n,x_2,\theta)$. Define a function $\bar{W}=\bar{W}(t,x_2,\theta)$ as a limit of 
	$W(t_n+t,x_2,\theta)$ as $n\to \infty$. Also, define a tuple $(\bar{x}_2,\bar{\theta})$ as the limit of $(x_{2,n},\theta_n)$ as $n\to \infty$. Then note that
	$$\bar{W}(0,\bar{x}_2,\bar{\theta})=K(\psi^+_\infty(\bar{x}_2,\bar{\theta})+\varepsilon).$$ If $(x_{2,n},\theta_n)$ converges to a tuple of finite values $(\bar{x}_2,\bar{\theta})$ as $n\to \infty$, then by the maximum principle of Lemma \ref{maxprinciple} in the box $B\eqdef (x_2,\theta) \in [t+\epsilon,t+1]\times [\bar{x}_2+\epsilon,\bar{x}_2+1]\times [-\pi,\pi]$ for some $\epsilon\ll 1$ and obtain that $$\bar{W}(t,x_2,\theta) \le (K-\delta)(\psi^+_\infty(x_2,\theta)+\varepsilon)\text{ if } (t,x_2,\theta)\in B,$$ for some $\delta>0$. Therefore, for any $t\gg 1$, $$W(t,x_2,\theta)\le (K-\delta)(\psi^+_\infty(x_2,\theta)+\varepsilon),$$ 
		for any $x_2\ge 0$ and $\theta\in[-\pi,\pi]$. Then we can readjust $K \mapsto K-\delta$ and repeat reducing the size of $K$ until it becomes 0. 
		Therefore, the only difficulty left is in the case when the limit $\lim_{n\to \infty} x_{2,n}=\bar{x}_2=\infty$. 
		We define 
		$$ U(t,x_2,\theta) \eqdef  K\left(\psi^+_\infty(x_2,\theta)+\varepsilon\right)-W(t,x_2,\theta).$$ Note that $\psi^+_\infty(\infty,\theta)=\frac{1}{2},$ by Lemma \ref{stationaryx2infty}. 
		Then we know that $U$ solves the same equation
		$$\partial_t U -\sin\theta \partial_{x_2}U = \partial^2_\theta U,$$  with the condition
		$U(t,0,\theta)\ge \varepsilon >0$ for $t\ge 0$ and $\theta \in [-\pi,-\pi/2)\cup (-\pi/2,0]$,	 and $U(0,x_2,\theta)\ge 0$ for any $x_2\ge 0$ and $\theta \in[-\pi,\pi]$ by \eqref{longtimecomparison}. 
		
		We first note that Lemma \ref{subsolution} implies that for a sufficiently large $\bar{t}>0$, there exists some $\bar{x}_2>0$ such that $U(\bar{t},\bar{x}_2,\theta) \ge \delta>0$ for some $\delta>0$ for any $\theta\in[-\pi,\pi]$. We obtain this by rescaling and repeating the argument in Lemma \ref{subsolution} until we extend the domain to $[-\pi,\pi]$, since $t_*$ and $C_0$ depends only on $R$. Now our goal is to show that $\limsup_{t\to \infty}\sup_{x_2,\theta} U(t,x_2,\theta)$ is strictly positive. Then we can reduce the size of $K$ by the positive gap $\limsup_{t\to \infty}\sup_{x_2,\theta} U(t,x_2,\theta)$, and we iterate the arguments above with a new $K$ and obtain that $W$ is indeed zero for a sufficiently large $t\gg 1$ and $\varepsilon\to 0$.
		Suppose that $U(\bar{t},\bar{x}_2,\theta)\ge 1$ for any $\theta\in[-\pi,\pi]$ and that $U(t,0,\theta)\ge 1$ for $t\ge 0$ and $\theta \in [-\pi,0]$ by rescaling. We claim that there exists $k>0$ such that for any $R>0$ large we have 
		$U(t,R,\theta)\ge k>0$, if $t$ is sufficiently large. We prove this by constructing a subsolution. Define 
		$$V(t,x_2,\theta)\eqdef k -e^{-\sigma t}e^{-\lambda x_2}Q(\theta),$$ for some $\sigma>0$ sufficiently small, $ \lambda>0$, and a $2\pi$-periodic(-in-$\theta$) $C^2$ function $Q$ which solves
		$$ 
		-\sigma Q+\lambda_2 \sin\theta Q-Q''\ge 0.$$ For example, $Q$ is given by the Mathieu function. We assume without loss of generality that $|V|\le 1$ by rescaling. 
		Then we note that $V$ is a subsolution that satisfies $\partial_t V -\sin\theta \partial_{x_2}V  -\partial^2_\theta V \le 0$ with
		$|V|\le 1$. Then since $V-U$ is a subsolution with non-positive values when $t=\bar{t}$ and when $x_2=0$ and $\theta\in[-\pi,0]$, the maximum principle of Lemma \ref{maxprinciple} implies that $V(t,x_2,\theta )\le U(t,x_2,\theta)$ for $t\ge \bar{t}$. By passing $\lambda \to 0$ and then $t\to \infty$ we obtain that $U(t,x_2,\theta) \ge k$ for a sufficiently large $x_2$.  This completes the proof.
	\end{proof}

	\begin{remark}
		In addition, we remark that this asymptotics of the solutions to the adjoint problem indicates that the function is asymptotically going to be supported only on the set of the points $(x_2,\theta)=(0,0)$ and $=(0,-\pi)$.
		Since the dynamics for $\rho_1$ in $t$ and $x_1$ variables are under free-transport equations, we will get the dynamics supported on particular lines and the evolution of the lines are the translations on $x_1=\pm t$. \end{remark}
	
		This completes the proof of Proposition \ref{mainlongtimelemma} and Corollary \ref{mainlongtimecorollary}, which then imply Theorem \ref{maintheorem}$_{(5)}$. In the next subsection, we further discuss the dynamics of the mass that has been accumulated at the singular boundary $(x_2,\theta)=(0,0)$ and $(0,-\pi).$

	\subsection{Dynamics for the mass at \texorpdfstring{$(\lowercase{x}_2,\theta)=(0,0)$}{} or \texorpdfstring{$(0,-\pi)$}{}}
	We also discuss the asymptotics of the mass that has been accumulated on either $(x_2,\theta)=(0,0)$ or $(0,-\pi)$ in the rest of this section. Denote these densities as $\rho^+=\rho^+(t,x_1)$ and $\rho^-=\rho^-(t,x_1)$, respectively.  
	Then we discuss the dynamics for the mass reached $(x_2,\theta)=(0,0)$ or $(0,-\pi)$.

	Consider the backward-in-$t$ adjoint problem \eqref{backwardintimeproblem}-\eqref{backward boundary}. 
	We first consider a special test function $\phi_\epsilon$ in the weak formulation Definition \ref{weaksoldef} that is independent of $x_1$ variable. Then choose the initial profile $\phi_\epsilon(T)$ such that it is $=1$ on $x_2<\epsilon$ and $\theta\in [-\pi,0]$ and is supported only on $x_2<2\epsilon.$ Then by the previous argument, we know that $\phi_{in}(x_2,\theta)\to 1$ as $T\to 1$.
	Then by the dominated convergence theorem, we can pass $\epsilon\rightarrow 0$ and obtain by the weak formulation and the duality identity that
	$$\int_\mathbb{R} \left( \rho^+(T,x_1)+\rho^-(T,x_1)\right)dx_1=\int f_{in}\phi_{in} dxd\theta.$$ Now by letting $T\rightarrow \infty$, we observe that
	$$\lim_{T\rightarrow \infty}\int_\mathbb{R} \left( \rho^+(T,x_1)+\rho^-(T,x_1)\right)dx_1=\int f_{in}\ dxd\theta.$$ 
	\begin{remark}
		We also remark that the total mass density $\rho^+_\infty$ that travels in the direction $\theta=0$
		satisfies $$ \rho^+_\infty=\int \psi^+_\infty(x_2,\theta)f_{in}(x_1,x_2,\theta)dx_2d\theta,$$where $\psi^+_\infty$ is the stationary solution to \eqref{stationary} with \eqref{stationary psiplusboundary} and \eqref{periodic stationary}. We also have a similar argument for $\rho^-_\infty$ which is the total mass density that travels in the direction $\theta=-\pi$ as
		$$ \rho^-_\infty=\iint \psi^-_\infty(x_2,\theta)f_{in}(x_1,x_2,\theta)dx_2d\theta,$$ where $\psi^-_\infty$ is the stationary solution to \eqref{stationary} with \eqref{stationary psiminusboundary} and \eqref{periodic stationary}.
		Thus, we observe that
		$$\varphi(t,x_2,\theta)\rightarrow \alpha \psi^-_\infty (x_2,\theta)+\beta \psi^+_\infty (x_2,\theta),$$ as $t\rightarrow \infty$. Then we have the distribution $f(t,x_1,x_2,\theta)$ converges to 
		$$f\rightharpoonup \rho^+_\infty\delta(x_2)\delta(\theta)+\rho^-_\infty\delta(x_2)\delta(\theta+\pi),$$ as $t\to \infty$ where $\delta(\cdot)$ is the Dirac delta function.
	\end{remark}

	We have observed the dynamics above that in the 1-dimensional case, all the mass density integrated with respect to $x_1$ converges to two particular regions $\theta=0$ or $\theta=-\pi$ on the line $x_2=0$. 
	Then we can easily observe that after the polymer arrives at $x_2=0$ with either $\theta=0$ or $\theta=-\pi$, it is trapped there at $(x_2,\theta)=(0,0)$ or $=(0,-\pi)$ while the dynamics in $t,x_1$ will be just the free transport. In other words, $\rho^\pm(t,x_1)$ would satisfy the following free transport equation,
	$$\partial_t \rho^\pm \pm \partial_{x_1}\rho^\pm = 0,$$ as $\cos\theta = \pm 1$ if $\theta =-\frac{\pi}{2}\pm \frac{\pi}{2}.$ Therefore, we obtain that
	$$\rho^+(t,x)= \rho^+_\infty(x_1-t) \text{ and }\rho^-(t,x)= \rho^-_\infty(x_1+t).$$
	Then we also obtain from the duality argument that
	$$\int_{\mathbb{R}}dx_1(\rho^+_\infty(x_1-t)+\rho^-_\infty(x_1+t))=\iint_{\mathbb{R}^2_+\times [-\pi,\pi]} dxd\theta f_{in}(x,\theta).$$

	\appendix
	
	\section{Derivation of an initial-boundary value problem}

	\label{sec:formalderivation}
	
	Similarly to the whole plane situation in Section \ref{wholeplanederiv}, we define for $j\ge 0$ that
	$$\xi_j\eqdef x_0+\sum_{i=1}^{j}\epsilon n_i$$and introduce the Hamiltonian 
	$$\mathcal{H}_N=\frac{1}{2\epsilon}\sum_{j=1}^{N-1}(n_{j+1}-n_j)^2=-\frac{1}{\epsilon}\sum_{j=1}^{N-1} (n_j\cdot n_{j+1}-1).$$
	Then we can have that the Gibbs measure $\mu_N\in \mathcal{M}_+((\mathbb{S}^1)^{N}\times (\mathbb{R}^2)^N)$ is defined as
	$$\mu_N(x_1,...,x_N;n_1,...,n_{N}) =\frac{1}{Z_N}\prod_{i=1}^N\delta(x_i-\xi_i)\chi_{\{x_i\in\mathbb{R}^2_+\}}\exp \left(\frac{1}{\epsilon}\sum_{j=1}^{N}(n_{j-1}\cdot n_j-1)\right),$$ with $n_0=n_1$ where the characteristic function $\chi$ restricts all the position of monomers $x_j$ in the upper half-plane and the normalization factor $Z_N$ is defined as 
	$$Z_N\eqdef \int_{(\mathbb{R}^2)^{N}}\int_{(\mathbb{S}^1)^{N}}\prod_{j=1}^{N}dx_jdn_j\  \mu_N(x_1,...,x_N;n_1,...,n_{N}).$$ 
	We note that for $e_2\eqdef (0,1)$,
	$x_{k+1}\cdot e_2\ge 0$ is equivalent to $x_k\cdot e_2 +\epsilon n_k\cdot e_2 \ge 0$. This is then equivalent to $n_k\cdot e_2 \ge -\frac{x_k\cdot e_2}{\epsilon}.$
	Here, we emphasize that the polymer dynamic is Markovian and the next state $(x_{k+1},n_{k+1})$ would depend only on the present state $(x_{k},n_{k})$. 
	Thus, we can write the Markov process as
	\begin{multline*}
	\mathbb{P}(x_{k+1},n_{k+1}|x_{k},n_{k})= \\C_k \delta (x_{k+1}-x_k-\epsilon n_{k+1})\chi_{\{x_{k+1}\in \mathbb{R}^2_+\}}\exp\left(-\frac{1}{\epsilon}(n_{k+1}\cdot n_k -1)\right),
	\end{multline*}where $C_k=\frac{Z_k}{Z_{k+1}}.$
	Therefore, the probability distribution satisfies
	\begin{multline*}
		f_{k+1}(x_{k+1},n_{k+1})
		=\prod_{j=1}^{k}\int dx_jdn_j\ \mathbb{P}(x_{j+1},n_{j+1}|x_{j},n_{j})
		\\=\int_{\mathbb{S}^1}dn_k\int_{\mathbb{R}^2}dx_k\ \mathbb{P}(x_{{k+1}},n_{{k+1}}|x_{k},n_{k})f_{k}(x_{k},n_{k})\\
		=\int_{\mathbb{S}^1}dn_k\int_{\mathbb{R}^2}dx_k \ C_k \delta (x_{k+1}-x_k-\epsilon n_{k+1})\\\times \exp\left(-\frac{1}{\epsilon}(n_{k+1}\cdot n_k -1)\right)
		f_{k}(x_{k},n_{k})\chi_{\{n_k\cdot e_2 \ge -\frac{x_k\cdot e_2}{\epsilon}\}}\\
		=\int_{\mathbb{S}^1}dn_k\int_{\mathbb{R}\times [-\epsilon\sin\varphi,\infty)}dx_k\ C_k \delta (x_{k+1}-x_k-\epsilon n_{k+1})\exp\left(-\frac{1}{\epsilon}(n_{k+1}\cdot n_k -1)\right)\\\times
		 f_{k}(x_{k},n_{k}),
	\end{multline*}where $\varphi$ is the angle between the $e_1\eqdef (1,0)$ and $n_k$.
	Now define $F(t,x,n)=f_j(x,n)$ with $j=\frac{t}{\epsilon}.$ Then we have
	\begin{multline*}F(k\epsilon +\epsilon,x_{k+1},n_{k+1})\\
		=C_k\int_{\mathbb{S}^1}dn_k\int_{\mathbb{R}\times [-\epsilon\sin\varphi,\infty)}dx_k\ \delta (x_{k+1}-x_k-\epsilon n_{k+1})  F(k\epsilon, x_k,n_k)\\\times\exp\left(\frac{1}{\epsilon}(n_k\cdot n_{k+1}-1)\right).\end{multline*}
	Now we subdivide the situation into two cases: $x_k \in \mathbb{R}\times (\epsilon\sin\varphi,\infty)$ and $x_k\in \mathbb{R}\times [-\epsilon\sin\varphi,\epsilon\sin\varphi].$ The former case refers to the polymer away from the boundary $\partial \mathbb{R}^2_+$ and the latter one refers to the polymer chain near the boundary.
	
	\noindent\textbf{Case 1: away from the boundary.} If  $x_k \in \mathbb{R}\times (\epsilon\sin\varphi,\infty)$, or equivalently the distance $dist(x_k,\partial \mathbb{R}^2_+)>\epsilon,$ we use the formal ansatz that $F$ is smooth and take the Taylor expansion as
	$$
	F(k\epsilon +\epsilon,x_{k+1},n_{k+1})=	F(k\epsilon ,x_{k+1},n_{k+1})+\epsilon \frac{\partial F}{\partial t}(k\epsilon, x_{k+1},n_{k+1}),
	$$and 
	\begin{multline*}
		C_k\int_{\mathbb{S}^1}dn_k\int_{\mathbb{R}\times (\epsilon\sin\varphi,\infty)}dx_k\ \delta (x_{k+1}-x_k-\epsilon n_{k+1}) F(k\epsilon, x_k,n_k)\\\times
		\exp\left(\frac{1}{\epsilon}(n_k \cdot n_{k+1}-1)\right)\\
		\hspace{-50mm}=C_k\int_{\mathbb{S}^1}dn_k \bigg(F(k\epsilon, x_{k+1},n_{k+1})-\epsilon n_{k+1}\cdot \frac{\partial F}{\partial x}\\
		+(n_k-n_{k+1})\cdot \nabla_k F(k\epsilon,x_{k+1},n_{k+1})\\+\frac{1}{2}(n_k-n_{k+1})\nabla_k^2 F(k\epsilon,x_{k+1},n_{k+1})(n_k-n_{k+1})\bigg)\\\times\exp\left(\frac{1}{\epsilon}(n_k \cdot n_{k+1}-1)\right).
	\end{multline*}
	Note that \begin{equation}\label{normalizedgibbs.gen}C_k\int_{\mathbb{S}^1}dn_k \exp (\epsilon^{-1} (n_k\cdot n_{k+1}))=1,\end{equation} and $$\int_{\mathbb{S}^1}dn_k \ (n_k-n_{k+1})\exp (\epsilon^{-1} (n_k\cdot n_{k+1}))=0.$$
	Then by defining $t=k\epsilon$, we have
	\begin{multline*}
		\epsilon \frac{\partial F}{\partial t}(t, x_{k+1},n_{k+1})+\epsilon n_{k+1}\cdot \frac{\partial F}{\partial x}(t, x_{k+1},n_{k+1})\\
		=\Delta_n F(t,x_{k+1},n_{k+1})C_k\int_{\mathbb{S}^1}dn_k \bigg(\frac{1}{2}(n_k-n_{k+1})\otimes (n_k-n_{k+1})\bigg)\\\times
		 \exp\left(\frac{1}{\epsilon}(n_k \cdot n_{k+1}-1)\right).
	\end{multline*}
	Note that 
\begin{multline}\label{CnOe2}C_k\int_{\mathbb{S}^1}dn_k \bigg(\frac{1}{2}(n_k-n_{k+1})\otimes (n_k-n_{k+1})\bigg)\exp\left(\frac{1}{\epsilon}(n_k \cdot n_{k+1}-1)\right)\\\approx C_k\int_{\mathbb{S}^1}(\xi\otimes \xi) \exp(-|\xi|^2/(2\epsilon))d\xi
\approx O(\epsilon),
\end{multline}by \eqref{normalizedgibbs.gen} where $\xi=n_k-n_{k+1}$.
	
	Thus, in the limit $\epsilon\rightarrow 0$, we obtain the Fokker-Planck equation in the half plane $x\in \mathbb{R}^2_+$ and $n\in \mathbb{S}^1$ as
	$$\frac{\partial F}{\partial t}(t, x ,n )+ n\cdot \frac{\partial F}{\partial x}(t, x ,n )= D\Delta_n F(t,x ,n ),$$ for the diffusion coefficient $D$ which is defined as
	$$D= \lim_{\epsilon \to 0} \frac{C_{\frac{t}{\epsilon}}}{\epsilon} \int_{\mathbb{S}^1}(\xi\otimes \xi) \exp(-|\xi|^2/(2\epsilon))d\xi,$$by \eqref{CnOe2}.

	\noindent\textbf{Case 2: near the boundary.} 
	Suppose that the $k^{th}$ state $x_k$ of the process is in the region nearby the boundary as $x_k\in \mathbb{R}\times [-\epsilon\sin\varphi,\epsilon\sin\varphi].$ Equivalently, we have that the distance $dist(x_k,\partial \mathbb{R}^2_+)\le\epsilon.$
	As a modeling assumption, we suppose that the monomer which collides the boundary will change its incoming angle against the boundary so that it minimizes the energy $$\frac{1}{2\epsilon}(n_{k+1}-n_k)^2=-\frac{1}{\epsilon} (n_k\cdot n_{k+1}-1),$$ among all possible angles that the monomer can have. The possible velocities are the velocities that makes $x_{k+1}\in \overline{\mathbb{R}^2_+}.$
	Equivalently, this means that it maximizes $n_k\cdot n_{k+1}=\cos \phi$, where $\phi$ is the angle between $n_k$ and $n_{k+1}$. 
	\begin{figure}[ht]
		\begin{center}
			\includegraphics[scale=0.385]{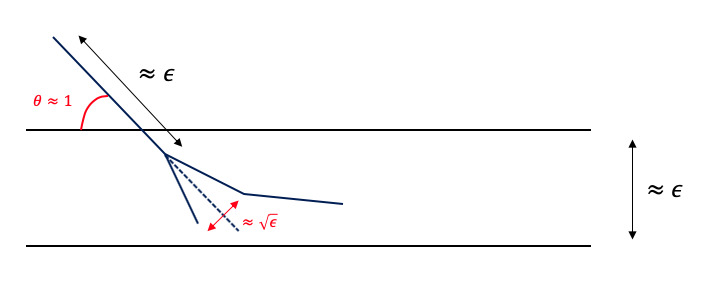}
			\includegraphics[scale=0.445]{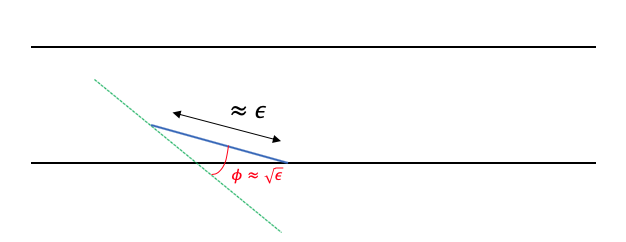}
		\end{center}
		\caption{The energy-minimizing transition}
	\end{figure}
	Here, we note that the change of the angle $\phi$ is in the order of $\sqrt{\epsilon}$ as we have
	$$\frac{1}{\epsilon} (n_k\cdot n_{k+1}-1)= \frac{1}{\epsilon} (\cos\phi-1)\approx -\frac{\theta^2}{\epsilon}\approx 1.$$ 
	After the state, the next velocity $n_{k+2}$ is determined so that $x_{k+2}\in \overline{\mathbb{R}^2_+}$ and it minimizes the energy $\frac{1}{2\epsilon}(n_{k+2}-n_{k+1})^2.$ Once the velocity of a monomer becomes parallel to the boundary, the polymer can now start deviating from the boundary as it can have any velocity following the Gibbs probability distribution $\exp(-\frac{1}{\epsilon} (n_{k+3}\cdot n_{k+2}-1 ))$. 
	\begin{figure}[ht]
		\begin{center}
			\includegraphics[scale=0.5]{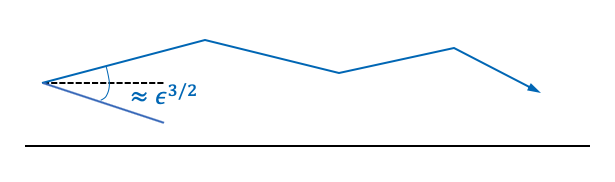}
		\end{center}
		\caption{The diffusion in the direction of each monomer near the boundary}
	\end{figure}
	Since the change of the angle is in the order of $\sim \sqrt{\epsilon}$, we note that the deviation distance from the boundary is of order $$|n_{j+1}-n_j||\phi|\sim \epsilon \cdot \sqrt{\epsilon}\sim \epsilon^{\frac{3}{2}}.$$
	Recall that the scale of the total number of monomers is $N\sim \frac{1}{\epsilon}.$ 
	Therefore, we conclude that the maximum distance of the deviation of the polymer from the boundary is 
	$$N\cdot \epsilon^{\frac{3}{2}}\sim \sqrt{\epsilon}.$$
	Therefore, we can conclude that the polymer near the boundary is trapped on the boundary as $\epsilon \rightarrow 0.$ This completes a formal or heuristic derivation of the \textit{trapping} boundary problem from our modeling assumptions.

	\section*{Acknowledgement} The authors gratefully acknowledge the support of the Hausdorff Research
	Institute for Mathematics
	(Bonn), through the Junior Trimester Program on Kinetic Theory, and the CRC
	1060 ``The Mathematics of Emergent Effects" of the University of Bonn funded
	through the Deutsche Forschungsgemeinschaft (DFG, German Research Foundation). J. W. Jang is supported by Basic Science Research Institute Fund of Korea, whose NRF grant number is 2021R1A6A1A10042944. J. J. L. Vel\'azquez is funded by DFG under Germany's Excellence Strategy-EXC-2047/1-390685813.
	\bibliographystyle{amsplaindoi}
	\bibliography{bibliography}
	
\end{document}